\documentclass[12pt,reqno]{NumPDEsArticle}

\usepackage[utf8]{inputenc}
\usepackage[T1]{fontenc}
\usepackage[english]{babel}
\usepackage{csquotes}
\usepackage{enumitem}
\usepackage{amsmath,amssymb}
\usepackage{booktabs}
\usepackage{biblatex}
\usepackage{stmaryrd}
\usepackage{diagbox}
\usepackage[caption=false]{subfig}
\definecolor{pyYellow}{HTML}{bcbd22}

\numberwithin{theorem}{section}
\numberwithin{equation}{section}

\usepackage{NumPDEsMacros}

\newcommand{\exact}{\star}
\newcommand{\coarse}{H}
\newcommand{\fine}{h}
\newcommand{\kk}{{\underline{k}}}
\newcommand{\elll}{{\underline\ell}}
\newcommand{\jj}{\underline j}

\let \Re \relax
\DeclareMathOperator{\Re}{Re}
\def\bfA{\boldsymbol{A}}
\def\bfb{\boldsymbol{b}}
\def\bff{\boldsymbol{f}}

\newcommand{\Cchild}{\const{C}{child}}

\newcommand{\Cmesh}{\const{C}{mesh}}

\newcommand{\lamalg}{\const{\lambda}{alg}}
\newcommand{\lamsym}{\const{\lambda}{sym}}
\newcommand{\thetamark}{\const{\theta}{mark}}
\newcommand{\qalg}{\const{q}{alg}}
\newcommand{\qsym}{\const{q}{sym}}
\newcommand{\qsymm}{\const[\overline]{q}{sym}}
\newcommand{\Cctr}{\const[\overline]{C}{sym}}

\newcommand{\rrevision}[1]{\TemplateRevision{red}{#1}}
\newcommand{\highlighted}[1]{\TemplateRevision{red}{#1}}

\def\Re{}
\title[Adaptive FEM with quasi-optimal computational cost]{Adaptive FEM with quasi-optimal overall cost\\ for nonsymmetric linear elliptic PDEs}

\author{Maximilian Brunner}
\author{Pascal Heid}
\author{Michael Innerberger}
\author{Ani Mira{\c c}i}
\author{Dirk Praetorius}
\author{Julian Streitberger}

\address{Department of Mathematics, Technical University of Munich, Boltzmannstr. 3,
	85748 Garching bei München, Germany \& Munich Center for Machine Learning (MCML)}
\email{pascal.heid@ma.tum.de}

\address{TU Wien, Institute of Analysis and Scientific Computing, Wiedner Hauptstr. 8--10/E101/4, 1040 Vienna, Austria}
\email{maximilian.brunner@asc.tuwien.ac.at}
\email{michael.innerberger@asc.tuwien.ac.at}
\email{ani.miraci@asc.tuwien.ac.at}
\email{dirk.praetorius@asc.tuwien.ac.at}
\email{julian.streitberger@asc.tuwien.ac.at\quad \rm (corresponding author)}



\thanks{The authors thankfully acknowledge support by the Austrian Science Fund (FWF) through the SFB \emph{Taming complexity in partial differential systems} (grant SFB F65) and the standalone project \emph{Computational nonlinear PDEs} (grant P33216). Additionally, the Vienna School of Mathematics supports Maximilian Brunner and Julian Streitberger.\newline
\indent{\bf Corrigendum.}\quad Unfortunately, there is a flaw in 
the numerical analysis of the published version [IMA J.\ Numer.\ Anal., 
\href{https://doi.org/10.1093/imanum/drad039}{DOI:10.1093/imanum/drad039}], 
which is corrected here. Neither the algorithm, nor the results are affected, 
but constants have to be adjusted. The preprint 
\href{https://arxiv.org/abs/2212.00353v2}{[arXiv:2212.00353v2]} is the original 
manuscript published in IMA J.\ Numer.\ Anal., 
\href{https://arxiv.org/abs/2212.00353v5}{[arXiv:2212.00353v5]} is the 
corrected manuscript (with color highlighting of essential changes), and 
\href{https://arxiv.org/abs/2212.00353v6}{[arXiv:2212.00353v6]} is the 
corrected manuscript (without color highlighting).}

\begin{document}
\maketitle
\begin{abstract}
	We consider a general nonsymmetric second-order linear elliptic PDE in the framework of the Lax--Milgram lemma. We formulate and analyze an adaptive finite element algorithm with arbitrary polynomial degree that steers the adaptive mesh refinement and the inexact iterative solution of the arising linear systems. More precisely, the iterative solver employs, as an outer loop, the so-called Zarantonello iteration to symmetrize the system and, as an inner loop, a uniformly contractive algebraic solver, e.g., an optimally preconditioned conjugate gradient method or an optimal geometric multigrid algorithm. We prove that the proposed inexact adaptive iteratively symmetrized finite element method (AISFEM) leads to full linear convergence and, for sufficiently small adaptivity parameters, to optimal convergence rates with respect to the overall computational cost, i.e., the total computational time. Numerical experiments underline the theory.
\end{abstract}

\section{Introduction} \label{section:introduction}
The mathematical understanding of optimal adaptivity for finite element methods (AFEMs) has reached a high level of maturity; see, e.g.,~\cite{bdd2004,stevenson2007,ckns2008,ks2011,cn2012,ffp2014,axioms}
for some contributions to linear PDEs. While the focus is usually on optimal convergence rates with respect to the degrees of freedom~\cite{bdd2004,ckns2008,ks2011,cn2012,ffp2014,axioms}, the cumulative nature of adaptivity should rather ask for optimal convergence rates with respect to the overall computational cost, i.e., the overall elapsed computational time. This, usually called \emph{optimal complexity}, has been thoroughly analyzed for adaptive wavelet methods~\cite{MR1803124,MR2035007} and it has also been addressed in the seminal work~\cite{stevenson2007} on AFEM for the Poisson model problem.
Recent works~\cite{ghps2021,hpw2021,hpsv2021} considered optimal complexity for energy minimization problems and, in particular, for symmetric linear elliptic PDEs. In contrast to this, optimal complexity for nonsymmetric linear elliptic PDEs remained an open question due to the lack of a contractive algebraic solver that is compatible with the variational structure of the PDE. Closing this gap is the topic of the present work. While the canonical candidate for solving the nonsymmetric discrete systems would be GMRES, we take a different path that is motivated by up-to-date proofs of the Lax--Milgram lemma and closely related to the Richardson iteration used in the context of optimal adaptive wavelet methods. Some comments on the challenges presented by GMRES and related future work are given below. 

As a model problem, we consider the nonsymmetric second-order linear elliptic PDE
\begin{align}\label{eq:intro:model_pb}
	-\div(\bfA \nabla u^\exact) + \bfb \cdot \nabla u^\exact + c u^\exact = f - \div \bff
	\quad \text{in } \Omega
	\quad \text{subject to} \quad
	u^\exact = 0
	\quad \text{on } \partial \Omega
\end{align}
on a polyhedral Lipschitz domain $\Omega \subset \R^d$ with $d \ge 1$, where $\bfA \in [L^\infty(\Omega)]^{d \times d}_{\rm sym}$ is a symmetric diffusion matrix, $\bfb \in [L^\infty(\Omega)]^d$ is a convection coefficient, $c \in L^\infty(\Omega)$ is a reaction coefficient, and $f \in L^2(\Omega)$ and $\bff \in [L^2(\Omega)]^d$ are the given data.

\noindent With $b(u, v)
\coloneqq
\dual{\bfA \nabla u}{\nabla v }_\Omega + \dual{\bfb \cdot  \nabla u +  c u}{ v }_\Omega
$ and $F(v)
\coloneqq \dual{f}{v}_\Omega + \dual{\bff}{\nabla v}_\Omega$, where $\dual{\cdot}{\cdot}_\Omega$ denotes the usual $L^2(\Omega)$-scalar product,
the weak formulation of~\eqref{eq:intro:model_pb} reads:
%Find $u^\exact \in \XX \coloneqq H_0^1 (\Omega)$ such that%, for all $v \in H_0^1 (\Omega)$,
\begin{align}\label{eq:intro:weakform}
	\text{Find } u^\exact \in \XX \coloneqq H_0^1 (\Omega)
	\quad \text{such that } \quad
	b(u^\exact, v)
	=
	F(v)
	\quad \text{for all } v \in \XX.% = H_0^1 (\Omega).
\end{align}
To ensure the existence and uniqueness of $u^\exact \in H_0^1 (\Omega)$, we assume that the bilinear form $b(\cdot, \cdot)$ is continuous and elliptic on $ H_0^1 (\Omega)$ so that the Lax–Milgram lemma applies.

To discretize~\eqref{eq:intro:weakform}, we employ a conforming finite element method based on a conforming simplicial triangulation $\TT_\ell$ of $\Omega$
% where $\ell \ge 0$ shall later denote the level index in the AFEM mesh hierarchy,
and a fixed polynomial degree $m \in \N$. With
\begin{align}\label{eq:def:X_ell}
	\XX_\ell \coloneqq \set{v_\ell \in H_0^1 (\Omega) \given v_\ell|_T \ \text{is a polynomial of degree} \le m, \text{ for all } T \in \TT_\ell},
\end{align}
the finite element formulation reads:
%The exact Galerkin solution}
%$u_\ell^\exact \in \XX_\ell$
%approximates $u^\exact$ by solving
\begin{align}\label{eq:intro:discrete}
\text{Find } u_\ell^\exact \in \XX_\ell
\quad \text{such that } \quad
b(u_\ell^\exact , v_\ell)
=
F(v_\ell)
\quad \text{for all } v_\ell \in \XX_\ell.
\end{align}
Existence and uniqueness of $u_\ell^\exact$ follow again from the Lax--Milgram lemma.
Note that~\eqref{eq:intro:discrete} leads to a \emph{nonsymmetric}, yet \emph{positive definite} linear system of equations. 
To derive an optimal nonsymmetric algebraic solver, we follow the constructive proof of the Lax--Milgram lemma and reduce
%To reduce 
the discrete formulations~\eqref{eq:intro:discrete} to symmetric problems
by employing
%, we employ 
the so-called Zarantonello symmetrization (sometimes referred to as Banach--Picard fixed-point iteration).
%, which is usually used to prove the Lax--Milgram lemma.
To this end, we define the bilinear form associated with the principal part of the PDE by
\begin{align}
a(u, v) \coloneqq \dual{\bfA \nabla u}{\nabla v}_\Omega
\quad \text{for all } u, v \in \XX. %H^1_0(\Omega).
\end{align}%
Note that $a(\cdot, \cdot)$ is continuous and elliptic on $\XX$ and consult Section~\ref{section:abstract} for details. For a given damping parameter $\delta > 0$, define the Zarantonello mapping $\Phi_\ell(\delta;\cdot) \colon \XX_\ell \to \XX_\ell$ by
\begin{align}\label{eq:intro:zarantonello}
a(\Phi_\ell(\delta; u_\ell), v_\ell)
=
a(u_\ell, v_\ell)
+ \delta \bigl[ F(v_\ell) - b(u_\ell, v_\ell) \bigr]
\quad \text{for all } v_\ell \in \XX_\ell;
\end{align}
see \cite{Zarantonello1960} or \cite[Section~25.4]{Zeidler1990}.
The Riesz--Fischer theorem (and also the Lax--Milgram lemma) proves existence and uniqueness of $\Phi_\ell(\delta; u_\ell) \in \XX_\ell$, i.e., the Zarantonello operator is well-defined.
In particular, $u_\ell^\exact = \Phi(\delta; u_\ell^\exact)$ is the only fixpoint of $\Phi(\delta; \cdot)$ for any $\delta > 0$. Moreover, choosing $\delta$ suitably small will lead to a contractive method to approximate $u_\ell^\exact$ in the spirit of the Banach fixpoint theorem with respect to the $a(\cdot,\cdot)$-induced energy norm $\enorm{v} \coloneqq a(v,v)^{1/2}$.
At this point, it thus remains to treat a symmetric, positive definite (SPD) linear system of equations corresponding to~\eqref{eq:intro:zarantonello}, that can be solved iteratively in practice for instance by the use of either a conjugate gradient (CG) method with an optimal preconditioner, see e.g.,~\cite{cnx2012}, or an optimal geometric multigrid (MG)
solver, see e.g.,~\cite{wz2017,imps2022}.

The proposed adaptive strategy of this work, hereafter referred to as AISFEM, begins with the initial guess $u_0^{0,0} \coloneqq u_0^{0,\jj} \coloneqq u_0^{0, \star} \coloneqq 0 \in \XX_0$ associated to a coarse mesh $\TT_0$. Finite element approximations $u_\ell^{k,j} \in \XX_\ell$ are successively computed, where $\ell \in \N_0$ is the mesh-refinement index of the $\ell$-th adaptively refined mesh. More precisely, $u_\ell^{k,j}$ is obtained after $j$ algebraic solver steps in the $k$-th step of the Zarantonello symmetrization approximating the unique
$u_\ell^{k,\star} \coloneqq \Phi_\ell(\delta; u_\ell^{k-1, \jj}) \in \XX_\ell$, where $u_\ell^{k-1,\jj} \in \XX_\ell$ denotes the final approximation of $u_\ell^{k-1,\star}$ when the algebraic solver is \emph{adaptively} terminated. In particular, our analysis provides stopping criteria for the algebraic solver as well as the (perturbed) Zarantonello symmetrization. We give a schematic view of our approach in Figure~\ref{intro:scheme}; see Algorithm~\ref{algorithm} in Section~\ref{section:algorithm} below for the formal statement.

\begin{figure}[!ht]
\centering
\includegraphics[width=\textwidth]{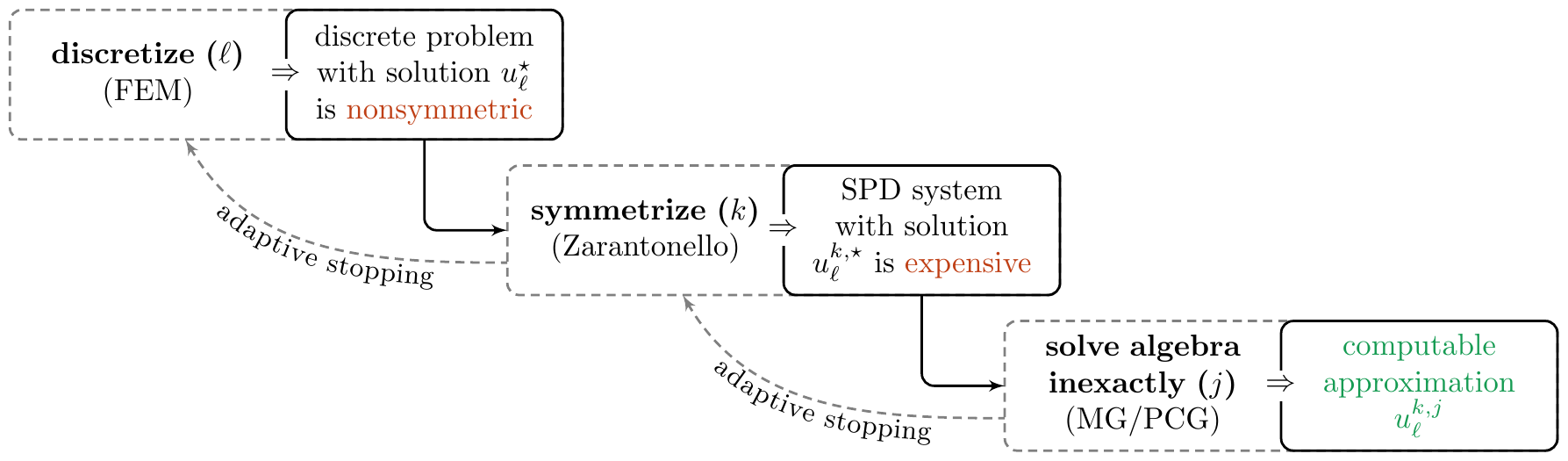}
\caption{\label{intro:scheme}Schematic view of the AISFEM algorithm components.}
\end{figure}

Overall, the adaptive strategy thus leads to a triple index set
\begin{align}\label{eq:intro:index-set}
\QQ \coloneqq \set{(\ell,k,j) \in \N_0^3 \given u_\ell^{k,j} \text{ is used by the AISFEM Algorithm~\ref{algorithm}}},
\end{align}
equipped with the natural lexicographic order $\abs{\cdot, \cdot, \cdot}$. This enables us to present the main contributions of this work: First,
in the spirit of~\cite{ghps2021,hpsv2021}, we prove that the quasi-error
\begin{align}
\Delta_\ell^{k,j}
\coloneqq
\enorm{u^\exact - u_\ell^{k,j}} + \enorm{u_\ell^{k,\exact} - u_\ell^{k,j}} + \eta_\ell(u_\ell^{k,j})
\quad \text{for all } (\ell,k,j) \in \QQ,
\end{align}
which is the sum of the overall error plus the algebraic solver error plus the residual error estimator, is linearly convergent with respect to the order of $\QQ$, i.e., $|\ell',k',j'| < |\ell,k,j|$ means that $u_{\ell'}^{k',j'}$ is computed earlier than $u_\ell^{k,j}$ within the (sequential) adaptive loop and $|\ell,k,j| - |\ell',k',j'| \in \N_0$ is the overall number of discretization, symmetrization, and algebraic solver steps in between.
In explicit terms, Theorem~\ref{theorem:aisfem:linear-convergence}  proves the existence of constants $\Clin > 0$ and $0 < \qlin < 1$ as well as an index $\ell_0 \in \N_0$ such that,
for all $(\ell,k,j), (\ell',k',j') \in \QQ \text{ with } |\ell,k,j| > |\ell',k',j'|$ and $\ell' \ge \ell_0$, there holds that
\begin{align}\label{eq:intro:linear-convergence}
\Delta_\ell^{k,j}
\le
\Clin \qlin^{|\ell,k,j| - |\ell',k',j'|} \, \Delta_{\ell'}^{k',j'}.
% \, \text{for all}\ .
\end{align}
The threshold level $\ell_0 \in \N_0$ arises from the lack of Galerkin orthogonality with respect to the $a(\cdot,\cdot)$-induced energy norm leading to a more involved analysis.
Second, as shown in Corollary~\ref{corollary:aisfem:linear-convergence}, this implies that, for any $s > 0$, there holds the equivalence
\begin{align}\label{eq:intro:complexity}
\sup_{\substack{(\ell,k,j) \in \QQ }} (\#\TT_\ell)^s \, \Delta_\ell^{k,j} < \infty
\quad \Longleftrightarrow \quad
\sup_{(\ell,k,j) \in \QQ} \Bigl( \sum_{\substack{(\ell',k',j') \in \QQ \\ |\ell',k',j'| \le |\ell,k,j| }} \#\TT_{\ell'} \Bigr)^s \, \Delta_\ell^{k,j} < \infty.
\end{align}
The interpretation of~\eqref{eq:intro:complexity} is that the AISFEM algorithm leads to algebraic convergence rate $s > 0$ with respect to the degrees of freedom (finite left-hand side) if and only if it leads to algebraic convergence rate $s$ with respect to the overall computational cost (finite right-hand side), i.e., with respect to the computational time. Third, extending available results from the literature~\cite{cn2012,ffp2014,bhp2017}, Theorem~\ref{theorem:aisfem:complexity} proves that, for sufficiently small adaptivity parameters, the proposed algorithm has optimal complexity (which follows from optimal rates with respect to the degrees of freedom and \eqref{eq:intro:complexity}).
Finally, we admit that the proposed strategy hinges crucially on the appropriate (sufficiently small) choice of the Zarantonello parameter $\delta > 0$ in \eqref{eq:intro:zarantonello} as well as on the parameter $\lamalg > 0$ in the stopping criterion for the algebraic solver in Algorithm~\ref{algorithm}(i.b.II) below. If these parameters are chosen too large, the proposed method may fail to converge. Besides this restriction, linear convergence \eqref{eq:intro:linear-convergence} is guaranteed for any choice of the other adaptivity parameters $\lamsym, \theta, \Cmark$ (see Algorithm~\ref{algorithm} below).

%%%%%%%%%%%%%%%%%%%%%%%%%%%%%%%%%%%%%%%%%%%%%%%%%%%%%%%%%%%%%%%%
\subsection*{Outline}
%%%%%%%%%%%%%%%%%%%%%%%%%%%%%%%%%%%%%%%%%%%%%%%%%%%%%%%%%%%%%%%%

The remainder of the work is organized as follows. Section~\ref{section:abstract} focuses on the setting and underlying assumptions. In Section~\ref{section:algorithm}, we present the AISFEM algorithm in full detail and highlight some of its properties. The main results of this work are presented in Section~\ref{section:main_results}, the proofs of which are given in Section~\ref{section:proofs}.
Numerical experiments in Section~\ref{section:numerics} underline the theoretical results, before the short Section~\ref{section:conclusion} concludes our results and outlines future work.
Throughout, $A \lesssim B$ denotes $A \le c \, B$ with a generic constant $c > 0$ that is independent of the discretization, but may depend on all problem parameters. Moreover, $A \simeq B$ abbreviates $A \lesssim B \lesssim A$. 

\section{Preliminaries}
\label{section:abstract}

In this section, we state all prerequisites to formulate the AISFEM algorithm (Algorithm~\ref{algorithm} in Section~\ref{section:algorithm} below). In particular, we collect the contraction properties of the Zarantonello symmetrization, the algebraic solver, the mesh-refinement strategy, and the required properties of the \textsl{a~posteriori} error estimator.

%%%%%%%%%%%%%%%%%%%%%%%%%%%%%%%%%%%%%%%%%%%%%%%%%%%%%%%%%%%%%%%%
\subsection{Abstract formulation of the model problem}\label{section:abstract-model-problem}
%%%%%%%%%%%%%%%%%%%%%%%%%%%%%%%%%%%%%%%%%%%%%%%%%%%%%%%%%%%%%%%%

According to the Rellich compactness theorem \cite[Theorem~5.8.2]{Kufner1977}, $\dual{\KK u}{v} \coloneqq \dual{\bfb \cdot \nabla u + c u}{v}_\Omega$ defines a compact linear operator $\KK \colon \XX \to \XX'$, where we recall that $\XX' = H^{-1}(\Omega)$ is the dual space of $\XX = H^1_0(\Omega)$. With this notation, the weak formulation~\eqref{eq:intro:weakform} takes the more abstract form
\begin{align}\label{eq:intro:abs_weakform}
	b(u^\exact, v) = a(u^\exact, v) + \dual{\KK u^\exact}{v} = F(v)
	\quad \text{for all } v \in \XX.
\end{align}
Since $b(\cdot,\cdot)$ is continuous and elliptic on $\XX$, i.e., there exists $\alpha_0 > 0$ such that
\begin{align}\label{eq:elliptic}
	\alpha_0 \, \norm{u}_{\XX}^2 \le b(u, u)
	\quad \text{for all } u \in \XX,
\end{align}
a simple compactness argument proves that also the principal part $a(\cdot,\cdot)$ is elliptic, i.e., there exists $\alpha_0' > 0$ such that
\begin{align}
	\alpha'_0 \, \norm{u}_{\XX}^2 \le a(u, u)
	\quad \text{for all } u \in \XX;
\end{align}
see, e.g.~\cite[Remark~3]{bhp2017}.
In particular, $a(\cdot, \cdot)$ is a scalar product on $\XX$ and the $a(\cdot,\cdot)$-induced energy norm $\enorm{v}^2 = a(v,v)$ is an equivalent norm on $\XX$, i.e., $\enorm{v} \simeq \norm{v}_{\XX}$ for all $v \in \XX$. Consequently, $b(\cdot, \cdot)$ is also elliptic and continuous with respect to $\enorm{\, \cdot \,}$, i.e., there exist (in practice unknown) constants $0 < \alpha \le L < \infty$ such that
\begin{align}\label{eq:lax-milgram}
	\alpha \, \enorm{u}^2
	\le
	b(u,u)
	\quad \text{and} \quad
	|b(u,v)|
	\le
	L \, \enorm{u} \, \enorm{v}
	\quad \text{for all } u, v \in \XX.
\end{align}
While this setting already guarantees the C\'ea-type quasi-optimality of Galerkin solutions $u_\ell^\exact \in \XX_\ell \subset \XX$ to~\eqref{eq:intro:discrete}, i.e.,
\begin{align}\label{eq:intro:cea}
	\enorm{u^\exact - u_\ell^\exact}
	\le
	\Ccea \, \min_{v_\ell \in \XX_\ell} \enorm{u^\exact - v_\ell}
	\quad \text{with} \quad \Ccea \coloneqq L/\alpha,
\end{align}
we recall from~\cite[Theorem~20]{bhp2017} that adaptivity improves the constant $\Ccea$ in the C\'ea-type estimate~\eqref{eq:intro:cea}: If $\XX_\ell \subseteq \XX_{\ell+1}$ and $\enorm{u^\exact - u_\ell^\exact} \to 0$ as $\ell \to \infty$, then~\eqref{eq:intro:cea} holds with a constant $1 \le C_\ell \le L/\alpha$ and $C_\ell \to 1$ as $\ell \to \infty$.

\begin{remark}
	The contractive Zarantonello symmetrization and hence the results of this work hold in an abstract framework beyond that of the introduction in Section~\ref{section:introduction}. More precisely, the analysis allows for an abstract separable Hilbert space $\XX$ over $\K \in \{\R, \C\}$ with norm $\norm{\cdot}_{\XX}$ and a weak formulation~\eqref{eq:intro:abs_weakform}, where $a(\cdot,\cdot)$ is a Hermitian and continuous sesquilinear form on $\XX$ and $\KK \colon \XX \to \XX'$ is a compact linear operator such that $b(\cdot,\cdot)$ is elliptic and continuous on $\XX$. Provided that a contractive algebraic solver is used (see Section~\ref{sec:algebraic_solver}), the analysis thus also applies to other boundary conditions (e.g., mixed Dirichlet--Neumann--Robin instead of homogeneous Dirichlet boundary conditions used in the introduction). 
\end{remark}

%%%%%%%%%%%%%%%%%%%%%%%%%%%%%%%%%%%%%%%%%%%%%%%%%%%%%%%%%%%%%%%%
\subsection{Mesh refinement}\label{subsection:mesh-refinement}
%%%%%%%%%%%%%%%%%%%%%%%%%%%%%%%%%%%%%%%%%%%%%%%%%%%%%%%%%%%%%%%%

From now on, let $\TT_0$ be a given conforming triangulation of $\Omega \subset \R^d$ with $d \ge 1$ which is admissible in the sense of~\cite{stevenson2008} for $d \ge 3$. For mesh refinement, we employ newest vertex bisection (NVB); see \cite{AFFKP13} for $d=1$, \cite{stevenson2008} for $d \ge 2$ and~\cite{kpp2013}
for $d=2$ with non-admissible $\TT_0$. For each triangulation $\TT_\coarse$ and marked elements $\MM_\coarse \subseteq \TT_\coarse$, let $\TT_\fine \coloneqq \refine(\TT_\coarse,\MM_\coarse)$ be the coarsest conforming triangulation where all $T \in \MM_\coarse$ have been refined, i.e.,
$\MM_\coarse \subseteq \TT_\coarse \backslash \TT_\fine$. We write $\TT_\fine \in \T(\TT_\coarse)$ if $\TT_\fine$ results from $\TT_\coarse$ by finitely many steps of refinement and, for $N \in \N_0$, we write $\TT_\fine \in \T_N(\TT_\coarse)$ if
$\TT_\fine \in \T(\TT­­­_{\!\!\coarse})$ and $\# \TT_\fine - \# \TT_\coarse \le N$. To abbreviate notation, let $\T\coloneqq\T(\TT_0)$.
Throughout, each triangulation $\TT_\coarse \in \T$ is associated with a finite-dimensional finite element space $\XX_\coarse \subset \XX$, see~\eqref{eq:def:X_ell}, and refinement $\TT_\fine \in \T(\TT_\coarse)$ implies nestedness $\XX_\coarse \subseteq \XX_\fine \subset \XX$.

Within the setting of AFEM, we will work with a hierarchy $\{ \TT_\ell \}_{\ell \in \N_0}$ generated by NVB refinements from the initial mesh $\TT_0$.

%%%%%%%%%%%%%%%%%%%%%%%%%%%%%%%%%%%%%%%%%%%%%%%%%%%%%%%%%%%%%%%%
\subsection{\textsl{A~posteriori} error estimator and axioms of adaptivity}\label{subsection:axioms}
%%%%%%%%%%%%%%%%%%%%%%%%%%%%%%%%%%%%%%%%%%%%%%%%%%%%%%%%%%%%%%%%
For $\TT_\coarse \in \T$, let
\begin{align}
	\begin{split}\label{eq:estimator:generic}
		\eta_\coarse(T; \cdot)\colon \XX_\coarse \to \R_{\ge 0}\quad \text{ for all } T \in \TT_\coarse
	\end{split}
\end{align}
be the local contributions of some computable error estimator.
We define
\begin{equation*}
	\eta_\coarse(\UU_\coarse; v_\coarse)
	\coloneqq
	\Bigl( \sum_{T \in \,\UU_\coarse} \eta_\coarse(T; v_\coarse)^2 \Bigr)^{1/2}
	\quad
	\text{for all } \UU_\coarse \subseteq \TT_\coarse \text{ and } v_\coarse \in \XX_\coarse.
\end{equation*}
To abbreviate notation, let $\eta_\coarse(v_\coarse) \coloneqq \eta_\coarse(\TT_\coarse; v_\coarse)$. Furthermore, we suppose that $\eta_\coarse$ satisfies the following \emph{axioms of adaptivity} from~\cite{axioms} with constants $\Cstab, \Crel$, $ \Cdrel > 0$ and $0 < \qred < 1$ only depending on the dimension $d$, the polynomial degree $m$, and shape regularity of $\TT_0$:

\begin{enumerate}[label={\textbf{(A\arabic*)}\ }, ref={A\arabic*}]
	\item \label{axiom:stability} \textbf{stability:} For all $\TT_\coarse \in \T$ and $\TT_\fine \in \T(\TT_\coarse)$, all $v_\fine \in \XX_\fine$ and all $v_\coarse \in \XX_\coarse$, and every $\UU_\coarse \subseteq \TT_\coarse \cap \TT_\fine$, it holds that
	\begin{align*}
		\vert \eta_\fine (\UU_\coarse, v_\fine) - \eta_\coarse(\UU_\coarse, v_\coarse) \vert \le \Cstab \, \enorm{v_\fine - v_\coarse}.
	\end{align*}
	\item \label{axiom:reduction} \textbf{reduction:} For all $\TT_\coarse \in \T$ and $\TT_\fine \in \T(\TT_\coarse)$, and all $v_\coarse \in \XX_\coarse$, it holds that
	\begin{align*}
		\eta_\fine(\TT_\fine \setminus \TT_\coarse, v_\coarse) \le \qred \, \eta_\coarse(\TT_\coarse \setminus \TT_\fine, v_\coarse).
	\end{align*}
	\item \label{axiom:reliability}  \textbf{reliability:} For all $\TT_\coarse \in \T$, the exact solutions $u^\exact \in \XX$ of~\eqref{eq:intro:weakform} and $u_\coarse^\exact \in \XX_\coarse$ of~\eqref{eq:intro:discrete} satisfy that
	\begin{align*}
		\enorm{u^\exact - u_\coarse^\exact} \le \Crel \, \eta_\coarse(u_\coarse^\exact).
	\end{align*}
	\item \label{axiom:discrete_reliability} \textbf{discrete reliability:}
	For all $\TT_\coarse \in \T$ and $\TT_\fine \in \T(\TT_\coarse)$, the corresponding exact discrete solutions satisfy that
	\begin{align*}
		\enorm{u_h^\star - u_H^\star} \le \Cdrel \, \eta_H(\TT_H \backslash \TT_h, u_H^\star).
	\end{align*}
\end{enumerate}
We note that these axioms \eqref{axiom:stability}--\eqref{axiom:discrete_reliability} are satisfied for the standard residual error estimators; see Section~\ref{section:numerics} below for the model problem~\eqref{eq:intro:model_pb} from the introduction.

%%%%%%%%%%%%%%%%%%%%%%%%%%%%%%%%%%%%%%%%%%%%%%%%%%%%%%%%%%%%%%%%
\subsection{Contractive Zarantonello symmetrization}
%%%%%%%%%%%%%%%%%%%%%%%%%%%%%%%%%%%%%%%%%%%%%%%%%%%%%%%%%%%%%%%%
Recall $0 < \alpha \le L$ from~\eqref{eq:lax-milgram}.
It is well known \cite[Section~25.4]{Zeidler1990} that the Zarantonello mapping $\Phi_\coarse(\delta;\cdot)$ introduced in~\eqref{eq:intro:zarantonello} is a contraction for sufficiently small $\delta > 0$, i.e., for $0 < \delta < 2 \alpha/L^2$. Indeed, for all $u_\coarse, w_\coarse \in \XX_\coarse$, there holds
\begin{align}\label{eq:zarantonello:contraction}
	\enorm{\Phi_\coarse(\delta; u_\coarse) - \Phi_\coarse(\delta; w_\coarse)}
	\le
	q[\delta] \, \enorm{u_\coarse - w_\coarse}
	\quad \! \! \text{with} \quad \! \!
	q[\delta] \coloneqq 1 - \delta(2\alpha - \delta L^2) < 1.
\end{align}
Theoretically, $\delta^\exact \coloneqq \alpha/L^2$ minimizes the expression in~\eqref{eq:zarantonello:contraction} resulting in $q[\delta^\exact] = 1 - \alpha^2/L^2$; see, e.g.,~\cite{hw2020}.

%%%%%%%%%%%%%%%%%%%%%%%%%%%%%%%%%%%%%%%%%%%%%%%%%%%%%%%%%%%%%%%%
\subsection{Contractive algebraic solver}\label{sec:algebraic_solver}
%%%%%%%%%%%%%%%%%%%%%%%%%%%%%%%%%%%%%%%%%%%%%%%%%%%%%%%%%%%%%%%%

We assume that we have at hand an iterative algebraic solver with iteration step $\Psi_\coarse \colon \XX' \times \XX_\coarse \to \XX_\coarse$. This means, given a linear and continuous functional $G \in \XX'$ and an approximation $w_\coarse \in \XX_\coarse$ of the unique solution $w_\coarse^\exact \in \XX_\coarse$ to
\begin{align}
	a(w_\coarse^\exact, v_\coarse) = G(v_\coarse) \quad \text{for all } v_\coarse \in \XX_\coarse,
\end{align}
the algebraic solver returns an improved $\Psi_\coarse(G; w_\coarse) \in \XX_\coarse$ in the sense that there exists a  constant $0 < \qalg < 1$, which is independent of $G$ and $\XX_\coarse$, such that
\begin{align}\label{eq:contraction:alg}
	\enorm{w_\coarse^\exact - \Psi_\coarse(G; w_\coarse)} \le \qalg \, \enorm{w_\coarse^\exact - w_\coarse}.
\end{align}

To simplify notation when the right-hand side $G$ is complicated or lengthy (as for the Zarantonello iteration~\eqref{eq:intro:zarantonello}), we shall write $\Psi_\coarse(w_\coarse^\exact; \cdot)$ instead of $\Psi_\coarse(G; \cdot)$, even though $w_\coarse^\exact$ is unknown and will never be computed.

In the framework of AFEM, possible examples for such contractive solvers include optimally preconditioned conjugate gradient methods or optimal geometric multigrid methods, see, e.g., \cite{cnx2012} or~\cite{wz2017}, respectively, for approaches focused on lowest-order discretizations and~\cite{imps2022} for an optimal multigrid method which is also robust with respect to the polynomial degree.

%%%%%%%%%%%%%%%%%%%%%%%%%%%%%%%%%%%%%%%%%%%%%%%%%%%%%%%%%%%%%%%%
\section{Completely adaptive algorithm}
\label{section:algorithm}
%%%%%%%%%%%%%%%%%%%%%%%%%%%%%%%%%%%%%%%%%%%%%%%%%%%%%%%%%%%%%%%%

In the following, we formulate an inexact adaptive iteratively symmetrized finite element method (AISFEM) in the spirit of \cite{hpsv2021}.
%, which is idealized in the sense that we assume that a sufficiently small $\delta > 0$ guaranteeing contraction~\eqref{eq:zarantonello:contraction} is employed.
For ease of presentation, we make the following conventions: Algorithm~\ref{algorithm} defines certain terminal indices $\elll$, $\kk[\ell]$, $\jj[\ell,k]$, indicated by underlining. We shall omit the arguments of $\kk$ and $\jj$ if these are clear from the context, e.g., we simply write
\begin{align*}
	u_\ell^{k,\jj} \coloneqq u_\ell^{k,\jj[\ell,k]}
	\quad \text{and} \quad
	u_\ell^{\kk,\jj} \coloneqq u_\ell^{\kk[\ell], \jj[\ell, \kk[\ell]]},
	\quad\text{etc.}
\end{align*}
A similar convention will be used for triple indices, e.g., $(\ell,k,\jj) = (\ell,k, \jj[\ell,k])$, etc.

\begin{algorithm}[adaptive iteratively symmetrized finite element method (AISFEM)]\label{algorithm}%idealized full AISFEM
	\textbf{Input:} Initial triangulation $\TT_0$, initial guess $u_0^{0, 0} \coloneqq u_0^{0, \jj} \coloneqq 0$, marking parameters $0 < \theta \le 1$ and $\Cmark \ge 1$, solver parameters $\lamsym, \lamalg > 0$,
	%tolerance $\tau > 0$,
	and damping parameter $\delta > 0$. 
	
	\noindent\textbf{Loop:} For $\ell = 0, 1, 2, \dots$, repeat the following steps {\rm(i)--(iv)}:
	\begin{enumerate}
		\item[\rm(i)] For all $k = 1, 2, 3, \dots$, repeat the following steps {\rm(a)--(d)}:
		\begin{enumerate}
			\item[\rm(a)] Define $u_\ell^{k,0} \coloneqq u_\ell^{k-1, \jj}$ and, for purely theoretical reasons, $u_\ell^{k,\star} \coloneqq \Phi_\ell(\delta; u_\ell^{k-1,\jj})$.
			\item[\rm(b)] For all $j=1,2,3, \dots$ repeat the following steps $(\rm I)$--$(\rm II)$:
			\begin{enumerate}
				\item[\normalfont\textrm{(I)} ]  Compute $u_\ell^{k, j} \coloneqq \Psi_\ell(u_\ell^{k,\star}; u_\ell^{k, j-1})$ and $\eta_\ell(T; u_\ell^{k,j})$ for all $T \in \TT_\ell$.
				\item[\normalfont\textrm{(II)} ] Terminate $j$-loop if $\enorm{u_\ell^{k, j} - u_\ell^{k,j-1}} \le \lamalg \, \bigl[\lamsym\eta_{\ell}(u_\ell^{k,j}) + \enorm{u_\ell^{k,j} - u_\ell^{k-1, \jj}} \bigr]$.
			\end{enumerate}
			\item[\rm(c)] Upon termination of the $j$-loop, define $\jj[\ell, k] \coloneqq j$.
			\item[\rm(d)] Terminate $k$-loop if $\enorm{u_\ell^{k, \jj} - u_\ell^{k-1, \jj}} \le \lamsym \, \eta_\ell(u_\ell^{k, \jj})$.
		\end{enumerate}
		\item[\rm(ii)] Upon termination of the $k$-loop, define $\kk[\ell] \coloneqq k$.
		\item[\rm(iii)] Determine $\MM_\ell \subseteq \TT_\ell$ of up to the constant $\Cmark$ minimal cardinality satisfying $\theta \, \eta_\ell(u_\ell^{\kk, \jj})^2 \le \eta_\ell(\MM_\ell; u_\ell^{\kk, \jj})^2$.
		\item[\rm(iv)] Generate $\TT_{\ell+1} \coloneqq \refine(\TT_\ell, \MM_\ell)$ and define $u_{\ell+1}^{0,0} \coloneqq u_{\ell+1}^{0,\jj} \coloneqq u_{\ell+1}^{0, \exact} \coloneqq u_\ell^{\kk,\jj}$.
	\end{enumerate}
	\textbf{Output:} Discrete approximations $u_\ell^{k,j}$ and corresponding error estimators $\eta_\ell(u_\ell^{k, j})$.
\end{algorithm}
\begin{remark}
	To give an interpretation of the stopping criteria in Step~{\rm(i.b.II)} and Step~{\rm(i.d)} of Algorithm~\ref{algorithm}, we note the following: Since the algebraic solver is contractive~\eqref{eq:contraction:alg}, the term $\enorm{u_\ell^{k, j} - u_\ell^{k, j-1}}$ provides \textsl{a~posteriori} error control of the algebraic error $\enorm{u_\ell^{k, \exact} - u_\ell^{k, j}}$, i.e., 
	\[
		\enorm{u_\ell^{k, \exact} - u_\ell^{k, j}} 
		\le
		\frac{\qalg}{1-\qalg} \, \enorm{u_\ell^{k, j} - u_\ell^{k, j-1}}.
	\]
	Moreover, for sufficiently small $\lamalg > 0$ and ongoing Zarantonello iterations, also the perturbed Zarantonello symmetrization is a contraction; see Lemma~\ref{lem:contraction_perturbed} below. With the same reasoning as for the algebraic solver, the term $\enorm{u_\ell^{k, \jj} - u_\ell^{k-1, \jj}} = \enorm{u_\ell^{k, \jj} - u_\ell^{k, 0}}$ thus provides \textsl{a~posteriori} error control of the symmetrization error $\enorm{u_\ell^{\exact} - u_\ell^{k, \star}} \approx \enorm{u_\ell^{\exact} - u_\ell^{k, \jj}}$ (at least if $1 \le k < \kk[\ell]$). With this understanding and the interpretation that the error estimator $\eta_{\ell}(u_\ell^{k, j})$ controls the discretization error $\enorm{u^{\exact} - u_\ell^{\exact}}$ (which is indeed true for $u_\ell^{k, j} = u_\ell^{\kk, \jj}$), the heuristics behind the stopping criteria is as follows: We stop the algebraic solver in Algorithm~\ref{algorithm}{\rm(i.b.II)} provided that the algebraic error $\enorm{u_\ell^{k, \exact} - u_\ell^{k, j}}$ is of the level of the discretization error plus the symmetrization error. Moreover, we stop the (perturbed) Zarantonello symmetrization in Algorithm~\ref{algorithm}{\rm(i.d)} provided that the symmetrization error $\enorm{u_\ell^{\exact} - u_\ell^{k, \jj}}$ is of the level of the discretization error. Up to the factors $\lamalg$ and $\lamsym$, this ensures that all three error sources of $\enorm{u^{\exact} - u_\ell^{\kk, \jj}}$ are equibalanced.
	
\end{remark}
For the analysis of Algorithm~\ref{algorithm}, we recall that the set $\QQ$ from \eqref{eq:intro:index-set} is given by
\begin{align*}
	\QQ \coloneqq \set{(\ell, k, j) \in \N_0^3 \given u_\ell^{k, j} \text{ is used in Algorithm~\ref{algorithm}}}.
\end{align*}
Together with this set, we define
\begin{subequations}\label{eq:def:final-indices}
	\begin{align}
		\elll &\coloneqq \sup\set{\ell \in \N_0 \given (\ell,0,0) \in \QQ} \in \N_0 \cup \{\infty\},
		\\
		\kk[\ell] &\coloneqq \sup\set{k \in \N_0 \given (\ell,k,0) \in \QQ} \in \N_0 \cup \{\infty\},
		\quad \text{whenever } (\ell,0,0) \in \QQ,
		\\
		\jj[\ell,k] &\coloneqq \sup\set{j \in \N_0 \given (\ell,k,j) \in \QQ} \in \N_0 \cup \{\infty\},
		\quad \text{whenever } (\ell,k,0) \in \QQ.
	\end{align}
\end{subequations}
Note that these definitions are consistent with that of Algorithm~\ref{algorithm}, but also cover the cases that the $\ell$-loop, the $k$-loop, or the $j$-loop in the algorithm do not terminate, respectively. We note that formally $\# \QQ = \infty$ and hence either $\elll = \infty$ or $\kk[\elll] = \infty$ or $\jj[\elll, \kk[\elll]] = \infty$, where the latter case is excluded by Lemma~\ref{lemma:termination-j}.

On $\QQ$, we define a total order by
\begin{align*}
	(\ell', k', j') \le (\ell, k, j)
	\quad \Longleftrightarrow \quad
	u_{\ell'}^{k', j'} \text{ is computed in Algorithm~\ref{algorithm} not later than } u_\ell^{k,j}.
\end{align*}
Furthermore, we introduce the total step counter $\vert \cdot, \cdot, \cdot \vert$, defined for all $(\ell, k, j) \in \QQ$, by
\begin{align}\label{eq:stepcounter}
	\vert \ell, k, j \vert \coloneqq \# \set{(\ell', k', j') \in \QQ \given (\ell', k', j') \le (\ell, k, j)}
	\in \N_0.
\end{align}

Our first observation is that the algebraic solver in the innermost loop of Algorithm~\ref{algorithm} always terminates.

\begin{lemma}\label{lemma:termination-j}
	Independently of the adaptivity parameters $\theta$, $\lamsym$, and $\lamalg$, the $j$-loop of Algorithm~\ref{algorithm} always terminates, i.e., $\jj[\ell,k] < \infty$ for all $(\ell,k,0) \in \QQ$.
\end{lemma}

\begin{proof}
  Let $(\ell,k,0) \in \QQ$.
  We argue by contradiction and assume that the stopping criterion in Algorithm~\ref{algorithm}(i.b.II) always fails and hence $\jj[\ell,k] = \infty$. By assumption~\eqref{eq:contraction:alg}, the algebraic solver is contractive and hence convergent with limit $u_\ell^{k,\star} \coloneqq \Phi_\ell(\delta; u_\ell^{k-1,\jj})$. Moreover, by failure of the stopping criterion in Algorithm~\ref{algorithm}(i.b.II), we thus obtain that
  \begin{align*}
	    \eta_{\ell}(u_\ell^{k,j}) + \enorm{u_\ell^{k,j} - u_\ell^{k-1, \jj}}
	    \lesssim \enorm{u_\ell^{k, j} - u_\ell^{k,j-1}}
	    \xrightarrow{j \to \infty} 0.
	  \end{align*}
  This yields $\enorm{u_\ell^{k,\star} - u_\ell^{k-1, \jj}} = 0$.
  Consequently, $u_\ell^{k-1, \jj}$ is a fixpoint of $\Phi_\ell(\delta; \cdot)$, cf.\ Algorithm~\ref{algorithm}(i.a), and hence $u_\ell^{k-1, \jj} = u_\ell^\exact$ by uniqueness of the fixpoint. In particular, the initial guess $u_\ell^{k,0} = u_\ell^{k-1, \jj} = u_\ell^{k,\star}$ is already the exact solution of the linear Zarantonello system and hence the algebraic solver guarantees that $u_\ell^{k,j} = u_\ell^{k,\star}$ for all $j \in \N_0$. Consequently, the stopping criterion in Algorithm~\ref{algorithm}(i.b.II) will be satisfied for $j = 1$. This contradicts our assumption, and hence we conclude that $\jj[\ell,k] < \infty$.
\end{proof}

\begin{remark}
	For the mathematical tractability, we formulated Algorithm~\ref{algorithm} in a way that $\#\QQ = \infty$. Any practical implementation will aim to provide a sufficiently accurate approximation $u_\ell^{k, j}$ in finite time. More precisely,  Algorithm~\ref{algorithm} will then be terminated after Algorithm~\ref{algorithm}{\rm(i.b.II)} if
	\begin{align}\label{eq1:remark3}
		\eta_{\elll}(u_\elll^{\kk,\jj}) + \enorm{u_\elll^{\kk,\jj} - u_\elll^{\kk-1, \jj}} + \enorm{u_\elll^{\kk,\jj} - u_\elll^{\kk, \jj-1}} \le \tau
	\end{align}
	where $\tau > 0$ is a user-specified tolerance. For $\tau = 0$, finite termination yields that $u_\elll^{\kk,\jj} = u^\exact$ with $\eta_\elll(u_\elll^{\kk,\jj}) = 0$. To see this, note that \eqref{eq1:remark3} implies $u_\elll^{\kk,\star} = u_\elll^{\kk,\jj} = u_\elll^{\kk, \jj-1}$ and $u_\elll^\star = u_\elll^{\kk,\jj} = u_\elll^{\kk-1, \jj}$ by uniqueness of the fixpoint of the contractive solver and the contractive Zarantonello symmetrization, respectively. Finally, the first summand in \eqref{eq1:remark3} states $\eta_\elll(u_\elll^\exact) = \eta_{\elll}(u_\elll^{\kk,\jj}) = 0$ and hence $u_\elll^{\kk,\jj} = u_\elll^\exact = u^\exact$ by reliability~\eqref{axiom:reliability} of the estimator.
\end{remark}

\begin{remark}
	Up to the algebraic stopping criterion in Algorithm~\ref{algorithm}{\rm(i.b.II)}, the AISFEM algorithm coincides with the adaptive algorithm from~\cite{hpsv2021}, where the (perturbed) Zarantonello iteration is employed for an adaptive iteratively linearized finite element method for the solution of an energy minimization problem with strongly monotone nonlinearity in the corresponding Euler--Lagrange equations. However, the present analysis is much more refined than that of~\cite{hpsv2021}: 
	
	{\rm (i)} To guarantee full linear convergence, \cite[Theorem~4]{hpsv2021} requires $\theta$ sufficiently small, $\lamsym$ sufficiently small with respect to $\theta$, and $\lamalg$ sufficiently small with respect to $\lamsym$.
	In contrast, the present analysis proves full linear convergence for arbitrary $0 < \theta \le 1$ and $0 < \lamsym \le 1$, and only requires $\lamalg$ to be sufficiently small
%		In contrast, we require $\lamalg$ to be sufficiently small with respect to $0 < \qalg < 1$ and $0 < \qsym < 1$ 
to preserve the contraction of the perturbed Zarantonello iteration (see Lemma~\ref{lem:contraction_perturbed} below in comparison to~\cite[Lemma~6]{hpsv2021}).
%, while $\theta$ and $\lamsym$ can be arbitrary.
	
	{\rm (ii)} Despite the linear model problem, our analytical setting is more involved: the compact perturbation in~\eqref{eq:intro:abs_weakform} prevents the use of energy arguments that guarantee a Pythagorean-type identity in terms of the energy error (see, e.g.,~\cite{hpsv2021, hpw2021}). Instead, we first need to exploit \textsl{a~priori} convergence of Algorithm~\ref{algorithm} (see Lemma~\ref{prop:plain-convergence}) to deduce a quasi-Pythagorean estimate in Lemma~\ref{lemma:quasi-pythagoras}, which then allows proving linear convergence (Theorem~\ref{theorem:aisfem:linear-convergence}). As a consequence (and beyond the results of~\cite{hpsv2021}), this finally yields that, for arbitrary $\theta$ and $\lamsym$, the convergence rates with respect to the number of the degrees of freedom and with respect to the overall computational work coincide (Corollary~\ref{corollary:aisfem:linear-convergence}).
\end{remark}

The following proposition provides a computable upper bound for the energy error $\enorm{u^\exact - u_\ell^{k, j}} $. Since Algorithm~\ref{algorithm} follows the structure of~\cite[Algorithm~1]{hpsv2021}, the proof can be obtained analogously to \cite[Proposition~2]{hpsv2021} and is thus omitted here.

\begin{proposition}[reliable error control]\label{proposition:reliability}
	Suppose that the estimator satisfies \eqref{axiom:stability} and \eqref{axiom:reliability}. Then, for all $(\ell, k, j) \in \QQ$, it holds that
	\begin{align} \label{eq:reliable-error}
		\enorm{u^\exact - u_\ell^{k, j}} \le \Crel' \begin{cases}
			\eta_{\ell}(u_\ell^{k,j}) + \enorm{u_{\ell}^{k, j} - u_{\ell}^{k-1, \jj}}& \\ \hphantom{\eta_{\ell}(u_\ell^{k,j})}+ \enorm{u_{\ell}^{k, j} - u_{\ell}^{k, j-1}} \quad &\text{if } 1 \le k \le \kk[\ell] \text{ and } 1 \le j < \jj[\ell, k], \\
			\eta_{\ell}(u_\ell^{k,\jj}) + \enorm{u_{\ell}^{k, \jj} - u_{\ell}^{k-1, \jj}} \quad &\text{if } 1 \le k \le \kk[\ell] \text{ and } j = \jj[\ell, k], \\
			\eta_{\ell}(u_\ell^{\kk,\jj}) &\text{if } k = \kk[\ell] \text{ and } j = \jj[\ell, \kk], \\
			\eta_{\ell-1}(u_{\ell-1}^{\kk,\jj}) &\text{if } \ell > 0 \text{ and } k = 0.
		\end{cases}
	\end{align}
	The constant $\Crel' >0$ depends only on $\Crel$, $\Cstab$, $\qalg$, $\lamalg$, $\qsym$, and $\lamsym$.
\end{proposition}

%%%%%%%%%%%%%%%%%%%%%%%%%%%%%%%%%%%%%%%%%%%%%%%%%%%%%%%%%%%%%%%%
\section{Main results}\label{section:main_results}
%%%%%%%%%%%%%%%%%%%%%%%%%%%%%%%%%%%%%%%%%%%%%%%%%%%%%%%%%%%%%%%%

In the following, we formulate the main results of the present work. We refer to Section~\ref{section:proofs} for the proofs and Section~\ref{section:numerics} for numerical experiments, which underline these theoretical results. First, recall from~\eqref{eq:zarantonello:contraction} that a sufficiently small parameter $\delta > 0$ ensures contraction of the Zarantonello mapping and hence
\begin{align}\label{eq:zarantonello:unperturbed}
	\enorm{u_\ell^\exact - u_\ell^{k,\exact}}
	\le \qsym \, \enorm{u_\ell^\exact - u_\ell^{k-1,\jj}}
	\quad \text{for all } (\ell,k,0) \in \QQ
\end{align}
with $0 < \qsym <1$.
The following theorem states full linear convergence of the quasi-error.

\begin{theorem}[full linear convergence of AISFEM]\label{theorem:aisfem:linear-convergence}
	Suppose that $\delta > 0$ is sufficiently small and that the estimator satisfies~\eqref{axiom:stability}--\eqref{axiom:reliability}. 
Choose $\lamalg^\exact > 0$ depending only on $\qalg$ from \eqref{eq:contraction:alg} and $\qsym$ from \eqref{eq:zarantonello:unperturbed} such that
	\begin{align}\label{eq1:lem:contraction_perturbed}
		0 < \qsymm \coloneqq \frac{\qsym + 2 \, \frac{\qalg}{1-\qalg}\, \lamalg^\exact}{1- 2 \, \frac{\qalg}{1-\qalg}\, \lamalg^\exact} < 1.
	\end{align}
	Then, for arbitrary $0 < \theta \le 1$ and $0 < \lamsym \le 1$, there exists \highlighted{$0 < \lamalg' \le \lamalg^\exact$} such that Algorithm~\ref{algorithm}, for all $0 < \lamalg \le \lamalg'$, 	
	%	 but sufficiently small $\lamalg$ satisfying %$0 < (1 + 1 / \lamsym) \lamalg < \lamalg^\star$,
% 	$0 < \lamalg \le \lamalg^\star$, Algorithm~\ref{algorithm} 
guarantees full linear convergence:
	There exist constants $\Clin > 0$ and $0 < \qlin < 1$ as well as an index $\ell_0 \in \N_0$ with $\ell_0 \le \elll$ such that the quasi-error
	\begin{align}\label{eq0:theorem:aisfem:linear-convergence}
		\Delta_\ell^{k,j}
		\coloneqq
		\enorm{u^\exact - u_\ell^{k,j}} + \enorm{u_\ell^{k,\exact} - u_\ell^{k,j}} + \eta_\ell(u_\ell^{k,j})
		\quad \text{for all } (\ell, k, j) \in \QQ
	\end{align}
	satisfies that, for all $(\ell,k,j), (\ell',k',j') \in \QQ$ with $|\ell,k,j| > |\ell',k',j'|$ and $\ell' \ge \ell_0$,
	\begin{align}\label{eq:theorem:aisfem:linear-convergence}
		\Delta_\ell^{k,j}
		\le
		\Clin \qlin^{|\ell,k,j| - |\ell',k',j'|} \, \Delta_{\ell'}^{k',j'}.
	\end{align}
	The constants $\Clin$ and $\qlin$ as well as the index $\ell_0$ depend only on~$\Cstab$, $\Crel$, $\qred$, $\qsym$, $\qalg$, $\theta$, $\lamsym$, $\lamalg$, and $\Ccea = L/\alpha$.
\end{theorem}

While the proof of Theorem~\ref{theorem:aisfem:linear-convergence} is postponed to Section~\ref{subsection:proof-linear-convergence}, we shall immediately prove the following important consequence of Theorem~\ref{theorem:aisfem:linear-convergence}: Algorithm~\ref{algorithm} guarantees that rates with respect to the number of degrees of freedom coincide with rates with respect to the overall computational cost.

\begin{corollary}\label{corollary:aisfem:linear-convergence}
	Let $s > 0$. Under the assumptions of Theorem~\ref{theorem:aisfem:linear-convergence}, the output of Algorithm~\ref{algorithm} guarantees that
	\begin{align}\label{eq:corollary:aisfem:linear-convergence}
		M(s) \coloneqq \! \sup_{\substack{(\ell,k,j) \in \QQ \\ \ell \ge\ell_0}} (\#\TT_\ell)^s \, \Delta_\ell^{k,j}
		\le \sup_{\substack{(\ell,k,j) \in \QQ \\ \ell \ge \ell_0}}\Bigl( \! \sum_{\substack{(\ell',k',j') \in \QQ \\ |\ell',k',j'| \le |\ell,k,j| \\ \ell' \ge \ell_0}} \! \#\TT_{\ell'}\Bigr)^s \, \Delta_\ell^{k,j}
		\le \frac{\Clin}{\bigl(1 - \qlin^{1/s}\bigr)^{s}} \, M(s).
	\end{align} 
	This yields the equivalence
	\begin{align}\label{eq2:corollary:aisfem:linear-convergence}
		\sup_{\substack{(\ell,k,j) \in \QQ}} (\#\TT_\ell)^s \, \Delta_\ell^{k,j} < \infty
		\quad \Longleftrightarrow \quad
		\sup_{(\ell,k,j) \in \QQ}
		\Bigl(\sum_{\substack{(\ell',k',j') \in \QQ \\ |\ell',k',j'| \le |\ell,k,j|}} \#\TT_{\ell'}\Bigr)^s \, \Delta_\ell^{k,j} < \infty.
	\end{align}
\end{corollary}

\begin{proof}
	The lower bound in~\eqref{eq:corollary:aisfem:linear-convergence} is obvious. To prove the upper bound, without loss of generality, we may assume that $M(s) < \infty$. 
	%Let $(\ell,k,j) \in \QQ$ with $\ell \ge \ell_0$. 
	By definition of $M(s)$, it follows that
	\begin{align}\label{eq3:corollary:aisfem:linear-convergence}
		\#\TT_{\ell'}
		\le M(s)^{1/s} [\Delta_{\ell'}^{k',j'}]^{-1/s} \quad \text{for } (\ell', k', j') \in \QQ \text{ with } \ell' \ge \ell_0.
	\end{align}
	For $|\ell,k,j| \ge |\ell',k',j'|$ and $\ell' \ge \ell_0$,
	full linear convergence~\eqref{eq:theorem:aisfem:linear-convergence} can be rewritten as
	\begin{align}\label{eq4:corollary:aisfem:linear-convergence}
		[\Delta_{\ell'}^{k',j'}]^{-1/s}
		\le \Clin^{1/s} [\qlin^{1/s}]^{|\ell,k,j| - |\ell',k',j'|} \, [\Delta_\ell^{k,j}]^{-1/s}.
	\end{align}
	The geometric series yields that
	\begin{align*}
		\sum_{\substack{(\ell',k',j') \in \QQ \\ |\ell',k',j'|
				\le
				|\ell,k,j| \\ \ell' \ge \ell_0}} \#\TT_{\ell'}
		\eqreff*{eq3:corollary:aisfem:linear-convergence}\le
		M(s)^{1/s} \, \sum_{\substack{(\ell',k',j') \in \QQ \\ |\ell',k',j'| \le |\ell,k,j| \\ \ell' \ge \ell_0}} [\Delta_{\ell'}^{k',j'}]^{-1/s}
		\eqreff*{eq4:corollary:aisfem:linear-convergence}\le
		M(s)^{1/s} \, \Clin^{1/s} \frac{1}{1 - \qlin^{1/s}} \, [\Delta_\ell^{k,j}]^{-1/s}.
	\end{align*}
	Rearranging this estimate, we see that
	\begin{align*}
		\Bigl( \sum_{\substack{(\ell',k',j') \in \QQ \\ |\ell',k',j'| \le |\ell,k,j| \\ \ell' \ge \ell_0}} \#\TT_{\ell'} \Bigr)^s \, \Delta_\ell^{k,j}
		\le M(s) \, \Clin \, \frac{1}{\bigl(1 - \qlin^{1/s}\bigr)^s}.
	\end{align*}
	Taking the supremum over all $(\ell,k,j) \in \QQ$ with $\ell \ge \ell_0$, we prove the second estimate in~\eqref{eq:corollary:aisfem:linear-convergence}. Moreover,
	\begin{align*}
		\QQ \backslash \set{(\ell,k,j) \in \QQ \given \ell \ge \ell_0}
		= \set{(\ell,k,j) \in \QQ \given \ell < \ell_0}
		\quad \text{is finite},
	\end{align*}
	i.e., the sets over which we compute the suprema in~\eqref{eq:corollary:aisfem:linear-convergence}--\eqref{eq2:corollary:aisfem:linear-convergence} differ only by finitely many index triples.
	This and~\eqref{eq:corollary:aisfem:linear-convergence} thus prove the equivalence in~\eqref{eq2:corollary:aisfem:linear-convergence}.
\end{proof}

To present our second main result on quasi-optimal computational cost, we first introduce the notion of approximation classes. For $\TT \in \T $ and $s>0$, define
\begin{align}\label{eq:def_approx_class}
	\norm{u^\exact}_{\A_s (\TT)} \coloneqq \sup_{N \in \N_0} \Bigl( \bigl( N+1 \bigr)^s \min_{\TT_{\rm opt} \in \T_N (\TT) } \bigl[ \enorm{u^\star - u^\star_{\rm opt}} + \eta_{\rm opt} (u^\star_{\rm opt}) \bigr] \Bigr),
\end{align}
with $u_{\rm opt}^\exact$ and $\eta_{\rm opt} $ denoting the exact discrete solution and the estimator on the optimal triangulation $\TT_{\rm opt} \in \T_N (\TT)$, respectively. When~\eqref{eq:def_approx_class} is finite, this means that a decrease of the error plus estimator with rate $s$ is possible along optimal meshes obtained by refining $\TT$.

\def\Calg{C_{\rm alg}}
\begin{theorem}[optimal computational complexity]\label{theorem:aisfem:complexity}
	Suppose that $\delta > 0$ is sufficiently small and that the estimator satisfies \eqref{axiom:stability}--\eqref{axiom:discrete_reliability}. 
Let $0 < \theta < \theta^\exact \coloneqq (1+ \Cstab^2 \, \Cdrel^2)^{-1} < 1$.
Define  $\lamsym^\exact \coloneqq \min\{1, \Calg^{-1}\,\Cstab^{-1}\}$, where
	\begin{equation}\label{eq:def:Calg}
				\Calg \coloneqq \frac{1}{1-\qsym} \, \Bigl(\frac{2 \, 
				\qalg}{1-\qalg} \, \lamalg^\exact + \qsym \Bigr). \tag{Calg}
%		\Calg \coloneqq \frac{\qsym}{1-\qsym} \, \Bigl(2 \, \frac{\qalg}{1-\qalg} \, \lamalg' + 1 \Bigr) +  2 \, \frac{\qalg}{1-\qalg} \, \lamalg'.
	\end{equation}%
	 Choose $0< \lamsym < \lamsym^\exact$ sufficiently small such that 
	\begin{equation}\label{eqxx:theorem:aisfem:complexity}
		0 < \thetamark \coloneqq \Bigl( \frac{\theta^{1/2}  + \, \lamsym / \lamsym^\exact}{1-\lamsym / \lamsym^\exact} \Bigr)^{2} < \theta^\exact.
	\end{equation}
	Then, for any $0 < \lamalg \le \lamalg'$ with $\lamalg' > 0$ from Theorem~\ref{theorem:aisfem:linear-convergence},
	Algorithm~\ref{algorithm} guarantees, for all $s > 0$, that
	\begin{subequations}\label{eq:theorem:aisfem:complexity}
		\begin{align}
			\copt \, \norm{u^\exact}_{\A_s(\TT_{0})}
			&\le \sup_{\substack{(\ell,k,j) \in \QQ}} \Bigl(\sum_{\substack{(\ell',k',j') \in \QQ \\ |\ell',k',j'| \le |\ell,k,j|}} \#\TT_{\ell'}\Bigr)^s \, \Delta_\ell^{k,j}, \\
			\label{eqx:theorem:aisfem:complexity}
			\sup_{\substack{(\ell,k,j) \in \QQ \\ \ell \ge \ell_0}} \Bigl(\sum_{\substack{(\ell',k',j') \in \QQ \\ |\ell',k',j'| \le |\ell,k,j| \\ \ell' \ge \ell_0}} \#\TT_{\ell'}\Bigr)^s \, \Delta_\ell^{k,j} &\le \Copt \,  \max\{\norm{u^\exact}_{\A_s(\TT_{\ell_0})}, \Delta_{\ell_0}^{0,0}\}.
		\end{align}
	\end{subequations}
	where $\ell_0 \in \N$ is the index from Theorem~\ref{theorem:aisfem:linear-convergence}.
	The constant $\copt > 0$ depends only on $\Ccea = L/\alpha$, $\Cstab$, $\Crel$, $s$, and the use of NVB refinement; the constant $\Copt > 0$ depends only on~$\Cstab$, $\Cdrel$, $\Cmark$, $\Ccea = L/\alpha$, $\Crel'$, $\Clin$, $\qlin$, $\# \TT_{\ell_0}$, $\qred$, $\lamsym$, $\qsym$, $\theta$, $s$, and the use of NVB refinement.
	In particular, this proves the equivalence
	\begin{align}\label{eq2:theorem:aisfem:complexity}
		\norm{u^\exact}_{\A_s(\TT_{0})} < \infty
		\quad \Longleftrightarrow \quad
		\sup_{(\ell,k,j) \in \QQ}
		\Bigl(\sum_{\substack{(\ell',k',j') \in \QQ \\ |\ell',k',j'| \le |\ell,k,j|}} \#\TT_{\ell'}\Bigr)^s \, \Delta_\ell^{k,j} < \infty,
	\end{align}
	which yields optimal complexity of Algorithm~\ref{algorithm}.
\end{theorem}

The proof is postponed to Section~\ref{subsection:proof:complexity}.

%%%%%%%%%%%%%%%%%%%%%%%%%%%%%%%%%%%%%%%%%%%%%%%%%%%%%%%%%%%%%%%%
%%%%%%%%%%%%%%%%%%%%%%%%%%%%%%%%%%%%%%%%%%%%%%%%%%%%%%%%%%%%%%%%
\section{Proofs}
\label{section:proofs}
%%%%%%%%%%%%%%%%%%%%%%%%%%%%%%%%%%%%%%%%%%%%%%%%%%%%%%%%%%%%%%%%
%%%%%%%%%%%%%%%%%%%%%%%%%%%%%%%%%%%%%%%%%%%%%%%%%%%%%%%%%%%%%%%%

%%%%%%%%%%%%%%%%%%%%%%%%%%%%%%%%%%%%%%%%%%%%%%%%%%%%%%%%%%%%%%%%
\subsection{Contraction of perturbed Zarantonello symmetrization}
%%%%%%%%%%%%%%%%%%%%%%%%%%%%%%%%%%%%%%%%%%%%%%%%%%%%%%%%%%%%%%%%

Recall that for $\delta < 2\, \delta^\exact = 2 \, \alpha/L^2$, the Zarantonello mapping is a contraction~\eqref{eq:zarantonello:contraction}. However, Algorithm~\ref{algorithm} does not compute $u_\ell^{k,\star} \coloneqq \Phi_\ell(\delta; u_\ell^{k-1,\jj})$ exactly, but relies on an approximation $u_\ell^{k,\jj} \approx u_\ell^{k,\star}$.
The next lemma states that, for a sufficiently small stopping parameter $\lamalg > 0$ in Algorithm~\ref{algorithm}, the Zarantonello symmetrization remains a contraction under this perturbation (up to the final iteration). Its proof essentially follows along the lines of~\cite[Lemma~6]{hpsv2021}. However, the present work considers a stopping criterion of the algebraic solver in Algorithm~\ref{algorithm}(i.b.II) which allows to choose $\lamalg$ independently of $\lamsym$.

\begin{lemma}\label{lem:contraction_perturbed}
	Let $\lamalg^\exact > 0$ and $0 < \qsymm < 1$ as in Theorem~\ref{theorem:aisfem:linear-convergence}. Then, for all stopping parameters $0 < \lamalg \le \lamalg^\exact$ and $\lamsym > 0$,
	it holds that
	\begin{align}\label{eq2:lem:contraction_perturbed}
		\enorm{u_\ell^\exact - u_\ell^{k, \jj}} \le \qsymm \, \enorm{u_\ell^\exact - u_\ell^{k-1, \jj}}
		\quad \text{for all } (\ell,k,\jj) \in \QQ \text{ with } 1 \le k < \kk[\ell].
	\end{align}
	Moreover, for $k = \kk[\ell]$, it holds that
\begin{equation}\label{eq2:lem:contraction_perturbed:kk}
%\medmuskip = 0mu
 \enorm{u_\ell^\exact - u_\ell^{\kk,\jj}}
 \le
 \qsym \, \enorm{u_\ell^\exact - u_\ell^{\kk-1,\jj}} + \frac{2 \, \qalg}{1-\qalg} \, \lamalg \, \lamsym \, \eta_{\ell}(u_\ell^{\kk, \jj})
 \quad \text{for all } (\ell,\kk,\jj) \in \QQ.\tag{5.1$^+$}
\end{equation}%
\end{lemma}

\begin{proof}
Let $(\ell,k,\jj) \in \QQ$ and suppose first that $1 \le k < \kk[\ell]$.
	By using the triangle inequality and the contraction~\eqref{eq:zarantonello:unperturbed} of the unperturbed Zarantonello iteration, we obtain that
	\begin{align}\label{eq1:proof:lemma9}
		\enorm{u_\ell^\exact - u_\ell^{k, \jj}} \le \enorm{u_\ell^\exact - u_\ell^{k, \exact}} + \enorm{u_\ell^{k, \exact} - u_\ell^{k, \jj}}
		\eqreff*{eq:zarantonello:unperturbed}\le \qsym  \, \enorm{u_\ell^\exact - u_\ell^{k-1, \jj}} +  \enorm{u_\ell^{k,\exact} - u_\ell^{k, \jj}}.
	\end{align}
	It remains to treat the algebraic error term and to show that it is sufficiently contractive. We use the contraction~\eqref{eq:contraction:alg} of the algebraic solver, i.e.,
	\begin{align}\label{eq:algebra:contraction}
		\enorm{u_\ell^{k,\star} - u_\ell^{k,j}}
		\le \qalg \, \enorm{u_\ell^{k,\star} - u_\ell^{k,j-1}}
		\quad \text{for all } (\ell,k,j) \in \QQ \text{ with } j \ge 1,
	\end{align}%
	the met algebraic stopping criterion in Algorithm~\ref{algorithm}(i.b.II), and the not met stopping criterion in Algorithm~\ref{algorithm}(i.d) to obtain that
	\begin{align*}
		\enorm{u_\ell^{k,\exact} - &u_\ell^{k, \jj}}
		\eqreff*{eq:algebra:contraction}\le \frac{\qalg}{1- \qalg} \, \enorm{u_\ell^{k, \jj} - u_\ell^{k, \jj -1}}
		\stackrel{\rm (i.b.II)}\le \lamalg \, \frac{\qalg}{1- \qalg} \, \bigl[\lamsym \, \eta_{\ell}(u_\ell^{k,\jj}) + \enorm{u_\ell^{k,\jj} - u_\ell^{k-1, \jj}} \bigr] \\
		&\stackrel{\rm(i.d)}< 2 \, \lamalg \, \frac{\qalg}{1- \qalg} \,  \enorm{u_\ell^{k,\jj} - u_\ell^{k-1, \jj}}
		\le 2 \, \lamalg\, \frac{\qalg}{1- \qalg} \,\bigl[  \enorm{u_\ell^{\exact} - u_\ell^{k,\jj}} + \enorm{u_\ell^{\exact} - u_\ell^{k-1, \jj}} \bigr].
	\end{align*}
	Combining the last estimate with \eqref{eq1:proof:lemma9} and rearranging the terms lead us to
	\begin{align*}
		\enorm{u_\ell^{\exact} - u_\ell^{k, \jj}}
		\le \frac{\qsym + 2 \, \lamalg \, \frac{\qalg}{1- \qalg}}{1-2 \, \lamalg  \, \frac{\qalg}{1- \qalg}} \, \enorm{u_\ell^{\exact} - u_\ell^{k-1, \jj}} 
		\eqreff{eq1:lem:contraction_perturbed}\le \qsymm \, \enorm{u_\ell^\exact - u_\ell^{k-1, \jj}}.
	\end{align*}
	This concludes the proof of~\eqref{eq2:lem:contraction_perturbed}.
	
Now suppose that $k = \kk[\ell]$. By the met algebraic stopping criterion in Algorithm~\ref{algorithm}(i.b.II) followed by the met stopping criterion of the Zarantonello iteration in Algorithm~\ref{algorithm}(i.d), we obtain that
	\begin{equation*}
		\enorm{u_\ell^{\kk, \jj} - u_\ell^{\kk, \jj-1}} 
		\stackrel{\rm(i.b.II)}\le 
		\lamalg \, \bigl[\lamsym \, \eta_{\ell}(u_\ell^{\kk, \jj}) + \enorm{u_\ell^{\kk, \jj} - u_\ell^{\kk-1, \jj}}\bigr] 
		\stackrel{\rm(i.d)}\le 2 \, \lamalg \, \lamsym \, \eta_{\ell}(u_\ell^{\kk, \jj}).
	\end{equation*}
	Together with the contraction~\eqref{eq:algebra:contraction} of the algebraic solver, this yields that
%	Combining the last two estimates yields
	\begin{equation}\label{eq:step1:lemma}
		\enorm{u_\ell^{\kk, \exact} - u_\ell^{\kk, \jj}} 
				\eqreff{eq:algebra:contraction}\le 
		\frac{\qalg}{1-\qalg} \, \enorm{u_\ell^{\kk, \jj} - u_\ell^{\kk, \jj-1}}
		\le 
		\frac{2 \, \qalg}{1-\qalg} \, \lamalg \, \lamsym \, \eta_{\ell}(u_\ell^{\kk, \jj}).
	\end{equation}
By contraction~\eqref{eq:zarantonello:unperturbed} of the unperturbed Zarantonello iteration, we obtain that
\begin{align*}
 \enorm{u_\ell^\exact - u_\ell^{\kk,\jj}}
 &\le
 \enorm{u_\ell^\exact - u_\ell^{\kk,\star}}
 + \enorm{u_\ell^{\kk,\star} - u_\ell^{\kk,\jj}}
 \\& 
 \eqreff*{eq:zarantonello:unperturbed}\le
 \qsym \, \enorm{u_\ell^\exact - u_\ell^{\kk-1,\jj}} + \frac{2 \, \qalg}{1-\qalg} \, \lamalg \, \lamsym \, \eta_{\ell}(u_\ell^{\kk, \jj}).
\end{align*} 
This concludes also the proof of~\eqref{eq2:lem:contraction_perturbed:kk}.
\end{proof}

An important consequence of the contraction~\eqref{eq2:lem:contraction_perturbed} of the perturbed Zarantonello iteration is that $\kk[\elll] = \infty$ implies that the exact solution is already discrete $u^\exact = u_\elll^\exact \in \XX_\elll$.

\begin{lemma}\label{lemma:new:ghps}
	Suppose that the estimator satisfies stability~\eqref{axiom:stability} and reliability~\eqref{axiom:reliability},
	and that the perturbed Zarantonello iteration is contractive~\eqref{eq2:lem:contraction_perturbed}. Then, $\elll < \infty$ implies that $\kk[\elll] = \infty$ as well as $u^\exact = u_\elll^\exact$ with $\eta_\elll(u_\elll^\exact) = 0$.
\end{lemma}

\begin{proof}
	Since $\jj[\elll,\kk] < \infty$ by virtue of Lemma~\ref{lemma:termination-j}, it follows for $\underline{\ell}<\infty$ that $\kk[\elll] = \infty$ and hence by the not met stopping criterion in Algorithm~\ref{algorithm}(i.d) that
	\begin{align*}
		\eta_\elll(u_\elll^{k,\jj}) < \lamsym^{-1} \, \enorm{u_\elll^{k,\jj} - u_\elll^{k-1, \jj}}
		\quad \text{for all } k \in \N.
	\end{align*}
	Since the perturbed Zarantonello iteration is convergent (see Lemma~\ref{lem:contraction_perturbed}) with limit $u_\elll^\exact$ (and thus $(u_\elll^{k,\jj})_{k \in \N_0}$ is a Cauchy sequence), we infer that
	\begin{align*}
		\eta_\elll(u_\elll^\exact)
		\eqreff{axiom:stability}\le \eta_\elll(u_\elll^{k,\jj}) + \Cstab \, \enorm{u_\elll^\exact - u_\elll^{k,\jj}}
		\xrightarrow{k \to \infty} 0.
	\end{align*}
	This proves $\eta_\elll(u_\elll^\exact) = 0$, whence with reliability~\eqref{axiom:reliability}, we conclude $u_\elll^\exact = u^\exact$.
\end{proof}

%%%%%%%%%%%%%%%%%%%%%%%%%%%%%%%%%%%%%%%%%%%%%%%%%%%%%%%%%%%%%%%%
\subsection{A~priori convergence}
%%%%%%%%%%%%%%%%%%%%%%%%%%%%%%%%%%%%%%%%%%%%%%%%%%%%%%%%%%%%%%%%

For general second-order linear elliptic PDEs, an \textsl{a~priori} convergence result 
%(for the exact Galerkin solutions) 
is required to ensure that there holds a quasi-Pythagorean estimate; see Lemma~\ref{lemma:quasi-pythagoras} below.

\begin{lemma}[a~priori convergence]\label{prop:plain-convergence}
With $\elll \in \N_0 \cup \{\infty\}$ from~\eqref{eq:def:final-indices}, define the discrete limit space $\XX_\infty \coloneqq {\rm closure}\bigl( \bigcup_{\ell=0}^{\elll} \XX_\ell \bigr)$. Then, there exists $u_\infty^\exact \in \XX_\infty$ such that
\begin{align}\label{eq:weakform:infty}
		b(u_\infty^\exact, v_\infty) = F(v_\infty)
		\quad \text{for all } v_\infty \in \XX_\infty,
\end{align}
and it holds that 
\begin{align}\label{eq:prop:plain-convergence}
 \norm{u_\infty^\exact - u_\ell^\exact} \to 0 \quad\text{as } \ell \to \elll.
\end{align}
In particular, this implies $u_\elll^\exact = u_\infty^\exact$ if $\elll < \infty$.
Moreover, with $\Ccea = L/\alpha$ from~\eqref{eq:intro:cea}, there holds the C\'ea-type estimate
\begin{align}\label{eq:cea:infty}
	\enorm{u_\infty^\exact - u_\ell^\exact}
	\le
	\Ccea \, \min_{v_\ell \in \XX_\ell} \enorm{u_\infty^\exact - v_\ell}
%	\quad \text{with} \quad \Ccea = L/\alpha
	\quad \text{for all } \ell \in \N_0 \text{ with } \ell \le \elll.
\end{align}
Moreover, reliability~\eqref{axiom:reliability} implies that
%as well as 
\begin{align}\label{axiom:reliability:infty}
 \enorm{u_\infty^\exact - u_\ell^\exact}
 \le (\Ccea + 1) \, \Crel \, \eta_\ell(u_\ell^\exact) 
 \quad \text{for all } \ell \in \N_0 \text{ with } \ell \le \elll.
\end{align}%
%and
%\begin{align}\label{eq:reliability:infty}
% \enorm{u_\infty^\exact - u_\ell^{\kk,\jj}}
% \le (\Ccea + 1) \Crel' \, \eta_\ell(u_\ell^{\kk,\jj}) 
%\end{align}
%with the constant $\Crel > 0$ from.}
%with the constants $\Crel, \Crel' > 0$ from~\eqref{axiom:reliability} and Proposition~\ref{proposition:reliability}, respectively.}
%
% Suppose that the perturbed Zarantonello iteration is contractive~\eqref{eq2:lem:contraction_perturbed} and that the estimator satisfies \eqref{axiom:stability}--\eqref{axiom:reliability}. Then, it follows that
%	\begin{align}\label{eq:prop:plain-convergence}
%		\begin{split}
%		\enorm{u^\exact - u_\elll^\exact}
%		+ \enorm{u^\exact - u_\elll^{k,\jj}}
%		+ \eta_\ell(u_\elll^{k,\jj}) &\xrightarrow{k \to \kk} 0 \quad \text{for } \elll < \infty, 
%		\\
%		\enorm{u_\infty^\exact - u_\ell^\exact}
%		&\xrightarrow{\ell \to \infty} 0 \quad \text{for } \elll = \infty.
%		\end{split}
%	\end{align}
\end{lemma}

\begin{proof}
Existence and uniqueness of $u_\infty^\exact$ follow from the Lax--Milgram lemma.
		Since $u_\ell^\exact \in \XX_\ell \subseteq \XX_\infty$ is a Galerkin approximation of $u_\infty^\exact$, the C\'ea lemma~\eqref{eq:intro:cea} holds with $u^\exact$ being replaced by $u_\infty^\exact$, and the definition of $\XX_\infty$ proves that
	\begin{align*}%\label{eq:prop:plaincv-disclim}
		\enorm{u_\infty^\exact - u_\ell^\exact}
%		\ \eqreff*{eq:intro:cea}{\le}\ 
		\le \Ccea \, \min_{v_\ell \in \XX_\ell} \, \enorm{u_\infty^\exact - v_\ell}
		\xrightarrow{\ell \to \elll} 0.
	\end{align*}
Reliability~\eqref{axiom:reliability:infty} follows from the triangle inequality, nestedness of spaces $\XX_\ell \subseteq \XX_\infty$, and the C\'ea lemma~\eqref{eq:intro:cea}, since
\begin{align*}
 \enorm{u_\infty^\exact - u_\ell^\exact}
 \le \enorm{u^\exact - u_\infty^\exact} + \enorm{u^\exact - u_\ell^\exact}
 \eqreff{eq:intro:cea}\le (\Ccea \!+\! 1) \, \enorm{u^\exact - u_\ell^\exact}
 \eqreff{axiom:reliability}\le (\Ccea \!+\! 1) \, \Crel \, \eta_\ell(u_\ell^\exact).
\end{align*}%
	This concludes the proof.
\end{proof}

%%%%%%%%%%%%%%%%%%%%%%%%%%%%%%%%%%%%%%%%%%%%%%%%%%%%%%%%%%%%%%%%
\subsection{Quasi-Pythagorean estimate}
%%%%%%%%%%%%%%%%%%%%%%%%%%%%%%%%%%%%%%%%%%%%%%%%%%%%%%%%%%%%%%%%

While symmetric PDEs satisfy a Pythagorean identity in the energy norm (with $\eps = 0$ and $\ell_0=0$ in~\eqref{eq:pythagoras} below), the situation is more involved for nonsymmetric PDEs. The following result generalizes~\cite[Lemma~18]{bhp2017} by considering general $v_\ell \in \XX_\ell$ and by additionally proving the lower bound in~\eqref{eq:pythagoras}. Moreover, it is given here in terms of the \textsl{a~priori} limit $u_\infty^\exact$.
Although the proof follows essentially that of~\cite{bhp2017}, we include it for the sake of completeness.

\begin{lemma}[quasi-Pythagorean estimate]\label{lemma:quasi-pythagoras}
%	Suppose that the estimator satisfies the axioms~\eqref{axiom:stability}--\eqref{axiom:reliability}.
Recall the \textsl{a~priori} limit $u_\infty^\exact \in \XX_\infty$ from Lemma~\ref{prop:plain-convergence} and the compact linear operator $\KK$ from Section~\ref{section:abstract-model-problem}.
%	Suppose that
%	the exact Galerkin approximations satisfy convergence $\enorm{u^\exact - u_\ell^\exact} \to 0$ along the sequence of nested spaces $\XX_\ell \subseteq \XX_{\ell+1}$ as $\ell \to \elll$. 
Then,
	for all $0 < \eps < 1$, there exists an index $\ell_0 \in \N_0$ with $\ell_0 \le \elll$ such that, for all $\ell_0 \le \ell \le \elll$,
	\begin{align}\label{eq:pythagoras}
		\frac{1}{1 + \eps} \, \enorm{u_\infty^\exact- v_\ell}^2
		\le
		\enorm{u_\infty^\exact - u_\ell^\exact}^2 + \enorm{u_\ell^\exact - v_\ell}^2
		\le
		\frac{1}{1 - \eps} \, \enorm{u_\infty^\exact- v_\ell}^2
		\quad \text{for all } v_\ell \in \XX_\ell.
	\end{align}
\end{lemma}

\begin{proof}
	The proof is split into four steps.
	
	\medskip
	{\bf Step~1.} If $\elll < \infty$, Lemma~\ref{prop:plain-convergence} proves that $u_\infty^\exact = u_\elll^\exact$. We choose $\ell_0 = \elll$ and obtain that~\eqref{eq:pythagoras} holds with equality and $\eps = 0$, since $\ell = \elll$ and hence $u_\infty^\exact = u_\ell^\exact$. Consequently, \eqref{eq:pythagoras} holds also for all $0 < \varepsilon < 1$. Therefore, it only remains to prove~\eqref{eq:pythagoras} for $\elll = \infty$.
	
	\medskip
	{\bf Step~2.}
	Suppose $\elll = \infty$.
	Let $\ell \in \N_0$ and $v_\ell \in \XX_\ell$.
	The limit formulation~\eqref{eq:weakform:infty} yields
	\begin{equation}\label{eq1:pythagoras}
		%\begin{split}
		\enorm{u_\infty^\exact - v_\ell}^2 = \enorm{u_\infty^\exact}^2 + \enorm{v_\ell}^2 - 2 \, \Re a(u_\infty^\exact, v_\ell)
		%\\&
		\eqreff{eq:weakform:infty}=
		\enorm{u_\infty^\exact}^2 + \enorm{v_\ell}^2 - 2 \, \Re \bigl[ F(v_\ell) - \dual{\KK u_\infty^\exact}{v_\ell} \bigr].
		%\end{split}
	\end{equation}
	Analogously, from the discrete formulation~\eqref{eq:intro:discrete}
	%of formulation~\eqref{eq:intro:abs_weakform} 
	and the linearity of $\mathcal{K}$, we obtain that
	\begin{align}\label{eq2:pythagoras}
		\begin{split}
			\enorm{u_\ell^\exact - v_\ell}^2
			&= \ 
			\enorm{u_\ell^\exact}^2 + \enorm{v_\ell}^2 - 2 \, \Re a(u_\ell^\exact, v_\ell)
			\\&
			\eqreff*{eq:intro:discrete}= \ 
			\enorm{u_\ell^\exact}^2 + \enorm{v_\ell}^2 - 2 \, \Re \bigl[ F(v_\ell) - \dual{\KK u_\ell^\exact}{v_\ell} \bigr]
			\\&
			= \ 
			\enorm{u_\ell^\exact}^2 + \enorm{v_\ell}^2
			- 2 \Re \bigl[ F(v_\ell) - \dual{\KK u_\infty^\exact}{v_\ell}
			+ \dual{\KK(u_\infty^\exact - u_\ell^\exact)}{v_\ell} \bigr]
		\end{split}
	\end{align}
	as well as
	\begin{align}\label{eq3:pythagoras}
		F(u_\ell^\exact)
		\ \eqreff*{eq:intro:discrete}=\ 
		a(u_\ell^\exact,u_\ell^\exact) + \dual{\KK u_\ell^\exact}{u_\ell^\exact}
		= \enorm{u_\ell^\exact}^2 + \dual{\KK u_\ell^\exact}{u_\ell^\exact}.
	\end{align}
	For $v_\ell = u_\ell^\exact$, we see that
	\begin{align}\label{eq4:pythagoras}
		\begin{split}
			\enorm{u_\infty^\exact - u_\ell^\exact}^2
			\ \ &\eqreff*{eq1:pythagoras}= \ 
			\enorm{u_\infty^\exact}^2 + \enorm{u_\ell^\exact}^2 - 2 \Re \bigl[ F(u_\ell^\exact) - \dual{\KK u_\infty^\exact}{u_\ell^\exact} \bigr]
			\\
			\ &\eqreff*{eq3:pythagoras}= \
			\enorm{u_\infty^\exact}^2 - \enorm{u_\ell^\exact}^2
			+ 2 \Re \dual{\KK(u_\infty^\exact-u_\ell^\exact)}{u_\ell^\exact}.
		\end{split}
	\end{align}
	Summing~\eqref{eq2:pythagoras} and~\eqref{eq4:pythagoras}, we obtain that
	\begin{align}\label{eq5:pythagoras}
		\begin{split}
			&\enorm{u_\infty^\exact - u_\ell^\exact}^2 + \enorm{u_\ell^\exact - v_\ell}^2
			\\& \quad
			=
			\enorm{u_\infty^\exact}^2 + \enorm{v_\ell}^2
			- 2 \Re \bigl[ F(v_\ell) - \dual{\KK u_\infty^\exact}{v_\ell}
			- \dual{\KK(u_\infty^\exact - u_\ell^\exact)}{u_\ell^\exact - v_\ell} \bigr]
			\\& \quad
			\eqreff*{eq1:pythagoras}=
			\enorm{u_\infty^\exact - v_\ell}^2  +  2 \Re \dual{\KK(u_\infty^\exact - u_\ell^\exact)}{u_\ell^\exact - v_\ell}.
		\end{split}
	\end{align}
	
	\medskip
	{\bf Step~3.} We recall from~\cite[Lemma~17]{bhp2017} that the convergence \eqref{eq:prop:plain-convergence} of Lemma~\ref{prop:plain-convergence} yields that
	\begin{align*}
		e_\ell \coloneqq \begin{cases}
			\displaystyle\frac{u_\infty^\exact - u_\ell^\exact}{\enorm{u_\infty^\exact - u_\ell^\exact}} & \text{if } u_\infty^\exact \neq u_\ell^\exact, \\
			0 & \text{otherwise}
		\end{cases}
	\end{align*}
	defines a weakly convergent sequence in $\XX_\infty$ with $e_\ell \rightharpoonup 0$ as $\ell \to \infty$. We recall that compact operators turn weak convergence into norm convergence. With the operator norm $\enorm{\phi}' \coloneqq \sup\limits_{v \in \XX_\infty \backslash \{0\}} \abs{\phi(v)}/\enorm{v}$ of $\phi \in \XX_\infty'$, it thus follows that
	\begin{align*}
		\abs{\dual{\KK ( u_\infty^\exact - u_\ell^\exact )}{u_\ell^\exact - v_\ell}}
		\le \enorm{\KK e_\ell}' \, \enorm{u_\infty^\exact - u_\ell^\exact} \, \enorm{u_\ell^\exact - v_\ell}
		\quad \text{ and } \quad
		\enorm{\KK e_\ell}' \xrightarrow{\ell \to \infty} 0.
	\end{align*}
	Given $\eps > 0$, this provides an index $\ell_0 \in \N$ such that $\enorm{\KK e_\ell}' \le \eps$ for all $\ell \ge \ell_0$ and hence
	\begin{align}\label{eq6:pythagoras}
		\begin{split}
			2 \, \abs{\Re \dual{\KK ( u_\infty^\exact - u_\ell^\exact )}{u_\ell^\exact - v_\ell}}
			&\le 2 \eps \, \enorm{u_\infty^\exact - u_\ell^\exact} \, \enorm{u_\ell^\exact - v_\ell}
			\\&
			\le \eps \bigl[ \enorm{u_\infty^\exact - u_\ell^\exact}^2 + \enorm{u_\ell^\exact - v_\ell}^2 \bigr].
		\end{split}
	\end{align}
	
	\medskip
	{\bf Step~4.} Rearranging the identity~\eqref{eq5:pythagoras} and estimating the compact perturbation via~\eqref{eq6:pythagoras}, we obtain that
	\begin{align*}
		\enorm{u_\infty^\exact - v_\ell}^2
		\ &\eqreff*{eq5:pythagoras}= \
		\enorm{u_\infty^\exact - u_\ell^\exact}^2 + \enorm{u_\ell^\exact - v_\ell}^2
		- 2 \Re \dual{\KK ( u_\infty^\exact - u_\ell^\exact )}{u_\ell^\exact - v_\ell}
		\\
		&\eqreff*{eq6:pythagoras}\le \ 
		(1+\eps) \bigl[ \enorm{u_\infty^\exact - u_\ell^\exact}^2 + \enorm{u_\ell^\exact - v_\ell}^2 \bigr].
	\end{align*}
	This proves the lower estimate in~\eqref{eq:pythagoras}, and the upper estimate is proved analogously.
\end{proof}

%%%%%%%%%%%%%%%%%%%%%%%%%%%%%%%%%%%%%%%%%%%%%%%%%%%%%%%%%%%%%%%%
\subsection{Auxiliary contraction estimates}
%%%%%%%%%%%%%%%%%%%%%%%%%%%%%%%%%%%%%%%%%%%%%%%%%%%%%%%%%%%%%%%%
The following lemma extends \cite[Lemma~10]{ghps2021} to the present setting with a quasi-Pythagorean estimate.
\begin{lemma}[combined discretization-symmetrization error]\label{lemma:ghps}
Recall the \textsl{a~priori} limit $u_\infty^\exact \in \XX_\infty$ from Lemma~\ref{prop:plain-convergence}.
%	Suppose that
%	%Algorithm~\ref{algorithm} leads to $\#\QQ = \infty$, that
%	the perturbed Zarantonello iteration satisfies contraction~\eqref{eq2:lem:contraction_perturbed} and 
Suppose that the estimator satisfies \eqref{axiom:stability}--\eqref{axiom:reliability}.
Let $\lamalg^\exact > 0$ and $0 < \qsymm < 1$ be as in Theorem~\ref{theorem:aisfem:linear-convergence}.
Let $0 < \theta \le 1$ and $0 < \lamsym \le 1$.
Then, there exists $0 < \lamalg' \le \lamalg^\exact$ such that for all $0 < \lamalg \le \lamalg'$ the following holds:
%	Then, for all $0 < \theta \le 1$ and $\lamsym > 0$,
There exists an index $\ell_0 \in \N_0$ with $\ell_0 \le \elll$ and scalars $\nu > 0$ and $0 < \qlin < 1$ such that
	\begin{align}\label{eq1:lemma:ghps}
		\Lambda_\ell^k
		\coloneqq
		\bigl[ \, \enorm{u_\infty^\exact - u_\ell^{k, \jj}}^2 + \nu \, \eta_\ell(u_\ell^{k, \jj})^2 \, \bigr]^{1/2}
		\quad \text{for all } (\ell,k,\jj) \in \QQ
	\end{align}
	satisfies
	\begin{subequations}\label{eq4:lemma:ghps}
		\begin{align}
			\Lambda_\ell^{k+1} &\le \qlin \, \Lambda_\ell^k \hphantom{\Lambda_\ell^{\kk-1}} \quad \text{for all } (\ell, k+1, \jj) \in \QQ \text{ with }\ell \ge \ell_0 \text{ and } k + 1 < \kk[\ell], \label{eq2:lemma:ghps}
			\\
			\Lambda_{\ell+1}^{0} &\le \qlin \, \Lambda_\ell^{\kk-1} \hphantom{\Lambda_\ell^k} \quad \text{for all } (\ell+1, 0, 0) \in \QQ \text{ with } \ell \ge \ell_0. \label{eq3:lemma:ghps}
		\end{align}
	\end{subequations}
\end{lemma}

\begin{proof}
	Let $0 < \varepsilon < 1$ as well as $\nu, \omega > 0$ be free parameters to be fixed below. The proof consists of seven steps, where most of the work is necessary to prove~\eqref{eq3:lemma:ghps}.
	
	\medskip
	\textbf{Step 1.} Lemma~\ref{lemma:quasi-pythagoras} provides an index $\ell_0 = \ell_0(\varepsilon)$ such that for all $\ell_0 \le \ell \le \elll$ the quasi-Pythagorean estimate~\eqref{eq:pythagoras} holds true. For $(\ell, k+1, \jj) \in \QQ$ with $\ell_0 \le \ell$, we get that
	\begin{align}\label{eq1:combined-contraction}
		\begin{split}
			(\Lambda_\ell^{k+1})^2 &= \enorm{u_\infty^\exact - u_\ell^{k+1, \jj}}^2 + \nu \, \eta_\ell(u_\ell^{k+1, \jj})^2 \\
			&\eqreff*{eq:pythagoras}{\le} (1+\varepsilon) \, \enorm{u_\infty^\exact - u_\ell^\exact}^2 + (1+\varepsilon) \, \enorm{u_\ell^\exact - u_\ell^{k+1, \jj}}^2 + \nu \, \eta_\ell(u_\ell^{k+1, \jj})^2.
		\end{split}
	\end{align}
	Analogously, for $(\ell+1,0, 0) \in \QQ$ with $\ell \ge \ell_0$, nested iteration $u_{\ell+1}^{0, 0} = u_{\ell+1}^{0, \jj} = u_\ell^{\kk, \jj}$ shows that
	\begin{align}\label{eq4:combined-contraction}
		\begin{split}
			(\Lambda_{\ell+1}^0)^2 &= \enorm{u_\infty^\exact - u_\ell^{\kk, \jj}}^2 + \nu \, \eta_{\ell+1}(u_\ell^{\kk, \jj})^2 \\
			&\eqreff*{eq:pythagoras}{\le} (1+\varepsilon) \, \enorm{u_\infty^\exact - u_\ell^\exact}^2 + (1+\varepsilon) \, \enorm{u_\ell^\exact - u_{\ell}^{\kk, \jj}}^2 + \nu \, \eta_{\ell+1}(u_\ell^{\kk, \jj})^2.
		\end{split}
	\end{align}
	
	\medskip
	\textbf{Step 2.} Define $C_1 \coloneqq 6\, (1+\Ccea)^2 \, \Crel^{2}$ and $C_2 \coloneqq 6 \, (1+\Ccea)^2 \, \Crel^2 \, \Cstab^2$.
	Then, stability~\eqref{axiom:stability} and reliability~\eqref{axiom:reliability:infty} prove that, for all $v_\ell \in \XX_\ell$,
\begin{align}\label{eq2:combined-contraction}
		\begin{split}
			3  \, \enorm{u_\infty^\exact - u_\ell^\exact}^2
			%&\eqreff*{eq:intro:cea}\le
			%3  \, (1+\Ccea)^2 \enorm{u^\exact - u_\ell^\exact}^2
			%\\
			&\eqreff*{axiom:reliability:infty}\le
			3 \, (1+\Ccea)^2 \, \Crel^2 \, \eta_\ell(u_\ell^\exact)^2
			\\
			&\eqreff*{axiom:stability}\le
			6 \, (1+\Ccea)^2 \, \Crel^2 \, \eta_\ell(v_\ell)^2
			+ 6 \, (1+\Ccea)^2 \, \Crel^2 \, \Cstab^2 \, \enorm{u_\ell^\exact - v_\ell}^2
			\\&
			=  C_1 \, \eta_\ell(v_\ell)^2 + \, C_2 \, \enorm{u_\ell^\exact - v_\ell}^2.
		\end{split}
	\end{align}

	\medskip
	\textbf{Step 3.} For $(\ell, k+1, \jj) \in \QQ$ with $\ell \ge \ell_0$ and $k+1 < \kk[\ell]$, 
%	For $(\ell, k+1, 0) \in \QQ$, 
contraction~\eqref{eq2:lem:contraction_perturbed} of the perturbed Zarantonello iteration proves that
	\begin{align*}
		\enorm{u_\ell^{k+1, \jj} - u_\ell^{k, \jj}} \le \enorm{u_\ell^\exact - u_\ell^{k+1, \jj}} + \enorm{u_\ell^\exact - u_\ell^{k, \jj}} 
		\eqreff{eq2:lem:contraction_perturbed}\le (1+\qsymm) \,\enorm{u_\ell^\exact - u_\ell^{k, \jj}}.
	\end{align*}
	Define $C_3 \coloneqq (1+\qsymm)^2$.
	Using this with the not met stopping criterion in Algorithm~\ref{algorithm}(i.d) for $(\ell, k+1, \jj) \in \QQ$ with $k+1 < \kk[\ell]$ shows that
	\begin{align}\label{eq3:combined-contraction}
		\eta_\ell(u_\ell^{k+1, \jj})^2 \stackrel{\rm(i.d)}< \lamsym^{-2} \enorm{u_\ell^{k+1, \jj}-u_\ell^{k, \jj}}^2
		%\le \lamsym^{-2} (1+\qsymm)^2 \enorm{u_\ell^\exact - u_\ell^{k, \jj}}^2 \eqqcolon
		\le C_3 \lamsym^{-2} \, \enorm{u_\ell^\exact - u_\ell^{k, \jj}}^2.
	\end{align}
	In this case, we are thus led to
	\begin{align*}
		&(\Lambda_\ell^{k+1})^2
		\eqreff*{eq1:combined-contraction}\le
		(1-2 \varepsilon) \, \enorm{u_\infty^\exact - u_\ell^\exact}^2 + 3 \varepsilon \, \enorm{u_\infty^\exact - u_\ell^\exact}^2 + (1+\varepsilon) \, \enorm{u_\ell^\exact - u_\ell^{k+1, \jj}}^2 + \nu \, \eta_\ell(u_\ell^{k+1, \jj})^2 
		\\& \quad
		\eqreff*{eq2:combined-contraction}\le
		(1-2 \varepsilon) \, \enorm{u_\infty^\exact - u_\ell^\exact}^2 + (\nu + \varepsilon \, C_1) \, \eta_\ell(u_\ell^{k+1, \jj})^2 + (1+\varepsilon(1+C_2)) \, \enorm{u_\ell^\exact - u_\ell^{k+1, \jj}}^2 
		\\& \quad
		\eqreff*{eq2:lem:contraction_perturbed}\le
		(1-2 \varepsilon) \, \enorm{u_\infty^\exact - u_\ell^\exact}^2 + (\nu + \varepsilon \, C_1) \,  \eta_\ell(u_\ell^{k+1, \jj})^2 + (1+\varepsilon(1+C_2)) \, \qsymm^2 \, \enorm{u_\ell^\exact - u_\ell^{k, \jj}}^2 
		\\& \quad
		\eqreff*{eq3:combined-contraction}\le
		(1-2 \varepsilon) \, \enorm{u_\infty^\exact - u_\ell^\exact}^2 + \bigl[(\nu + \varepsilon \, C_1) C_3 \lamsym^{-2} + (1+\varepsilon(1+C_2)) \, \qsymm^2\bigr] \, \enorm{u_\ell^\exact - u_\ell^{k, \jj}}^2.
	\end{align*}
	Provided that
\begin{itemize}
\item[{$[I]$}] %$(\nu + \varepsilon \, C_1) C_3 \lamsym^{-2} + (1+\varepsilon(1+C_2)) \, \qsymm^2 = 
$\qsymm^2 +
		\nu C_3 \lamsym^{-2} + \eps \, [C_1C_3 \lamsym^{-2} + (1+C_2) \, \qsymm^2] \le 1-2\varepsilon,$
\end{itemize}
%	\begin{align*}
%		(\nu + \varepsilon \, C_1) C_3 \lamsym^{-2} + (1+\varepsilon(1+C_2)) \, \qsymm^2 = \qsymm^2 +
%		\nu C_3 \lamsym^{-2} + \eps \, [C_1C_3 \lamsym^{-2} + (1+C_2) \, \qsymm^2] \le 1-2\varepsilon,
%	\end{align*}
	the quasi-Pythagorean estimate~\eqref{eq:pythagoras} proves that
	\begin{align*}
		(\Lambda_\ell^{k+1})^2
		\le
		(1-2\varepsilon) \, \bigl[ \enorm{u_\infty^\exact - u_\ell^\exact}^2 + \enorm{u_\ell^\exact - u_\ell^{k, \jj}}^2 \bigr]
		\eqreff{eq:pythagoras}\le
		\frac{1-2\varepsilon}{1-\varepsilon} \, \enorm{u_\infty^\exact - u_\ell^{k, \jj}}^2
		\le
		\frac{1-2\varepsilon}{1-\varepsilon} \, (\Lambda_\ell^k)^2.
	\end{align*}
	Up to the choice of the parameters $\eps$ and $\nu$, this proves~\eqref{eq2:lemma:ghps} for any $0 < \lamalg \le \lamalg^\star$.

	\medskip
	\textbf{Step 4.} 
%This step concerns estimator reduction via mesh refinement and thus applies only to the case $(\ell+1,0, 0) \in \QQ$. 
For $(\ell+1,0, 0) \in \QQ$, 
stability~\eqref{axiom:stability}, reduction~\eqref{axiom:reduction}, and the D\"orfler marking in Algorithm~\ref{algorithm}(iii) yield that
	\begin{align} \label{eq5:combined-contraction}
	\begin{split}
			&\eta_{\ell+1}(u_\ell^{\kk,\jj})^2
			=
			\eta_{\ell+1}(\TT_{\ell+1} \cap \TT_\ell ; u_\ell^{\kk,\jj})^2
			+ \eta_{\ell+1}(\TT_{\ell+1} \backslash \TT_\ell ; u_\ell^{\kk,\jj})^2 %\notag
			\\& \quad
			\eqreff*{axiom:stability}=\,
			\eta_\ell(\TT_{\ell+1} \cap \TT_\ell ; u_\ell^{\kk,\jj})^2
			+ \eta_{\ell+1}(\TT_{\ell+1} \backslash \TT_\ell ; u_\ell^{\kk,\jj})^2 %\notag
			\\& \quad
			\eqreff*{axiom:reduction}\le \,
			\eta_\ell(\TT_{\ell+1} \cap \TT_\ell ; u_\ell^{\kk,\jj})^2
			+ \qred^{2} \, \eta_\ell(\TT_\ell \backslash \TT_{\ell+1} ; u_\ell^{\kk,\jj})^2 %\notag
			\\& \quad
			= \,
			\eta_\ell(u_\ell^{\kk,\jj})^2
			- (1 - \qred^{2} ) \, \eta_\ell(\TT_\ell \backslash \TT_{\ell+1} ; u_\ell^{\kk,\jj})^2 %\notag
			\\& \quad
			\le
			\eta_\ell(u_\ell^{\kk,\jj})^2
			- (1 - \qred^{2} ) \, \eta_\ell(\MM_\ell ; u_\ell^{\kk,\jj})^2 %\notag
			%\\&
			\le
			[1 - (1-\qred^{2} ) \, \theta] \, \eta_\ell(u_\ell^{\kk,\jj})^2
			\eqqcolon q_\theta \, \eta_\ell(u_\ell^{\kk, \jj})^2,
			\end{split}
			\end{align}
			where \(0 < q_\theta < 1\) by definition.
%	\begin{align}
%		\begin{split}
%			\eta_{\ell+1}(u_{\ell}^{\kk, \jj})^2
%			\eqreff{eq:step3-estimator}\le [1-(1-\qred^2)\theta] \, \eta_\ell(u_\ell^{\kk, \jj})^2 
%			\quad \text{with } 0 < q_\theta < 1.
%		\end{split}
%	\end{align}
	
	\medskip
	\textbf{Step 5.} 
Let $(\ell+1, 0, 0) \in \QQ$. With stability~\eqref{axiom:stability}, we infer from~\eqref{eq2:lem:contraction_perturbed:kk} that
\begin{align*}
 \eta_\ell(u_\ell^{\kk, \jj})
 &\eqreff{axiom:stability}\le \eta_\ell(u_\ell^{\kk-1, \jj})
 + \Cstab \, \enorm{u_\ell^{\kk, \jj} - u_\ell^{\kk-1, \jj}}
 \\&
 \eqreff{eq2:lem:contraction_perturbed:kk}\le \eta_\ell(u_\ell^{\kk-1, \jj}) + \Cstab (1+\qsym) \, \enorm{u_\ell^\exact - u_\ell^{\kk-1,\jj}}
 + \frac{2 \, \qalg}{1-\qalg} \, \Cstab \, \lamalg \lamsym \, \eta_{\ell}(u_\ell^{\kk, \jj}).
\end{align*}
For sufficiently small $0 < \lamalg \le \lamalg^\star$ with, e.g., $\frac{2 \, \qalg}{1-\qalg} \, \Cstab \, \lamalg \lamsym \le 1/2$, we thus derive
\begin{align*}
 \eta_\ell(u_\ell^{\kk, \jj})
 \le \frac{1}{1 - \frac{2 \, \qalg}{1-\qalg} \, \Cstab \, \lamalg \lamsym} \, \eta_\ell(u_\ell^{\kk-1, \jj})
 + \frac{\Cstab (1+\qsym)}{1 - \frac{2 \, \qalg}{1-\qalg} \, \Cstab \, \lamalg \lamsym} \, \enorm{u_\ell^\exact - u_\ell^{\kk-1,\jj}}.
\end{align*}%
Define $C_4 \coloneqq \Cstab^2 (1+\qsym)^2$ and $C(\lambda) \coloneqq \bigl[ 1 - \frac{2 \, \qalg}{1-\qalg} \, \Cstab \, \lamalg \lamsym \bigr]^{-2}$  with $\lambda = \lamalg \lamsym$, where we already note that $C(\lambda) \to 1$ is (strictly) monotonically decreasing as $\lambda \to 0$.
Stability~\eqref{axiom:stability} and the Young inequality in the form $(a+b)^2 \le (1+\omega) a^2 + (1+\omega^{-1}) b^2$ for $a,b \in \R$ and $\omega > 0$ show that
	\begin{align}\label{eq6:combined-contraction}
		\begin{split}
			\eta_\ell(u_\ell^{\kk, \jj})^2 
			%\,&\eqreff*{axiom:stability}
			%\le \, (1+\omega) \,C(\lambda) \,\eta_\ell(u_\ell^{\kk-1, \jj})^2 + (1+\omega^{-1}) \, \Cstab^2 \, \enorm{u_\ell^{\kk, \jj} - u_\ell^{\kk-1, \jj}}^2 \\
			&\le C(\lambda) \, \big[ (1+\omega) \, \, \eta_\ell(u_\ell^{\kk-1, \jj})^2 + (1+\omega^{-1}) \, C_4 \, \enorm{u_\ell^\exact - u_\ell^{\kk-1, \jj}}^2 \big].
		\end{split}
	\end{align}
	
	\medskip
	\textbf{Step 6.} For $(\ell+1,0, 0) \in \QQ$ with $\ell \ge \ell_0$, we have that
	\begin{align*}
		&(\Lambda_{\ell+1}^0)^2
		\ \, \eqreff*{eq4:combined-contraction}\le \
		(1-2\varepsilon) \, \enorm{u_\infty^\exact - u_\ell^\exact}^2 
		+ (1+\varepsilon) \, \enorm{u_\ell^\exact - u_{\ell}^{\kk, \jj}}^2 
		+ 3 \varepsilon \, \enorm{u_\infty^\exact - u_\ell^\exact}^2 
		+ \nu \,\eta_{\ell+1}(u_\ell^{\kk, \jj})^2 
		\\&
		\eqreff*{eq5:combined-contraction}\le \
		(1-2\varepsilon) \, \enorm{u_\infty^\exact - u_\ell^\exact}^2 
		+ (1+\varepsilon) \, \enorm{u_\ell^\exact - u_{\ell}^{\kk, \jj}}^2 
		+ 3 \varepsilon \, \enorm{u_\infty^\exact - u_\ell^\exact}^2 
		+ \nu q_\theta \,\eta_\ell(u_\ell^{\kk, \jj})^2 
		\\&
		\eqreff*{eq2:combined-contraction}\le \
		(1-2\varepsilon) \, \enorm{u_\infty^\exact - u_\ell^\exact}^2 
		+ (1+\varepsilon) \, \enorm{u_\ell^\exact - u_{\ell}^{\kk, \jj}}^2 
		+ \varepsilon \, C_2 \, \enorm{u_\ell^\exact - u_\ell^{\kk-1, \jj}}^2
		\\& \quad
		+ \varepsilon \, C_1 \, \eta_\ell(u_\ell^{\kk-1, \jj})^2 
		+ \nu q_\theta \,\eta_\ell(u_\ell^{\kk, \jj})^2 
		\\&
		\eqreff*{eq2:lem:contraction_perturbed:kk}\le \
		(1-2\varepsilon) \, \enorm{u_\infty^\exact - u_\ell^\exact}^2 
		+ \bigl[\varepsilon \, C_2 + (1+\eps)^2 \qsym^2\bigr] \, \enorm{u_\ell^\exact - u_\ell^{\kk-1, \jj}}^2
		\\& \quad
		+ \varepsilon \, C_1 \, \eta_\ell(u_\ell^{\kk-1, \jj})^2 
		+ \Bigl[ q_\theta + \nu^{-1} \, (1+\eps) (1+\eps^{-1}) \, \Bigl(\frac{2 \, \qalg}{1-\qalg}\Bigr)^2 \lamalg^2 \lamsym^2 \Bigr] \, \nu \,\eta_\ell(u_\ell^{\kk, \jj})^2.
		\end{align*}
		%%%
		%%%
		With $C_{\eps} \coloneqq (1+\eps) (1+\eps^{-1}) \, \bigl(\frac{2 \, \qalg}{1-\qalg}\bigr)^2$, we get
		\begin{align*}
		&(\Lambda_{\ell+1}^0)^2
		\le 
(1-2\varepsilon) \, \enorm{u_\infty^\exact - u_\ell^\exact}^2 
		+ \bigl[\varepsilon \, C_2 + (1+\eps)^2 \qsym^2\bigr] \, \enorm{u_\ell^\exact - u_\ell^{\kk-1, \jj}}^2
		\\& \qquad\qquad
		+ \varepsilon \, C_1 \, \eta_\ell(u_\ell^{\kk-1, \jj})^2 
		+ \bigl[ q_\theta + C_{\eps} \nu^{-1} \, \lamalg^2 \lamsym^2 \bigr] \, \nu \,\eta_\ell(u_\ell^{\kk, \jj})^2.
		\\&
		\eqreff*{eq6:combined-contraction}\le \
(1-2\varepsilon) \, \enorm{u_\infty^\exact - u_\ell^\exact}^2 
		+ \bigl[\varepsilon \, C_2 + (1+\eps)^2 \qsym^2\bigr] \, \enorm{u_\ell^\exact - u_\ell^{\kk-1, \jj}}^2
		\\& \qquad
		+ \varepsilon \, C_1 \, \eta_\ell(u_\ell^{\kk-1, \jj})^2 
		\\& \qquad
		+ \bigl[ q_\theta + C_{\eps} \nu^{-1} \, \lamalg^2 \lamsym^2 \bigr] \, \nu \, \rrevision{C(\lambda)} \, \big[(1+\omega) \, \eta_\ell(u_\ell^{\kk-1, \jj})^2 + (1+\omega^{-1}) \, C_4 \, \enorm{u_\ell^\exact - u_\ell^{\kk-1, \jj}}^2 \big]   
		\\&
		= \ 		
		(1-2\varepsilon) \, \enorm{u_\infty^\exact - u_\ell^\exact}^2 
		\\& \quad
		+ \Bigl(\varepsilon \, C_2 + (1+\eps)^2 \qsym^2 + (1+\omega^{-1}) \, C_4 \, \rrevision{C(\lambda)} \, \bigl[ q_\theta + C_{\eps} \nu^{-1} \, \lamalg^2 \lamsym^2 \bigr] \, \nu \Bigr) \, \enorm{u_\ell^\exact - u_\ell^{\kk-1, \jj}}^2
		\\& \quad
		+ \Bigl((1+\omega) C(\lambda)  \, \bigl[ q_\theta + C_{\eps} \nu^{-1} \, \lamalg^2 \lamsym^2 \bigr] + \varepsilon \, C_1 \nu^{-1} \Bigr) \, \nu \, \eta_\ell(u_\ell^{\kk-1, \jj})^2.
%
%\\&\color{cyan}
%		(1-2\varepsilon) \, \enorm{u_\infty^\exact - u_\ell^\exact}^2 
%		+ [\varepsilon \, C_2 + (1+\varepsilon) \, \qsymm^2 
%		+ C_4 q_\theta \nu (1+\omega^{-1})] \, \enorm{u_\ell^\exact - u_\ell^{\kk-1, \jj}}^2
%		\\&\quad
%		+ \bigl[\eps \, C_1 \nu^{-1} + q_\theta(1+\omega)\bigr] \nu \, \eta_{\ell}(u_\ell^{\kk-1, \jj})^2.
	\end{align*}
	%Recall that $0 < \lamsym \le 1$.
	Provided that
	\begin{itemize}
	\item[{$[I\!I]$}] $(1+\eps)^2 \qsym^2 + \rrevision{C(\lambda)} \, (1+\omega^{-1}) \, C_4 \, \bigl[ q_\theta + C_{\eps} \nu^{-1} \,  \lamalg^2\lamsym^2 \bigr] \, \nu  + \varepsilon \, C_2 \le 1-2\eps$,
	\item[{$[I\!I\!I]$}] $C(\lambda)  \, \Bigl[ (1+\omega) q_\theta  
		+ (1+\omega) C_{\eps}\nu^{-1} \,  \lamalg^2\lamsym^2 \Bigr]
		+ \varepsilon C_1 \nu^{-1} \le 1 - 2\eps$,
	\end{itemize}	
%	\begin{align*}
%		\eps \,  C_1 \nu^{-1} + q_\theta(1+\omega)
%		%=(1+\omega)q_\theta + \eps C_1\nu^{-1} 
%		\le 1-2\varepsilon
%	\end{align*}
%	\todo{and
%	\begin{align*}
%		\varepsilon \, C_2 + (1+\varepsilon)\qsymm^2 + C_4 q_\theta \nu (1+\omega^{-1})
%		= \qsymm^2 + \nu C_4 q_\theta (1+\omega^{-1})  + \eps(C_2 + \qsymm^2) \le 1-2\varepsilon,
%	\end{align*}}%
	the quasi-Pythagorean estimate~\eqref{eq:pythagoras} shows that
	\begin{align*}
		(\Lambda_{\ell+1}^0)^2
		&\le
		(1-2\varepsilon) \bigl[\enorm{u_\infty^\exact- u_\ell^\exact}^2 + \enorm{u_\ell^\exact - u_\ell^{\kk-1, \jj}}^2 + \nu\, \eta_\ell(u_\ell^{\kk-1, \jj})^2\bigr] \\
		\ &\eqreff*{eq:pythagoras}\le \
		\frac{1-2\varepsilon}{1-\varepsilon} \, \enorm{u_\infty^\exact - u_\ell^{\kk-1, \jj}}^2 + (1-2\varepsilon) \nu \, \eta_\ell(u_\ell^{\kk-1, \jj})^2
		\le
		\frac{1-2\varepsilon}{1-\varepsilon} \, (\Lambda_\ell^{\kk-1})^2.
	\end{align*}
	This proves~\eqref{eq3:lemma:ghps} up to the choice of the parameters $\omega, \nu$, and $\varepsilon$ in the following step.
	
	\medskip
	\textbf{Step 7.} A suitable choice of the parameters $\omega, \nu$, and $\varepsilon$ can be obtained as follows:
	\begin{itemize}
		\item first, we choose $\omega$ such that $(1+\omega)q_\theta < 1$;

		\item second, we choose $\nu$ such that $\qsymm^2 + \nu C_3 \lamsym^{-2}< 1$ and $\qsym^2 + (1+\omega^{-1}) C_4 \nu \le 1$;
		
		\item third, we choose $\varepsilon > 0$ sufficiently small such that 
		\begin{itemize}
		\item[$\bullet$] $\qsymm^2 + \nu C_3 \lamsym^{-2} + \eps \, [C_1C_3 \lamsym^{-2} + (1+C_2) \, \qsymm^2] \le 1-2\varepsilon$,
		\item[$\bullet$] $\highlighted{(1+\eps)^2 \qsym^2} + (1+\omega^{-1})C_4q_\theta\nu + \varepsilon \, C_2 < 1-2\eps$,
		\item[$\bullet$] $(1+\omega) q_\theta + \varepsilon C_1 \nu^{-1} < 1 - 2\eps$;
		\end{itemize}
		in particular, constraint $[I]$ from Step~3 is satisfied; 
		\item finally, we note that $C(\lambda) \to 1$ monotonically as $\lambda = \lamalg\lamsym\to0$. Hence, we can choose $0 < \lamalg' \le \min \big\{ \lamalg^\star, \frac{1 - \qalg}{4\qalg\Cstab}\, \lamsym^{-1} \big\}$ sufficiently small such that also the constraints $[I\!I]$ and $[I\!I\!I]$ from Step~6 are satisfied for all $0 < \lamalg \le \lamalg'$.
	\end{itemize}%
	This concludes the proof with $\displaystyle\qlin^2 \coloneqq \frac{1-2\varepsilon}{1-\varepsilon} < 1$
	for any $0 < \lamalg \le \lamalg'$.
\end{proof}

%%%%%%%%%%%%%%%%%%%%%%%%%%%%%%%%%%%%%%%%%%%%%%%%%%%%%%%%%%%%%%%%
\subsection{Proof of Theorem~\ref{theorem:aisfem:linear-convergence}} \label{subsection:proof-linear-convergence}
%%%%%%%%%%%%%%%%%%%%%%%%%%%%%%%%%%%%%%%%%%%%%%%%%%%%%%%%%%%%%%%%

The proof is split into five steps. For $(\ell,k,j) \in \QQ$, we consider%Recall the definitions
	\begin{align}\label{eq:***}
	\begin{split}
		\Delta_\ell^{k,j}
		\, &= \,
		\enorm{u_\infty^\exact - u_\ell^{k,j}} + \enorm{u_\ell^{k,\exact} - u_\ell^{k,j}} + \eta_\ell(u_\ell^{k,j}),
		\\
		\Lambda_\ell^k
		\,&\eqreff*{eq1:lemma:ghps}= \,
		\bigl[ \, \enorm{u_\infty^\exact - u_\ell^{k,\jj}}^2 + \nu \, \eta_\ell(u_\ell^{k,\jj})^2 \, \bigr]^{1/2},
		\end{split}
	\end{align}
	where, compared with~\eqref{eq0:theorem:aisfem:linear-convergence}, the quasi-error $\Delta_\ell^{k,j}$ has been redefined. Later, we shall conclude that indeed $u_\infty^\exact = u^\exact$ so that both definitions coincide. 

\medskip

\noindent
{\bf Step~1.} In the first step, we prove that
\begin{align}\label{eq:step2:proof:convergence}
	\Delta_\ell^{k,j}
	\lesssim \enorm{u_\ell^{k,\exact} - u_\ell^{k,j-1}}
	\quad \text{for all } (\ell,k,j) \in \QQ \text{ with } 1 \le k \le \kk[\ell]
	\text{ and } 1 \le j < \jj[\ell,k].
\end{align}
Together with reliability~\eqref{axiom:reliability:infty} and stability~\eqref{axiom:stability}, the definition of $\Delta_\ell^{k, j}$ shows that
\begin{align*}
	\Delta_\ell^{k, j} \ &\eqreff*{eq:***}=\
	\enorm{u_\infty^\exact - u_\ell^{k, j}} + \enorm{u_\ell^{k, \exact} - u_\ell^{k, j}} + \eta_{\ell}(u_\ell^{k, j}) \\
	&\le \,
	\enorm{u_\infty^\exact - u_\ell^{\exact}} + \enorm{u_\ell^\exact - u_\ell^{k, j}} + \enorm{u_\ell^{k, \exact} - u_\ell^{k, j}} + \eta_{\ell}(u_\ell^{k, j}) \\
	&\eqreff*{axiom:reliability:infty}\le \, (\Ccea+1)\Crel \, \eta_\ell(u_\ell^\exact) + \enorm{u_\ell^\exact - u_\ell^{k, j}} + \enorm{u_\ell^{k, \exact} - u_\ell^{k, j}} + \eta_{\ell}(u_\ell^{k, j}) \\
	&\eqreff*{axiom:stability}\lesssim \, 
	%(1+\Crel) \, 
	\eta_\ell(u_\ell^{k, j}) + 
	%(1 + \Cstab \Crel) \, 
	\enorm{u_\ell^\exact - u_\ell^{k, j}} + \enorm{u_\ell^{k, \exact} - u_\ell^{k, j}}.
\end{align*}
The contraction of the (unperturbed) Zarantonello iteration~\eqref{eq:zarantonello:unperturbed} proves that
\begin{align*}
	\enorm{u_\ell^\exact - u_\ell^{k,j}}
	\le
	\enorm{u_\ell^\exact - u_\ell^{k,\star}} + \enorm{u_\ell^{k,\star} - u_\ell^{k,j}}
	&\eqreff*{eq:zarantonello:unperturbed}\le
	\frac{\qsym}{1-\qsym} \, \enorm{u_\ell^{k,\exact} - u_\ell^{k-1, \jj}} + \enorm{u_\ell^{k,\star} - u_\ell^{k,j}}
	\\&\lesssim
	\enorm{u_\ell^{k,\exact} - u_\ell^{k,j}} + \enorm{u_\ell^{k,j} - u_\ell^{k-1,\jj}}.
\end{align*}
Furthermore, the contraction of the algebraic solver~\eqref{eq:algebra:contraction} proves that
\begin{align*}
	\enorm{u_\ell^{k,\exact} - u_\ell^{k,j}}
	\eqreff*{eq:algebra:contraction}\le
	\frac{\qalg}{1-\qalg} \, \enorm{u_\ell^{k,j} - u_\ell^{k,j-1}}.
\end{align*}
Combining the last three estimates with the not met stopping criterion of the algebraic solver in Algorithm~\ref{algorithm}(i.b.II) for $1 \le j < \jj[\ell,k]$, we conclude that
\begin{align*}
	\Delta_\ell^{k, j}
	\lesssim
	\eta_\ell(u_\ell^{k, j}) + \enorm{u_\ell^{k,j} - u_\ell^{k-1,\jj}} + \enorm{u_\ell^{k,j} - u_\ell^{k,j-1}}
	\stackrel{\rm(i.b.II)}\lesssim
	\enorm{u_\ell^{k,j} - u_\ell^{k,j-1}}.
\end{align*}
Finally, the triangle inequality and the contraction~\eqref{eq:algebra:contraction} imply \eqref{eq:step2:proof:convergence}.

\textbf{Step 2.} Next, we show that
\begin{align}\label{eq:step3:proof:convergence}
	\Delta_{\ell}^{k, \jj} \lesssim \Delta_{\ell}^{k,j}
	\quad \text{for all } (\ell,k,j) \in \QQ,
\end{align}
which is trivial for $j= \jj[\ell, k]$. To deal with $j = \jj[\ell, k] - 1$, note that the definition of $ \Delta_{\ell}^{k, j}$ shows that
\begin{align*}
	\Delta_{\ell}^{k, \jj} \,
	\ &\eqreff*{eq:***}=\ 
	\enorm{u_\infty^\exact - u_{\ell}^{k, \jj}} + \enorm{u_{\ell}^{k, \exact} - u_{\ell}^{k, \jj}} + \eta_\ell(u_{\ell}^{k, \jj})
	\\&
	\le \
	\enorm{u_\infty^\exact - u_{\ell}^{k, \jj-1}} + \enorm{u_{\ell}^{k, \exact} - u_{\ell}^{k, \jj-1}} + 2 \, \enorm{u_{\ell}^{k, \jj} - u_{\ell}^{k, \jj-1}} + \eta_\ell(u_{\ell}^{k, \jj}).
\end{align*}
Stability~\eqref{axiom:stability} and the algebraic solver contraction~\eqref{eq:algebra:contraction} lead us to
\begin{align*}
	2 \, \enorm{u_{\ell}^{k, \jj} - u_{\ell}^{k, \jj-1}} + \eta_\ell(u_{\ell}^{k, \jj})
	&\eqreff*{axiom:stability}\le
	(2+\Cstab) \, \enorm{u_{\ell}^{k, \jj} - u_{\ell}^{k, \jj-1}} +  \eta_{\ell}(u_{\ell}^{k, \jj-1})
	\\&
	\eqreff*{eq:algebra:contraction}\le
	(2+\Cstab) (1+\qalg) \, \enorm{u_{\ell}^{k, \exact} - u_{\ell}^{k, \jj-1}} + \eta_{\ell}(u_{\ell}^{k, \jj-1}).
\end{align*}
Combining the last two estimates verifies~\eqref{eq:step3:proof:convergence} for $j = \jj[\ell, k]-1$, i.e.,
\begin{align}\label{eq:step3a:proof:convergence}
	\Delta_{\ell}^{k, \jj}
	\lesssim
	\enorm{u_\infty^\exact - u_{\ell}^{k, \jj-1}} + \enorm{u_{\ell}^{k, \exact} - u_{\ell}^{k, \jj-1}} + \eta_\ell(u_{\ell}^{k, \jj-1})
	\eqreff{eq:***}=
	\Delta_{\ell}^{k, \jj-1}.
\end{align}
We prove the remaining case $j <  \jj[\ell,k]-1$ by~\eqref{eq:step2:proof:convergence} from Step~1 and the algebraic solver contraction~\eqref{eq:algebra:contraction}, i.e.,
\begin{align*}%\label{eq:deltalast_to_delta}
	\Delta_{\ell}^{k, \jj}
	\eqreff{eq:step3a:proof:convergence}\lesssim \Delta_{\ell}^{k, \jj-1}
	\eqreff{eq:step2:proof:convergence}\lesssim
	\enorm{u_{\ell}^{k, \exact} - u_{\ell}^{k, \jj-2}}
	\eqreff{eq:algebra:contraction}\le
	\qalg^{(\jj[\ell,k]-2)-j} \, \enorm{u_{\ell}^{k, \exact} - u_{\ell}^{k, j}}
	\le
	\Delta_{\ell}^{k, j}.
\end{align*}
This concludes the proof of~\eqref{eq:step3:proof:convergence}.

{\bf Step~3.} In this step, we prove that
\begin{align}\label{eq:step1:proof:convergence}
	\Lambda_\ell^0 \simeq \Delta_\ell^{0,0} = \Delta_\ell^{0,\jj}
	\,\,\, \text{and} \,\,\,
	\Lambda_\ell^k
	\lesssim
	\Delta_\ell^{k, \jj} \,
	\eqreff*{eq:step3:proof:convergence}\lesssim \,
	\Delta_\ell^{k,0}
	\lesssim
	\Lambda_\ell^{k-1}
	\text{ for all } (\ell,k,\jj) \in \QQ \text{ with } k \ge 1.
\end{align}
%To see this, recall that
%\begin{align*}
% \Delta_\ell^{k,\jj}
% \, \eqreff*{eq0:theorem:aisfem:linear-convergence}= \,
% \enorm{u^\exact - u_\ell^{k,\jj}} + \enorm{u_\ell^{k,\exact} - u_\ell^{k,\jj}} + \eta_\ell(u_\ell^{k,\jj}) \text{ and }
% \Lambda_\ell^k
% \,\eqreff*{eq1:lemma:ghps}= \,
% \bigl[ \, \enorm{u^\exact - u_\ell^{k,\jj}}^2 + \nu \, \eta_\ell(u_\ell^{k,\jj})^2 \, \bigr]^{1/2}.
%\end{align*}
Together with $u_\ell^{0,\star} = u_\ell^{0,\jj} = u_\ell^{0,0}$, the definition of $\Lambda_\ell^0$ and $\Delta_\ell^{0,0}$ proves that
$\Lambda_\ell^0 \simeq \Delta_\ell^{0,0} = \Delta_\ell^{0, \jj}$ as well as $\Lambda_\ell^k \lesssim \Delta_\ell^{k,\jj}$ for all $(\ell,k,\jj) \in \QQ$, where the hidden constants depend only on $\nu$. Together with~\eqref{eq:step3:proof:convergence} from Step~2, it thus only remains to prove $\Delta_\ell^{k,0} \lesssim \Lambda_\ell^{k-1}$ for $k \ge 1$. 

To this end, let $(\ell,k,\jj) \in \QQ$ with $k \ge 1$.
From contraction~\eqref{eq:zarantonello:unperturbed} of the unperturbed Zarantonello symmetrization and nested iteration $u_\ell^{k,0} = u_\ell^{k-1,\jj}$, we get that
\begin{align*}
	\Delta_\ell^{k,0}
	&=
	\enorm{u_\infty^\exact - u_\ell^{k-1,\jj}} + \enorm{u_\ell^{k,\star} - u_\ell^{k-1,\jj}} + \eta_\ell(u_\ell^{k-1,\jj})
	\\&
	\eqreff*{eq:zarantonello:unperturbed}\le
	\enorm{u_\infty^\exact - u_\ell^{k-1,\jj}} + (1+\qsym) \, \enorm{u_\ell^\star - u_\ell^{k-1,\jj}} + \eta_\ell(u_\ell^{k-1,\jj}).
\end{align*}
The C\'ea lemma~\eqref{eq:cea:infty} proves that
\begin{align*}
	\enorm{u_\ell^\star - u_\ell^{k-1,\jj}}
	\le
	\enorm{u_\infty^\exact - u_\ell^\star} + \enorm{u_\infty^\exact - u_\ell^{k-1,\jj}}
	\eqreff{eq:cea:infty}\lesssim
	\enorm{u_\infty^\exact - u_\ell^{k-1,\jj}}.
\end{align*}
Combining the last two estimates, we arrive at
\begin{align*}
	\Delta_\ell^{k,0}
	\lesssim
	\enorm{u_\infty^\exact - u_\ell^{k-1,\jj}} + \eta_\ell(u_\ell^{k-1,\jj})
	\simeq \Lambda_\ell^{k-1}.
\end{align*}
This concludes the proof of~\eqref{eq:step1:proof:convergence}.

{\bf Step~4.} In this step, we prove that
\begin{align}\label{eq:step4:proof:convergence}
	\sum_{j' = j}^{\jj[\ell,k]} \Delta_\ell^{k,j'}
	\lesssim
	\Delta_\ell^{k,\jj} + \Delta_\ell^{k,j}
	\quad \text{for all } (\ell,k,j) \in \QQ.
\end{align}
According to the right-hand side of~\eqref{eq:step4:proof:convergence}, it remains to consider the sum for $j' = j+1, \dots,\jj[\ell,k] - 1$. With~\eqref{eq:step2:proof:convergence} and contraction~\eqref{eq:algebra:contraction} of the algebraic solver, we get that
\begin{align*}%\label{eq:step3:proof:convergence}
	\sum_{j' = j+1}^{\jj[\ell,k]-1} \Delta_\ell^{k,j'}
	\, \eqreff{eq:step2:proof:convergence}\lesssim \,\,
	\sum_{j' = j+1}^{\jj[\ell,k]-1} \enorm{u_\ell^{k,\exact} - u_\ell^{k,j'-1}}
	\eqreff{eq:algebra:contraction}\le
	\enorm{u_\ell^{k,\exact} - u_\ell^{k,j}} \sum_{j' = j}^{\jj[\ell,k]-2} \qalg^{j'-j}.
\end{align*}
With the geometric series and $\enorm{u_\ell^{k,\exact} - u_\ell^{k,j}} \le \Delta_\ell^{k,j}$, this concludes the proof of~\eqref{eq:step4:proof:convergence}.

{\bf Step~5.}
%Define $\QQ^\star \coloneqq \set{(\ell,k)}{(\ell,k,0) \in \QQ \text{and } k < \kk[\ell]}$.
For $(\ell,k,j) \in \QQ$ with $\ell \ge \ell_0$, the preceding steps show that
\begin{align*}%\label{eq:step5:proof:convergence}
	&\sum_{\substack{(\ell',k',j') \in \QQ \\ (\ell',k',j') > (\ell,k,j)}} \!\! \Delta_{\ell'}^{k',j'}
	=
	\sum_{j' = j+1}^{\jj[\ell,k]} \Delta_\ell^{k,j'}
	+
	\!\! \!\sum_{\substack{(\ell',k',0) \in \QQ \\ (\ell',k',0) > (\ell,k,0)}}
	\! \sum_{j' = 0}^{\jj[\ell',k']} \Delta_{\ell'}^{k',j'}
	\\& \qquad \quad
	\eqreff*{eq:step4:proof:convergence}\lesssim
	\bigl[ \Delta_\ell^{k,\jj} + \Delta_\ell^{k,j} \bigr]
	+
	\! \! \! \! \sum_{\substack{(\ell',k',0) \in \QQ \\ (\ell',k',0) > (\ell,k,0)}} \! \!\bigl[ \Delta_{\ell'}^{k',\jj} + \Delta_{\ell'}^{k',0} \bigr]
	%\\&
	\eqreff*{eq:step3:proof:convergence}\lesssim
	\Delta_\ell^{k,j}
	+
	\! \!  \! \! \sum_{\substack{(\ell',k',0) \in \QQ \\ (\ell',k',0) > (\ell,k,0)}}  \! \! \Delta_{\ell'}^{k',0}.
	%\\&
\end{align*}
With linear convergence~\eqref{eq4:lemma:ghps} of $\Lambda_\ell^k$ from Lemma~\ref{lemma:ghps} and the geometric series, we thus see that
\begin{align*}
	\sum_{\substack{(\ell',k',0) \in \QQ \\ (\ell',k',0) > (\ell,k,0)}} \Delta_{\ell'}^{k',0}
	&=
	\sum_{k' = k+1}^{\kk[\ell]} \Delta_\ell^{k',0}
	+ \sum_{\ell' = \ell+1}^{\elll} \sum_{k'=0}^{\kk[\ell']} \Delta_{\ell'}^{k',0}
	\eqreff*{eq:step1:proof:convergence}\lesssim \,
	\sum_{k' = k}^{\kk[\ell]-1} \Lambda_\ell^{k'}
	+ \sum_{\ell' = \ell+1}^{\elll} \sum_{k'=0}^{\kk[\ell']-1} \Lambda_{\ell'}^{k'}
	\\&
	\eqreff*{eq4:lemma:ghps}\lesssim \, \Lambda_\ell^k + \Lambda_{\ell+1}^0
%	\eqreff{eq:step3:proof:convergence}\lesssim \Lambda_\ell^k
\le 2\Lambda_\ell^k
	\eqreff{eq:step1:proof:convergence}\lesssim \Delta_\ell^{k,\jj}
	\eqreff{eq:step3:proof:convergence}\lesssim \Delta_\ell^{k,j}.
\end{align*}
Altogether, this proves that
\begin{align*}%\label{eq:step5:proof:convergence}
	&\sum_{\substack{(\ell',k',j') \in \QQ \\ (\ell',k',j') > (\ell,k,j)}} \Delta_{\ell'}^{k',j'}
	\le C_{\rm sum} \Delta_\ell^{k,j}
	\quad \text{for all } (\ell,k,j) \in \QQ.
\end{align*}
According to basic calculus (see, e.g.,~\cite[Lemma 4.9]{axioms}), this is equivalent to linear convergence with respect to the lexicographic order on $\QQ$, i.e., for all $(\ell,k,j), (\ell',k',j') \in \QQ$ with $|\ell',k',j'| \le |\ell,k,j|$ and $\ell' \ge \ell_0$, it holds that
\begin{align*}%\label{eq:theorem:aisfem:linear-convergence}
 \Delta_\ell^{k,j}
 \le
 \Clin \, \qlin^{|\ell,k,j| - |\ell',k',j'|} \, \Delta_{\ell'}^{k',j'},
\end{align*}
where the constants $\Clin > 0$ and $0 < \qlin < 1$ depend only on $C_{\rm sum}$. This also yields that
\begin{align*}
 &\enorm{u^\exact - u_\ell^\exact}
 \eqreff{axiom:reliability}\lesssim 
 \eta_\ell(u_\ell^\exact) 
 \eqreff{axiom:stability}\lesssim 
 \eta_\ell(u_\ell^{k,j}) + \enorm{u_\ell^\exact - u_\ell^{k,j}}
 \le 
 \eta_\ell(u_\ell^{k,j}) + \enorm{u_\infty^\exact - u_\ell^{k,j}} + \enorm{u_\infty^\exact - u_\ell^\exact}
 \\& \quad
 \eqreff{eq:cea:infty}\lesssim 
 \eta_\ell(u_\ell^{k,j}) + \enorm{u_\infty^\exact - u_\ell^{k,j}}
 \to 0
 \quad \text{as } |\ell,k,j| \to \infty
\end{align*}
and hence $u_\infty^\exact = u^\exact$. In particular, the definitions of $\Delta_\ell^{k,j}$ from~\eqref{eq0:theorem:aisfem:linear-convergence} and~\eqref{eq:***} coincide. Overall, we thus conclude the proof of linear convergence~\eqref{eq:theorem:aisfem:linear-convergence}.
\qed

%\newpage

%%%%%%%%%%%%%%%%%%%%%%%%%%%%%%%%%%%%%%%%%%%%%%%%%%%%%%%%%%%%%%%%
\subsection{Proof of Theorem~\ref{theorem:aisfem:complexity}}\label{subsection:proof:complexity}
%%%%%%%%%%%%%%%%%%%%%%%%%%%%%%%%%%%%%%%%%%%%%%%%%%%%%%%%%%%%%%%%
The proof of Theorem~\ref{theorem:aisfem:complexity} requires the following auxiliary lemma stating that the error estimator \(\eta_\ell(u_\ell^{\kk, \jj})\) of the inexact but available final iterate of Algorithm~\ref{algorithm} is equivalent to the error estimator \(\eta_\ell(u_\ell^{\exact})\) of the (unknown) exact solution \(u_\ell^\exact\). While the statement is similar to~\cite[Lemma~7]{hpsv2021}, the present proof provides a minor clarification of the involved constant.

\begin{lemma}\label{lem:estimator-equivalence}
Recall $\Calg > 0$ from~\eqref{eq:def:Calg} and $\lamalg^\star > 0$ from Theorem~\ref{theorem:aisfem:linear-convergence}.
%		Define the constant
%		\begin{equation}
%			\Calg \coloneqq \frac{1}{1-\qsym} \, \Bigl(\frac{2 \, \qalg}{1-\qalg} \, \lamalg^\exact + \qsym \Bigr).
%		\end{equation}
		Then, for all \(0 < \theta \le 1\), \(0 < \lamalg \le \lamalg^\star\), \(0 < \lamsym < \lamsym^\exact = \min\{1, \Cstab^{-1} \Calg^{-1}\}\), and all \((\ell, \kk, \jj) \in \QQ\), it holds that
		\begin{align}\label{eq:step3-estimator-equivalence}
		\enorm{u_\ell^\exact - u_\ell^{\kk, \jj}} 
		 \le \Calg \, \lamsym \,  \eta_{\ell}(u_\ell^{\kk, \jj}).
		\end{align}
Moreover, there holds  equivalence
		\begin{equation}\label{eq:estimator-equivalence}
			\bigl[1 - \lambda_{\rm sym} / \lamsym^\exact\bigr] \eta_{\ell}(u_\ell^{\kk, \jj}) 
			\le \eta_{\ell}(u_\ell^\exact) \le \bigl[1 + \lambda_{\rm sym} / \lamsym^\exact\bigr] \eta_{\ell}(u_\ell^{\kk, \jj}).
		\end{equation}
\end{lemma}

\begin{proof}
	The proof consists of two steps.
	
	\medskip
	\textbf{Step~1.} 
	Recall from~\eqref{eq2:lem:contraction_perturbed:kk} that
	\begin{align*}
 \enorm{u_\ell^\exact - u_\ell^{\kk,\jj}}
 \eqreff{eq2:lem:contraction_perturbed:kk}\le
 \qsym \, \enorm{u_\ell^\exact - u_\ell^{\kk-1,\jj}} + \frac{2 \, \qalg}{1-\qalg} \, \lamalg \, \lamsym \, \eta_{\ell}(u_\ell^{\kk, \jj}).
\end{align*}
The stopping criterion in Algorithm~\ref{algorithm}(i.d) proves that
\begin{align*}
 \enorm{u_\ell^\exact - u_\ell^{\kk-1,\jj}}
 \le \enorm{u_\ell^\exact - u_\ell^{\kk,\jj}} + \enorm{u_\ell^{\kk,\jj} - u_\ell^{\kk-1,\jj}}
 \stackrel{\rm(i.d)}\le  \enorm{u_\ell^\exact - u_\ell^{\kk,\jj}} + \lamsym \eta_\ell(u_\ell^{\kk,\jj}).
\end{align*}
Combining these estimates with $0 < \lamalg \le \lamalg^\star$, we prove~\eqref{eq:step3-estimator-equivalence}, since 
\begin{align*}
 \enorm{u_\ell^\exact - u_\ell^{\kk,\jj}} 
 \le
 \frac{1}{1-\qsym} \Big( \frac{2 \, \qalg}{1-\qalg} \, \lamalg +  \qsym  \Big) \, \lamsym \, \eta_{\ell}(u_\ell^{\kk, \jj})
 \le \Calg \lamsym \, \eta_{\ell}(u_\ell^{\kk, \jj}).
\end{align*}

	\medskip
	\textbf{Step~2.} With the definition \(\lambda_{\rm sym}^\exact = \min\{1, \Cstab^{-1} \Calg^{-1}\}\), stability~\eqref{axiom:stability} and~\eqref{eq:step3-estimator-equivalence} show
	\begin{align*}
		\eta_\ell(u_\ell^\exact) 
		\eqreff*{axiom:stability}\le \eta_{\ell}(u_\ell^{\kk, \jj}) + \Cstab \, \enorm{u_\ell^{\exact} - u_\ell^{\kk, \jj}} 
		&\eqreff*{eq:step3-estimator-equivalence}\le \bigl[1 + \Cstab \, \Calg \lamsym \bigr] \, \eta_\ell(u_\ell^{\kk, \jj})
		\\&
 \le \bigl[1 + \lamsym / \lamsym^\exact \bigr] \, \eta_\ell(u_\ell^{\kk, \jj}).
\end{align*}
%In particular, this proves that
%\begin{align*}	
% 	\eta_\ell(u_\ell^\exact) 	
%		\eqreff*{eq:step3-estimator-equivalence}\le \bigl[1 + \Cstab \, \Calg\bigr] \, \eta_\ell(u_\ell^{\kk, \jj})
%		= 
%\le\bigl[1 + \lamsym / \lamsym^\exact \bigr] \, \eta_\ell(u_\ell^{\kk, \jj}).
%	\end{align*}
%	
%	\medskip
%	\textbf{Step~5.} 
If \(0 < \lambda_{\rm sym} <  \lambda_{\rm sym}^\exact\), the analogous argument also proves the converse inequality
	\begin{equation}
		\bigl[1 - \lamsym / \lamsym^\exact\bigr] \, \eta_\ell(u_\ell^{\kk, \jj})  \le \eta_\ell(u_\ell^\exact).
	\end{equation}
	This concludes the proof.
\end{proof}

\begin{proof}[Proof of Theorem~\ref{theorem:aisfem:complexity}]
	It is sufficient to show that
\begin{subequations}\label{eq:equiv:complexity}
	\begin{align}
		\label{eq:equiv:lowerbound}\norm{u^\exact}_{\A_s(\TT_{0})}
		&\lesssim \sup_{\substack{(\ell,k,j) \in \QQ}} (\#\TT_\ell)^s \, \Delta_\ell^{k,j}, \\
		\label{eq:equiv:upperbound}\sup_{\substack{(\ell,k,j) \in \QQ \\ \ell \ge\ell_0}} (\#\TT_\ell)^s \, \Delta_\ell^{k,j} &\lesssim  \max\{\norm{u^\exact}_{\A_s(\TT_{\ell_0})}, \Delta_{\ell_0}^{0,0}\}.
	\end{align}
\end{subequations}
Then,~\eqref{eqx:theorem:aisfem:complexity} follows from~\eqref{eq:equiv:upperbound} and Corollary~\ref{corollary:aisfem:linear-convergence}.
We split the proof into six steps.

\medskip
\textbf{Step 1.} We first show~\eqref{eq:equiv:lowerbound} for the case $\elll = \infty$. Algorithm~\ref{algorithm} ensures that $\# \TT_{\ell} \to \infty$ as $\ell \to \infty$. We recall that in NVB refinement an element is split into at least two but at most $\Cchild$ child elements. In particular, for all $\ell \ge 0$, we have that
\begin{align}\label{eq:Cchild}
	\# \TT_{\ell +1} \le \Cchild \, \# \TT_\ell.
\end{align}
For any given $N \in \N$, we can argue similarly as in the proof
of~\cite[Proposition~4.15]{axioms}. Choose the maximal index $\ell' \in \N_0$ such that $\# \TT_{\ell'} - \# \TT_{0}  \le N$. 
The maximality of $\ell'$ leads us to
\begin{align}\label{eq:estimate-N}
	N+1 < \# \TT_{\ell'+1} - \# \TT_{0} + 1
	\le \# \TT_{\ell'+1}
	\eqreff{eq:Cchild}\le
	\Cchild \,  \# \TT_{\ell'}.
\end{align}
Since $\TT_{\ell'} \in \T_N(\TT_{0})$, we have that
\begin{align}\label{eq:minopt_est}
	\min_{\TT_{\rm opt} \in \T_N (\TT_{0})} \bigl[ \enorm{u^\star - u^\star_{\rm opt}} + \eta_{\rm opt} (u^\star_{\rm opt}) \bigr]
	\le   \enorm{u^\star - u^\star_{\ell'}} + \eta_{\ell'} (u^\star_{\ell'}),
\end{align}
and stability~\eqref{axiom:stability} and the C\'ea lemma~\eqref{eq:intro:cea} show, for $(\ell',k',j') \in \QQ$, that
\begin{align}
	\enorm{u^\star - u^\star_{\ell'}} + \eta_{\ell'} (u^\star_{\ell'})
	&\eqreff*{axiom:stability}\le  \enorm{u^\star - u^\star_{\ell'}} +
	\eta_{\ell'} (u^{k',j'}_{\ell'}) + \Cstab \, \enorm{u^\star_{\ell'}-u^{k',j'}_{\ell'}} \nonumber\\
	&\le (1+\Cstab ) \enorm{u^\star - u^\star_{\ell'}} +
	\eta_{\ell'} (u^{k',j'}_{\ell'}) + \Cstab \, \enorm{u^\star-u^{k',j'}_{\ell'}} \nonumber \\
	&\eqreff*{eq:intro:cea}\le
	\bigl(\Ccea \, (1+\Cstab ) + \Cstab \bigr) \,\enorm{u^\star - u^{k',j'}_{\ell'}} +
	\eta_{\ell'} (u^{k',j'}_{\ell'})\nonumber \\
	&\le \bigl(\Ccea \, (1+\Cstab ) + \Cstab \bigr) \, \Delta_{\ell'}^{k',j'}.\label{eq:opt_to_quasierr}
\end{align}
A combination of the previous estimates leads us to
\begin{align*}
	\bigl( N+1 \bigr)^s &\min_{\TT_{\rm opt} \in \T_N (\TT_{0}) } \bigl[ \enorm{u^\star - u^\star_{\rm opt}} + \eta_{\rm opt} (u^\star_{\rm opt}) \bigr]
	\ \eqreff*{eq:minopt_est}\le \ 
	\bigl( N+1 \bigr)^s   \bigl[\enorm{u^\star - u^\star_{\ell'}} + \eta_{\ell'} (u^\star_{\ell'}) \bigr]  \\
	&\eqreff*{eq:estimate-N} \le \ 
	\Cchild^s \, \bigl(\# \TT_{\ell'} \bigr)^s  \bigl[ \enorm{u^\star - u^\star_{\ell'}} + \eta_{\ell'} (u^\star_{\ell'}) \bigr]
	\eqreff{eq:opt_to_quasierr}\lesssim \bigl(\# \TT_{\ell'} \bigr)^s  \Delta_{\ell'}^{k',j'}
	\le \sup_{\substack{(\ell,k,j) \in \QQ}} \bigl( \# \TT_\ell
	\bigr)^s \, \Delta_\ell^{k,j}.
\end{align*}
Finally, taking the supremum over all $N$ yields the sought result
\begin{align*}
	\norm{u^\exact}_{\A_s(\TT_{0})}
	\lesssim \sup_{\substack{(\ell,k,j) \in \QQ }} \bigl( \# \TT_\ell
	\bigr)^s \, \Delta_\ell^{k,j}.
\end{align*}

\medskip
\textbf{Step 2.} We proceed to show~\eqref{eq:equiv:lowerbound} for the case $\elll < \infty$. Recall from Lemma~\ref{lemma:new:ghps} that $\eta_{\elll}(u_\elll^\exact) = 0$ and $u_\elll^\exact = u^\exact$. Without loss of generality, we may assume $\elll > 0$, since otherwise $\norm{u^\exact}_{\A_s(\TT_{0})} = 0$. Combined with reliability~\eqref{axiom:reliability}, this yields that
\begin{align}\label{eq:estimate-approx-class}
	\begin{split}
		\norm{u^\exact}_{\A_s(\TT_{0})}
		&\eqreff*{eq:def_approx_class}=
		\, \sup_{N \in \N_0} \Bigl( \bigl( N+1 \bigr)^s \min_{\TT_{\rm opt} \in \T_N (\TT_{0}) } \bigl[ \enorm{u^\star - u^\star_{\rm opt}} + \eta_{\rm opt} (u^\star_{\rm opt}) \bigr] \Bigr)\\
		&\eqreff*{axiom:reliability}\le \, (1+\Crel) \, \sup \limits_{0 \le N < \# \TT_{\elll} - \# \TT_{0}} \Bigl((N+1)^s \min_{\TT_{\rm opt} \in \T_N (\TT_{0}) }   \eta_{\rm opt}(u_{\rm opt}^\exact)  \Bigr).
	\end{split}
\end{align}
We argue as in Step~1 above: Let $0 \le N < \# \TT_{\elll} - \# \TT_0$. Choose the maximal index $0 \le \ell' < \elll$ with $\# \TT_{\ell'} - \# \TT_0 \le N$. Arguing along the lines of \eqref{eq:estimate-N}--\eqref{eq:opt_to_quasierr}, we see that 
\begin{align*}
	\sup \limits_{0 \le N < \# \TT_{\elll} - \# \TT_{0}} \Bigl(\bigl(N+1 \bigr)^s \, \min \limits_{\TT_{\rm opt} \in \T_N(\TT_{0})} \eta_{\rm opt}(u^\exact_{\rm opt})\Bigr) \lesssim  \sup_{\substack{(\ell,k,j) \in \QQ }} \bigl( \# \TT_\ell
	\bigr)^s \, \Delta_\ell^{k,j}.
\end{align*}
Combining this with \eqref{eq:estimate-approx-class}, we conclude the lower bound \eqref{eq:equiv:lowerbound} also in this case.

\medskip
\textbf{Step 3.} We prove~\eqref{eq:equiv:upperbound} for $ \norm{u^\exact}_{\A_s(\TT_{\ell_0})} < \infty$, since the result becomes trivial if $\norm{u^\exact}_{\A_s(\TT_{\ell_0})} = \infty$. First, we show that for all
$\ell' \ge \ell_0$ with $(\ell'+1, 0, 0) \in \QQ$, there exists $ \RR_{\ell'}  \subseteq \TT_{\ell'}$ such that
\begin{align}\label{eq2:marking_est}
	\# \RR_{\ell'}  \lesssim \norm{u^\exact}_{\A_s(\TT_{\ell_0})}^{1/s}  \, (\Delta_{\ell'+1}^{0, \jj})^{-1/s} \quad \text{and} \quad \thetamark \eta_{\ell'}(u_{\ell'}^\exact)^{2} \le \eta_{\ell'}(\RR_{\ell'}, u_{\ell'}^\exact)^{2} .
\end{align}
Since $0 < \thetamark = (\theta^{1/2}+ \, \lamsym / \lamsym^\exact)^{2}  \, (1-\lamsym / \lamsym^\exact)^{-{2}} < \theta^\exact$, and because there holds \eqref{axiom:discrete_reliability}, \cite[Lemma~4.14]{axioms} ensures, for all $\ell' \ge \ell_0$, the existence of a set $\RR_{\ell'} \subseteq \TT_{\ell'}$ satisfying
\begin{align}\label{eq:marking_est}
	\# \RR_{\ell'}  \lesssim \norm{u^\exact}_{\A_s(\TT_{\ell_0})}^{1/s} \, \eta_{\ell'} (u^\star_{\ell'})^{-1/s}  \quad \text{and} \quad \thetamark \eta_{\ell'}(u_{\ell'}^\exact)^{2} \le \eta_{\ell'}(\RR_{\ell'}, u_{\ell'}^\exact)^{2} .
\end{align}
Since $\lamsym/\lamsym^\exact < 1$ by assumption, the estimator equivalence~\eqref{eq:estimator-equivalence} shows that
\begin{align}
	\bigl[1-\lamsym / \lamsym^\exact \bigr] \, \eta_{\ell'}(u_{\ell'}^{\kk, \jj})  \le  \eta_{\ell'}(u_{\ell'}^\exact),
\end{align}%
which leads us to
\begin{align*}
	\# \RR_{\ell'}  \eqreff{eq:marking_est}\lesssim \norm{u^\exact}_{\A_s(\TT_{\ell_0})}^{1/s} \, \eta_{\ell'} (u_{\ell'}^{\kk, \jj})^{-1/s}.
\end{align*}
Moreover, thanks to nested iteration, Step~3 of the proof of Theorem~\ref{theorem:aisfem:linear-convergence}, Step~3 of the proof of Lemma~\ref{lemma:ghps}, and reliability \eqref{eq:reliable-error} of Proposition~\ref{proposition:reliability}, there holds that
\begin{align}\label{eq:Delta-estimator}
	\begin{split}
		\Delta_{\ell'+1}^{0, \jj}
		\eqreff{eq:step1:proof:convergence}\simeq \Lambda_{\ell'+1}^0
		&= \bigl[ \enorm{u^\exact - u_{\ell'}^{\kk, \jj}}^2 + \nu \, \eta_{{\ell'}+1}(u_{\ell'}^{\kk, \jj})^2 \bigr]^{1/2}
		\\
		&
		\eqreff*{eq5:combined-contraction}\lesssim
		\bigl[ \enorm{u^\exact - u_{\ell'}^{\kk, \jj}}^2 + \, \eta_{\ell'}(u_{\ell'}^{\kk, \jj})^2 \bigr]^{1/2}
		\eqreff{eq:reliable-error}\lesssim\eta_{\ell'} (u_{\ell'}^{\kk, \jj}).
	\end{split}
\end{align}
By summarizing the last two estimates, we obtain~\eqref{eq2:marking_est}.

\medskip
\textbf{Step 4.} For $(\ell'+1, 0, 0) \in \QQ$ with $\ell' \ge \ell_0$, we show that
\begin{align}\label{eq:dorfleropt}
	\# \MM_{\ell'} \le C_{\rm mark} \, \# \RR_{\ell'}.
\end{align}
with the constant $\Cmark \ge 1$ from Algorithm~\ref{algorithm}.
Recall the definition 
$$
\thetamark \eqreff{eqxx:theorem:aisfem:complexity}= \Big(\frac{\theta^{1/2}  + \, \lamsym/\lamsym^\exact}{1-\lamsym / \lamsym^\exact}\Big)^{2} 
\quad\text{with}\quad
\lamsym^\exact = \min\{1,\Calg^{-1} \Cstab^{-1}\}.
$$
 This shows that

\begin{align}\label{eq:cstinopt}
	\begin{split}
		\enorm{u_{\ell'}^\exact - u_{\ell'}^{\kk, \jj}} \, &\eqreff*{eq:step3-estimator-equivalence}\le \Calg \, \lamsym \, \eta_{\ell'}(u_{\ell'}^{\kk,\jj}) 
		\\&
		\le \Cstab^{-1} \, \frac{\lamsym}{\lamsym^\exact} \, \eta_{\ell'}(u_{\ell'}^{\kk,\jj})
		= \Cstab^{-1}\Bigl( \thetamark^{1/2}  \bigl[1- \lamsym/\lamsym^\exact\bigr] - \theta^{1/2}  \Bigr) \, \eta_{\ell'}(u_{\ell'}^{\kk,\jj}).
	\end{split}
\end{align}
Now, we can estimate
\begin{align*}
	\thetamark^{1/2} \bigl[1- \lamsym/\lamsym^\exact\bigr] \, \eta_{\ell'}(u_{\ell'}^{\kk',\jj'})
	&\eqreff*{eq:estimator-equivalence}\le \thetamark^{1/2} \eta_{\ell'}(u_{\ell'}^\exact)
	\eqreff*{eq:marking_est}\le \eta_{\ell'}(\RR_{\ell'}, u_{\ell'}^\exact)
	\\
	&\eqreff*{axiom:stability}\le \eta_{\ell'}(\RR_{\ell'}, u_{\ell'}^{\kk',\jj'}) + \Cstab \, \enorm{u_{\ell'}^\exact - u_{\ell'}^{\kk', \jj'}}
	\\
	&
	\eqreff*{eq:cstinopt}\le \eta_{\ell'}(\RR_{\ell'}, u_{\ell'}^{\kk',\jj'}) + \Bigl( \thetamark^{1/2} \bigl[1- \lamsym/\lamsym^\exact\bigr] - \theta^{1/2}  \Bigr) \eta_{\ell'}(u_{\ell'}^{\kk',\jj'}).
\end{align*}
Rearranging the terms, we obtain that $\RR_{\ell'}$ from Step~3 satisfies the D\"orfler marking criterion of Algorithm~\ref{algorithm}(iii) with the same parameter $\theta$, i.e., there holds
\begin{align}\label{eq:doerflerRl}
	\theta \, \eta_{\ell'}(u_{\ell'}^{\kk',\jj'})^{2}  \le \eta_{\ell'}(\RR_{\ell'}, u_{\ell'}^{\kk',\jj'})^{2} .
\end{align}
Hence, quasi-minimality of the set of marked elements $\MM_{\ell'}$ implies \eqref{eq:dorfleropt}.

\medskip
\textbf{Step 5.} Consider the case $(\ell, k, j) \in \QQ$ with $\ell \ge \ell_0$. Full linear convergence from Theorem~\ref{theorem:aisfem:linear-convergence} yields that
\begin{align}\label{eq:lin_cv_sum}
	\sum_{\substack{(\ell',k',j') \in \QQ \\ |\ell',k',j'| \le |\ell,k,j| \\ \ell' \ge \ell_0}} (\Delta_{\ell'}^{k', j'})^{-1/s}
	\eqreff{eq:theorem:aisfem:linear-convergence}\lesssim (\Delta_{\ell}^{k, j})^{-1/s} \sum_{\substack{(\ell',k',j') \in \QQ \\ |\ell',k',j'| \le |\ell,k,j| \\ \ell' \ge \ell_0}} (\qlin^{1/s})^{|\ell, k, j| - |\ell', k', j'|} \lesssim (\Delta_{\ell}^{k, j})^{-1/s}.
\end{align}
Recall that NVB refinement satisfies the mesh-closure estimate, i.e., there holds that
\begin{align}\label{eq:meshclosure}
	\# \TT_\ell - \# \TT_0\le \Cmesh \sum_{\ell' = 0}^{\ell-1} \# \MM_{\ell'} \quad \text{for all } \ell \ge 0,
\end{align}
where $\Cmesh > 1$ depends only on $\TT_{0}$.
Thus, for $(\ell, k, j) \in \QQ$ with $\ell > \ell_0$, we have by the mesh-closure estimate~\eqref{eq:meshclosure}, optimality of D\"orfler marking \eqref{eq:dorfleropt}, and full linear convergence \eqref{eq:lin_cv_sum} that 
\begin{align*}%\label{eq:mesh_closure_est}
	&\# \TT_{\ell} - \# \TT_{\ell_0}
	\eqreff*{eq:meshclosure}\lesssim \sum_{\ell' = \ell_0}^{\ell-1}  \# \MM_{\ell'}
	\eqreff{eq:dorfleropt}\lesssim \sum_{\ell' = \ell_0}^{\ell-1}  \#  \RR_{\ell'}
%	\\&
	\eqreff*{eq2:marking_est}\lesssim \  \norm{u^\exact}_{\A_s(\TT_{\ell_0})}^{1/s} \sum_{\ell' = \ell_0}^{\ell-1}  (\Delta_{\ell'+1}^{0, \jj})^{-1/s}
	\\& \quad
	\le  \norm{u^\exact}_{\A_s(\TT_{\ell_0})}^{1/s} \! \!  \sum_{\substack{(\ell',k',j') \in \QQ \\ |\ell',k',j'| \le |\ell,k,j| \\ \ell' \ge \ell_0}} (\Delta_{\ell'}^{k', j'})^{-1/s}
	\eqreff{eq:lin_cv_sum}\lesssim \norm{u^\exact}_{\A_s(\TT_{\ell_0})}^{1/s} (\Delta_{\ell}^{k, j})^{-1/s}.
\end{align*}
Rearranging the terms and noting that $\# \TT_{\ell} - \# \TT_{\ell_0} +1 \le 2 \, (\#\TT_{\ell} - \# \TT_{\ell_0})$, we obtain that
\begin{align*}%\label{eq:mesh_closure_est}
	(\# \TT_{\ell} - \# \TT_{\ell_0} + 1)^s \Delta_{\ell}^{k, j}
	\lesssim \norm{u^\exact}_{\A_s(\TT_{\ell_0})} \quad \text{for} \quad \ell > \ell_0.
\end{align*}
Trivially, full linear convergence~\eqref{eq:theorem:aisfem:linear-convergence} proves that
\begin{align*}%\label{eq:mesh_closure_est}
	(\# \TT_{\ell} - \# \TT_{\ell_0} + 1)^s \Delta_{\ell_0}^{k, j}
	= \Delta_{\ell_0}^{k, j}
	\eqreff{eq:theorem:aisfem:linear-convergence}\lesssim \Delta_{\ell_0}^{0, 0}  \quad \text{for} \quad \ell = \ell_0.
\end{align*}
We recall from~\cite[Lemma~22]{bhp2017} that for all $\TT_\coarse \in \T$ and all $\TT_\fine \in \T(\TT_\coarse)$, it holds that
\begin{align}\label{eq:bhp-lemma22}
	\# \TT_\fine - \# \TT_\coarse +1
	\le \# \TT_\fine
	\le  \# \TT_\coarse \, (\# \TT_\fine - \# \TT_\coarse +1).
\end{align}
Overall, we have thus shown that
\begin{align*}
	(\# \TT_{\ell})^s \Delta_\ell^{k,j} \eqreff{eq:bhp-lemma22}\lesssim (\# \TT_{\ell} - \# \TT_{\ell_0} + 1)^s \Delta_\ell^{k,j}
	\lesssim \max \{\norm{u^\exact}_{\A_s(\TT_{\ell_0})}, \Delta_{\ell_0}^{0,0}\}
\end{align*}
for all $(\ell, k, j) \in \QQ$ with $\ell \ge \ell_0$.
This concludes the proof of the upper bound in~\eqref{eq:equiv:upperbound} and hence that of \eqref{eq:theorem:aisfem:complexity}.

\medskip
\textbf{Step 6.} We prove the equivalence in \eqref{eq2:theorem:aisfem:complexity} by combining the steps above.  Recall that
\begin{align*}
	\QQ \backslash \set{(\ell,k,j) \in \QQ \given \ell \ge \ell_0}
	= \set{(\ell,k,j) \in \QQ \given \ell < \ell_0}
	\quad \text{is finite}
\end{align*}
and that $\norm{u^\exact}_{\A_s(\TT_{0})} < \infty$ is equivalent to $\norm{u^\exact}_{\A_s(\TT_{\ell_0})} < \infty$. Thus, the claim follows immediately by the equivalence in \eqref{eq:theorem:aisfem:complexity}. This concludes the proof.
\end{proof}

%%%%%%%%%%%%%%%%%%%%%%%%%%%%%%%%%%%%%%%%%%%%%%%%%%%%%%%%%%%%%%%%
%%%%%%%%%%%%%%%%%%%%%%%%%%%%%%%%%%%%%%%%%%%%%%%%%%%%%%%%%%%%%%%%
\section{Numerical experiments}
\label{section:numerics}
%%%%%%%%%%%%%%%%%%%%%%%%%%%%%%%%%%%%%%%%%%%%%%%%%%%%%%%%%%%%%%%%
%%%%%%%%%%%%%%%%%%%%%%%%%%%%%%%%%%%%%%%%%%%%%%%%%%%%%%%%%%%%%%%%
We consider the model problem~\eqref{eq:intro:model_pb} from the introduction. The \textsc{Matlab} implementation of the following experiments is embedded into the open source software package MooAFEM from \cite{MooAFEM}. In the following, Algorithm~\ref{algorithm} employs the optimal local $hp$-robust multigrid method from \cite{imps2022} as algebraic solver and the standard residual error estimator $\eta_\ell$. Given $T \in \TT_\ell$ and $v_\ell \in \XX_\ell$, the local contribution of $\eta_{\ell}$ reads
\begin{equation*}
	\eta_\ell(T; v_\ell)^2 \coloneqq h_T^2 \, \norm{-\div(\boldsymbol{A} \, \nabla v_\ell - \bff) + \boldsymbol{b} \cdot \nabla v_\ell + c \, v_\ell - f}_{L^2(T)}^2 + h_T \, \norm{\jump{(\boldsymbol{A} \, \nabla v_\ell - \bff) \cdot \boldsymbol{n}}}_{L^2(\partial T \cap \Omega)}^2.
\end{equation*}
It is well-known (see, e.g., \cite[Section~6.1]{axioms}) that $\eta_{\ell}$ satisfies the axioms \eqref{axiom:stability}--\eqref{axiom:discrete_reliability} from Section~\ref{subsection:axioms}.
%%%%%%%%%%%%%%%%%%%%%%%%%%%%%%%%%%%%%%%%%%%%%%%%%%%%%%%%%%%%%%%%
\subsection{Diffusion-convection-reaction on L-shaped domain} \label{section:Lshape}
%%%%%%%%%%%%%%%%%%%%%%%%%%%%%%%%%%%%%%%%%%%%%%%%%%%%%%%%%%%%%%%%
In this subsection, we consider the problem \eqref{eq:intro:model_pb} on the L-shaped domain $\Omega = (-1,1)^2 \setminus \bigl([0,1] \times [-1,0]\bigr) \subset \R^2$ with coefficients $\boldsymbol{A}(x) = \mathrm{Id}$, $\bfb(x) = x$, and $c(x) = 1$, and right-hand side $f(x) = 1$, i.e.,
\begin{align*}
	-\Delta u^{\exact}(x) + x \cdot \nabla u^\exact(x) + u^\exact(x) = 1
	\quad \text{for } x \in \Omega
	\quad \text{subject to} \quad
	u^\exact(x) = 0
	\quad \text{for } x \in \partial \Omega.
\end{align*}

%%%%%%%%%%%%%%%%%%%%%%%%%%%%%%%%%%%%%%%%%%%%%%%%%%%%%%%%%%%%%%%%
\subsubsection*{\textbf{Optimality of AISFEM}}
%%%%%%%%%%%%%%%%%%%%%%%%%%%%%%%%%%%%%%%%%%%%%%%%%%%%%%%%%%%%%%%%
We first display the optimality of Algorithm~\ref{algorithm} with respect to the computational cost stated in Theorem~\ref{theorem:aisfem:complexity} using the equivalence $\# \TT_{\ell} \simeq {\rm dim} \, \XX_\ell$. Numerically, we test with the parameters $\lamsym = \lamalg = 0.1$, $\delta = 0.5$, and $\theta = 0.5$ and, unless stated explicitly, the stopping criterion ${\rm dim} \, \XX_\ell > 10^7$. Note that both the total error and the algebraic error are unknown in all practical purposes. Therefore, we cannot study the decay of the quasi-error, but rather consider the equivalent error estimator $\eta_{\ell}(u_\ell^{\kk, \jj}) \simeq \Delta_\ell^{\kk,\jj}$. Figure~\ref{fig:optimality} shows that the proposed algorithm achieves optimal rates $-m/2$ for several polynomial degrees $m$ both with respect to the computational costs and the elapsed computational time after a short preasymptotic phase.
\begin{figure}[htp!]
	\centering
	\subfloat{
		\includegraphics[scale=0.8]{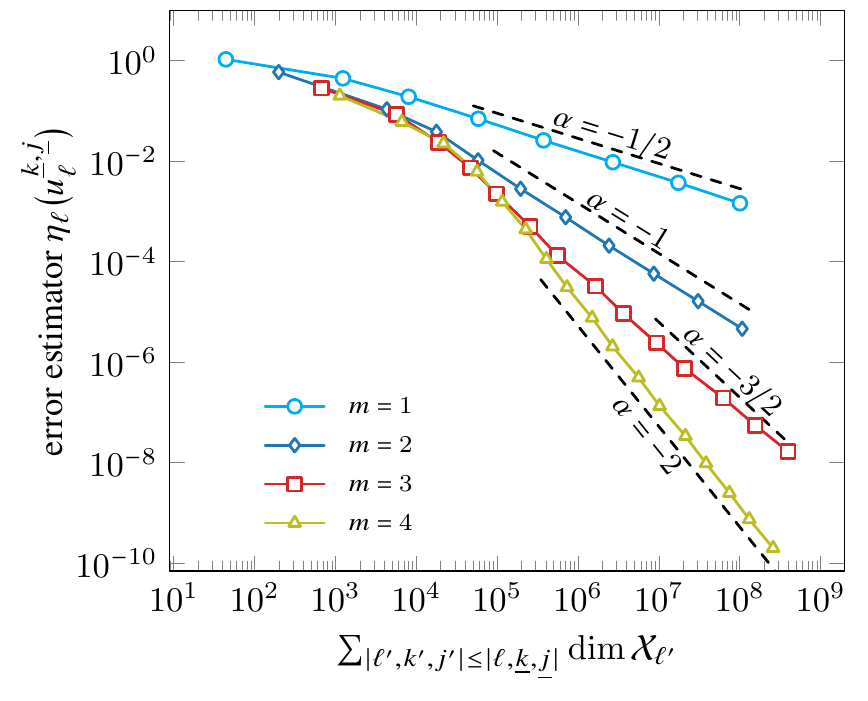}
		\label{fig:convergence}
	}
	\subfloat{
		\includegraphics[scale=0.8]{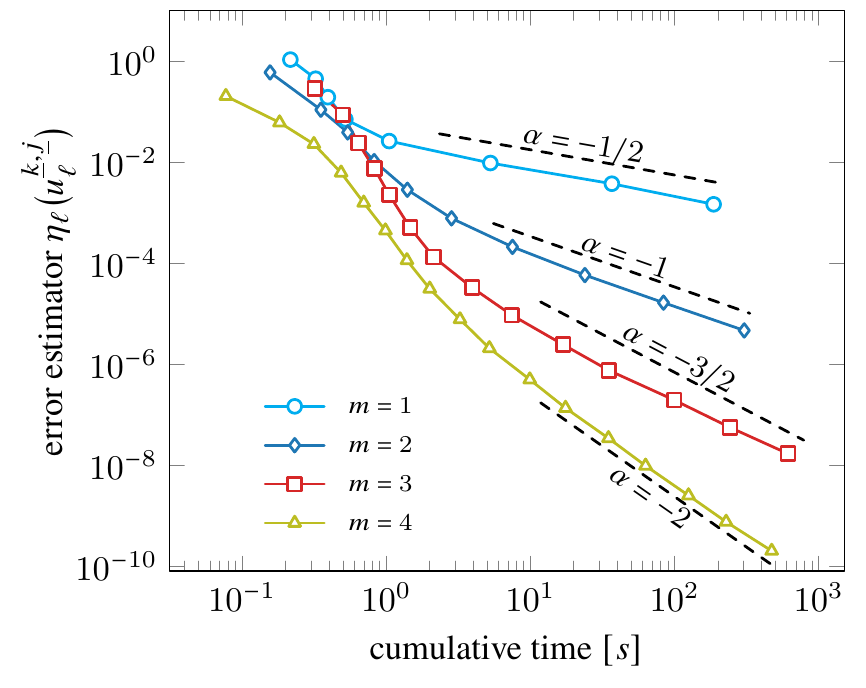}
		\label{fig:complexity}
	}
	\caption{\label{fig:optimality}\textit{Optimality of AISFEM for the diffusion-convection-reaction problem on the L-shaped domain from Section~\ref{section:Lshape}.} Convergence history plot of the error estimator $\eta_{\ell}(u_\ell^{\kk, \jj})$ over the computational costs (left) and the elapsed computational time (right) for different polynomial degrees $m$.}
\end{figure}

\subsubsection*{\textbf{Optimality of the iteratively symmetrized solver}} 
Optimality of AISFEM is possible when the inherent symmetrization and algebraic procedures are treated efficiently.
In Figure~\ref{fig:timings}, we present the time required for our iteratively symmetrized solver compared to the \textsc{Matlab} built-in direct solver (backslash) of the linear system related to \eqref{eq:intro:discrete}. We note that the displayed timings are comparing the direct solve time itself with the remaining time (including the setup of the Zarantonello system, computation of the error estimator, and mesh refinement). Hence, the presented numbers favor the built-in direct solver over the \textsc{Matlab}-implemented multigrid code. Nevertheless, the combination of the Zarantonello symmetrization with the optimal local multigrid solver from \cite{imps2022} appears to be of comparable speed to the built-in direct solver with the observation that as the dimension of the linear system increases, the backslash performance begins to degrade. Moreover, Figure~\ref{fig:iterations} shows that the iteration numbers of the solver remain uniformly bounded in the levels for various choices of the parameters $\lamsym$ and $\theta$. Note that when $\lamsym$ decreases, a higher accuracy of the Zarantonello symmetrization is required. Therefore, the iteration numbers are expected to increase with smaller $\lamsym$ as seen in Figure~\ref{fig:iterations} (left). Moreover, the iteration numbers are also expected to increase as $\theta$ becomes larger. This is due to the aggressive refinement leading to hierarchies of low numbers of levels but with considerable increase in the dimension of the linear systems. This may lead to the conclusion that $\theta$ should be chosen very small in order to have less iterations per level, but studying the \emph{cumulative} solver steps in Figure~\ref{fig:iterations} (right) shows that this is not the best strategy.
\begin{figure}[htpb!]
	\centering
	\includegraphics[scale=0.885]{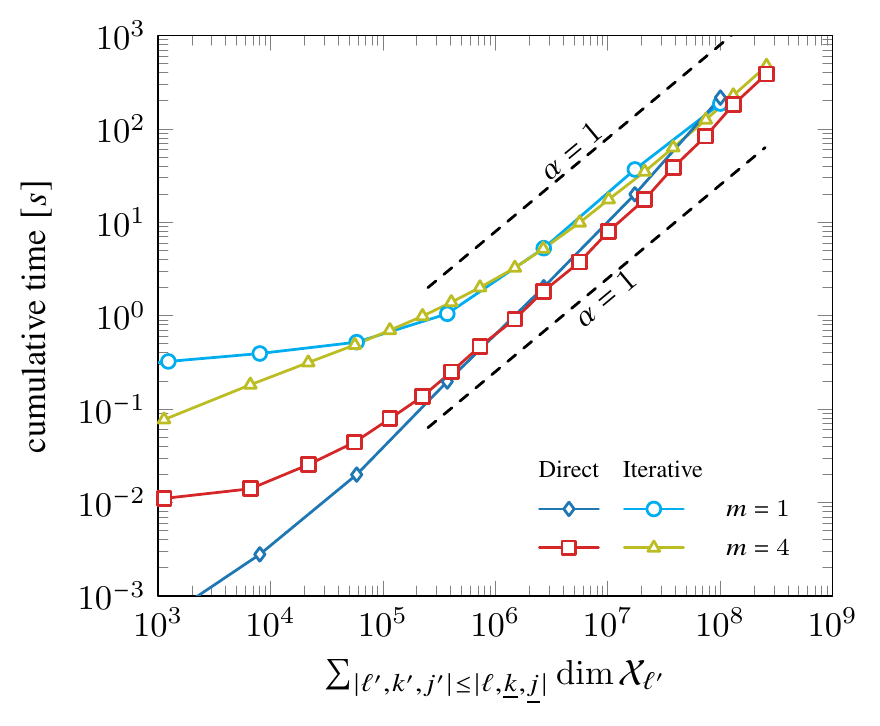}
	\label{fig:backslash}
	
	\caption{\textit{Optimality of the combined iterative solver for the diffusion-convection-reaction problem on the L-shaped domain from Section~\ref{section:Lshape}.} \label{fig:timings} Cumulative time for the direct solve and AISFEM over the computational costs.}
\end{figure}

\begin{table}
	\begin{tabular}{c|c|c|c|c|c|c|c|c|c} 
		\toprule
		\diagbox{$\lamsym$}{$\theta$}& {0.1} & {0.2} & {0.3} & {0.4} & {0.5} & {0.6} & {0.7} & {0.8} & {0.9}\\ \midrule
		$10^{-1}$  & 533 & 470 & \colorbox{pyYellow}{402} & 424 & 497 & 608 & 801 & 971 & 1513 \\
		$10^{-2}$  & 3084 &  1878 & 1566 & 1482 & 1524  & 1624 & 1869 & 2485 & 4266  \\
		$10^{-3}$  & 6543  & 4490 & 3478 & 2831 & 2894  & 3371 & 3826 & 4729 & 6956  \\
		$10^{-4}$  & 10791  & 6621 & 5211 & 4381 & 4475 & 4777 & 5979 & 7398 & 10901 \\ \bottomrule
	\end{tabular}
	\vspace{0.4cm}
	\caption{\label{tab:cost} \textit{Optimal selection of parameters with respect to the computational costs for the experiment from Section~\ref{section:Lshape}.} For the comparison, we consider the weighted costs $ \big[ \eta_{\ell}(u_\ell^{\kk, \jj}) \, \sum_{| \ell', k', j' | \le | \ell, \kk, \jj |} {\rm dim} \, \mathcal{X}_{\ell'} \big]$ with stopping criterion $\eta_\ell(u_\ell^{\kk, \jj}) < 10^{-5}$ for various choices of $\lambda_{\rm sym}$ and $\theta$ with the optimal choice highlighted in color.}
\end{table}

\begin{figure}[htpb!]
	\centering
	\subfloat{
		\includegraphics[scale=0.8]{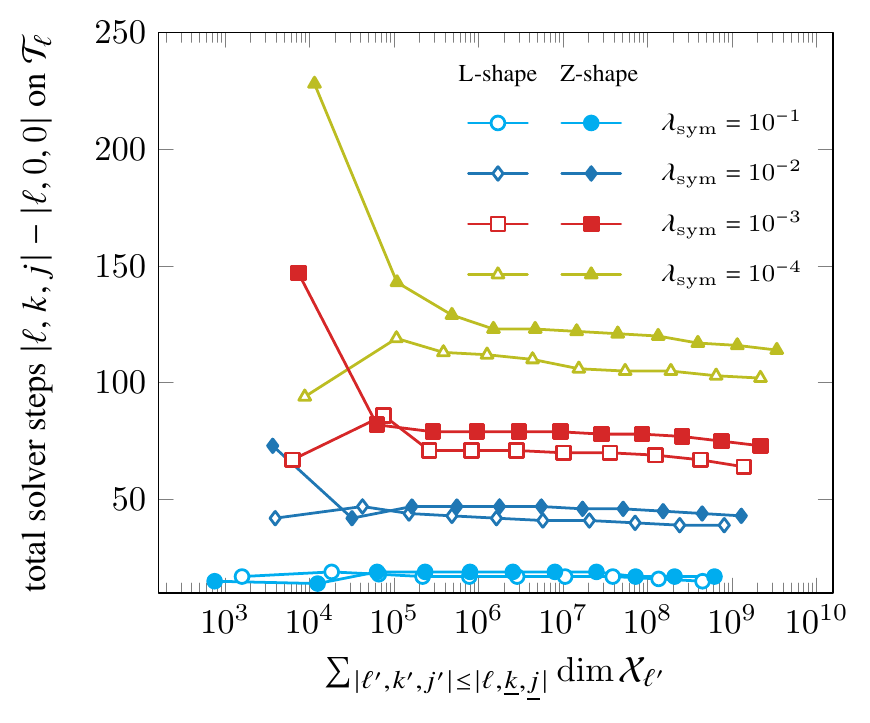}
		\label{fig:lambda:iterations}
	}
	\subfloat{
		\includegraphics[scale=0.8]{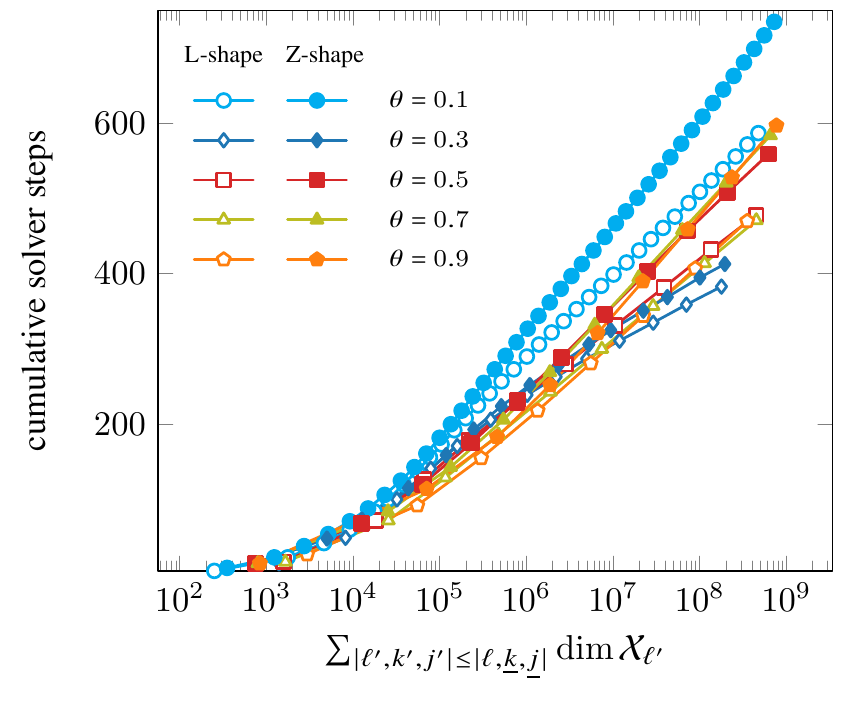}
		\label{fig:theta:iterations}
	}
	\caption{\label{fig:iterations} \textit{Uniform bound on the iteration numbers for the diffusion-convection-reaction problem on the L-shaped domain from Section~\ref{section:Lshape} and the strong convection problem on the Z-shaped domain from Section~\ref{section:Zshape}.} Number of total solver steps $| \ell, k, j | - | \ell, 0, 0|$ on the level $\ell$ for various selections of the symmetrization stopping parameter $\lambda_{\rm sym}$ with fixed $\theta = 0.5$ (left) and the cumulative solver steps for different marking parameter $\theta$ with fixed $\lambda_{\rm sym} = 0.1$ (right).}
\end{figure}

%%%%%%%%%%%%%%%%%%
\subsubsection*{\textbf{Parameter study of AISFEM}}
%%%%%%%%%%%%%%%%%%
We now investigate which parameters yield the best contraction in the iteratively symmetrized steps~\ref{algorithm}(ii)--(iii). Since the parameters depend on the contraction factors $\qalg$ from \eqref{eq:contraction:alg} and $\qsym$ from \eqref{eq:zarantonello:unperturbed}, we study a setting where the exact discrete solution $u_\ell^\exact$ to \eqref{eq:intro:discrete} and the exact Zarantonello solution $u_\ell^{k, \exact}$ to \eqref{eq:intro:zarantonello} are computed. Then, we compute $\qalg(\ell, k,j)$ for $(\ell, k, j) \in \QQ$ and define the level-wise contraction factors $\qalg(\ell)$ as the maximum over all $\qalg(\ell, k,j)$ for fixed $\ell \in \N_0$ and analogously for $\qsym$. From now on, we fix the polynomial degree $m=2$ and the parameters $\lamalg = 10^{-2}$ for the numerical experiments. We investigate the behavior of the combined solver for various choices of $\lamsym \in \set{10^{-1}, 10^{-2}, 10^{-3}, 10^{-4}}$ and $\theta \in \set{0.1, 0.3, 0.5, 0.7, 0.9}$. Figure~\ref{fig:Lshape:lamalg} shows upper bounds $\lambda_{\rm alg} < \overline{\lambda}_{\rm alg}^\exact$  = $(1-q_{\rm sym})(1-q_{\rm alg})/(4 \, q_{\rm alg})$ (see the implicit definition of $\overline{\lambda}_{\rm alg}^\exact$ in \eqref{eq1:lem:contraction_perturbed}) and Figure~\ref{fig:Lshape:contractions} displays contraction factors $\qsym \approx 1/2$ and $\qsymm \approx 1/2$, independently of the choice of $\theta$ and $\lamsym$. Note that $\qsym$ being close to $\qsymm$ means that the perturbed, i.e., iteratively symmetrized, Zarantonello step is of comparable performance to the unperturbed Zarantonello iteration. Moreover, Table~\ref{tab:cost} shows that the optimal combination of the parameters with respect to the computational costs is $\theta = 0.3$ and $\lamsym = 10^{-1}$. Furthermore, it appears that the choice of $\theta$ has a stronger impact on the costs than the selection of $\lamsym$.

\begin{figure}[htp!]
	\centering
	\subfloat{
		\includegraphics[scale=0.8]{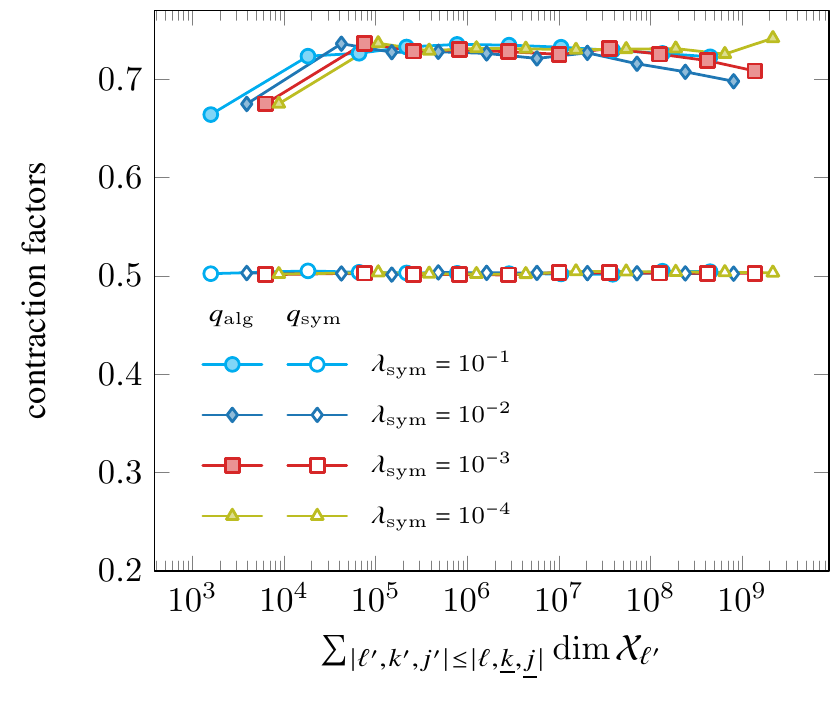}
	}
	\subfloat{
		\includegraphics[scale=0.8]{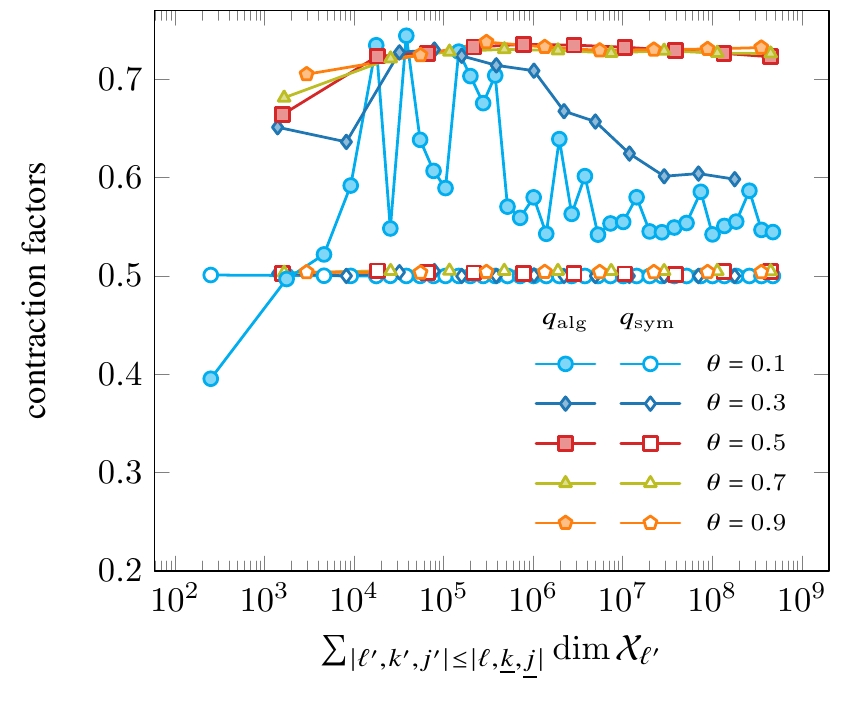}
	} 
	\caption{\label{fig:Lshape:contractions}\textit{Uniform contraction of the iterative solver for the diffusion-convection-reaction problem on the L-shaped domain from Section~\ref{section:Lshape}.} Experimental contraction factors $q_{\rm alg}$, $q_{\rm sym}$ and $\overline{q}_{\rm sym}$ for various choices of the symmetrization stopping parameter $\lambda_{\rm sym}$ with fixed $\theta = 0.5$ (left) and different marking parameter $\theta$ with fixed $\lambda_{\rm sym} = 0.1$ (right).}
\end{figure}

\begin{figure}
	\centering
	\subfloat{
		\includegraphics[scale=0.69]{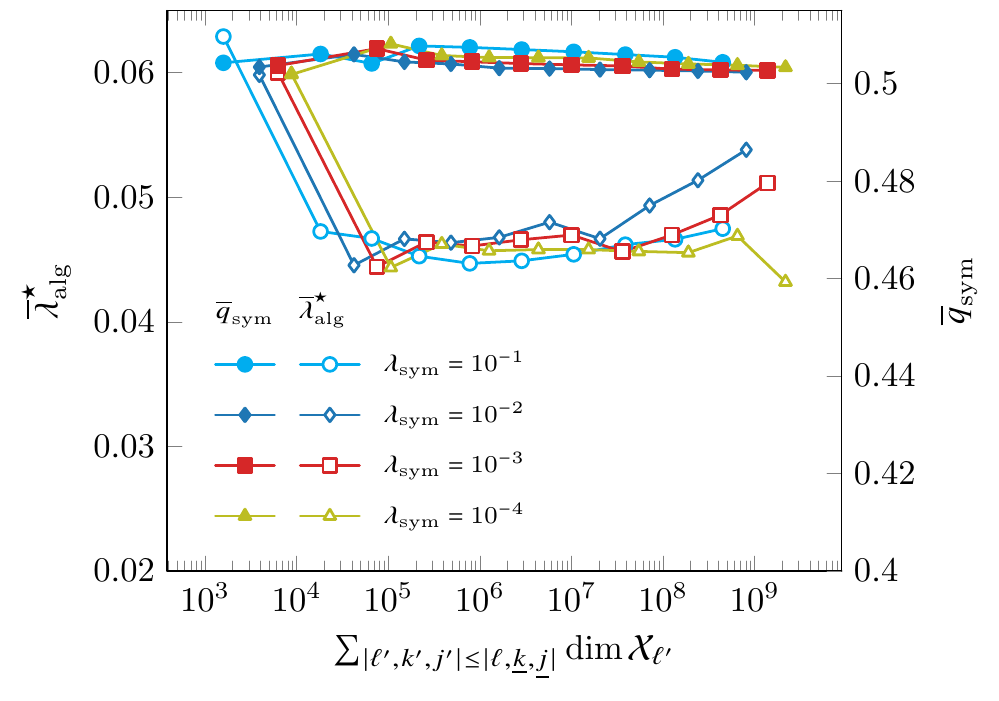}
	}
	\subfloat{
		\includegraphics[scale=0.69]{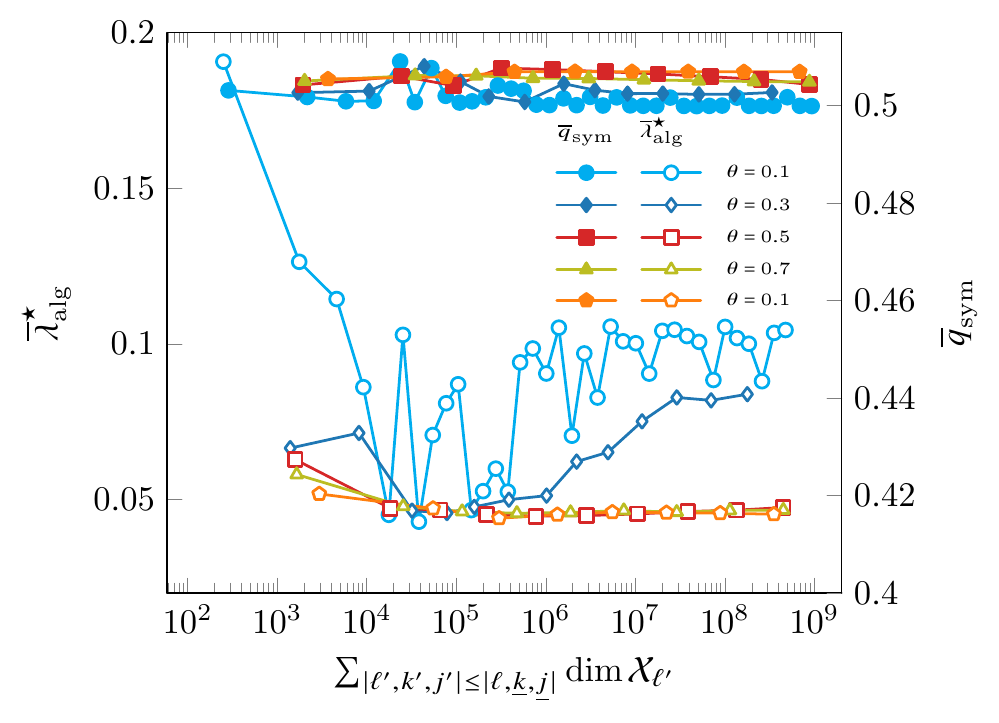}
	} 
	\caption{\label{fig:Lshape:lamalg} Computed upper bounds for $\lambda_{\rm alg}^\exact < \overline{\lambda}_{\rm alg}^\exact$ for various choices of the symmetrization stopping parameter $\lambda_{\rm sym}$ with fixed $\theta = 0.5$ (left) and different marking parameter $\theta$ with fixed $\lambda_{\rm sym} = 0.1$ (right), where we emphasize the double scaling of the $y$-axis for $\lambda_{\rm alg}^\exact$ resp. $ \overline{q}_{\rm sym}$ in both figures.}
\end{figure}

%%%%%%%%%%%%%%%%%%%%%%%%%%%%%%%%%%%%%%%%%%%%%%%%%%%%%%%%%%%%%%%%
\subsection{Strong convection on Z-shaped domain}\label{section:Zshape}
%%%%%%%%%%%%%%%%%%%%%%%%%%%%%%%%%%%%%%%%%%%%%%%%%%%%%%%%%%%%%%%%
In this subsection, we consider the problem \eqref{eq:intro:model_pb} on the Z-shaped domain $\Omega = (-1,1)^2 \setminus {\rm conv}\set{(0,0), (-1,0), (-1,-1)} \subset \R^2$ with coefficients $\boldsymbol{A}(x) = \mathrm{Id}$ and ${\bfb}(x) = (5,5)^\top$, and right-hand side $f(x) = 1$, i.e.,
\begin{align*}
	-\Delta u^{\exact}(x) + (5,5)^\top \cdot \nabla u^\exact(x) = 1
	\quad \text{for } x \in \Omega
	\quad \text{and} \quad
	u^\exact(x) = 0
	\quad \text{for } x \in \partial \Omega.
\end{align*}
Figure~\ref{fig:zshape:optimality} shows that even for a strong convection combined with a strong singularity at the origin, the adaptive algorithm recovers the optimal convergence rates $-m/2$ for several polynomial degrees $m$ both with respect to the cumulative costs and computational time. In Figure~\ref{fig:iterations} we see that the number of solver steps per level $\ell$ behaves similarly to the diffusion-convection-reaction problem on the L-shape from Section~\ref{section:Lshape} with an increase due to the stronger singularity. Furthermore, Figure \ref{fig:zshape:contractions} displays upper bounds on $\lamalg \le \lamalg^\exact < \overline{\lambda}_{\rm alg}^\exact$ and the contraction factor $\qsymm \approx 1/2$ (after an initial phase of reduced contraction) for the perturbed Zarantonello system in \eqref{eq1:lem:contraction_perturbed}.
\begin{figure}[htbp!]
	\centering
	\subfloat{
		\includegraphics[scale=0.8]{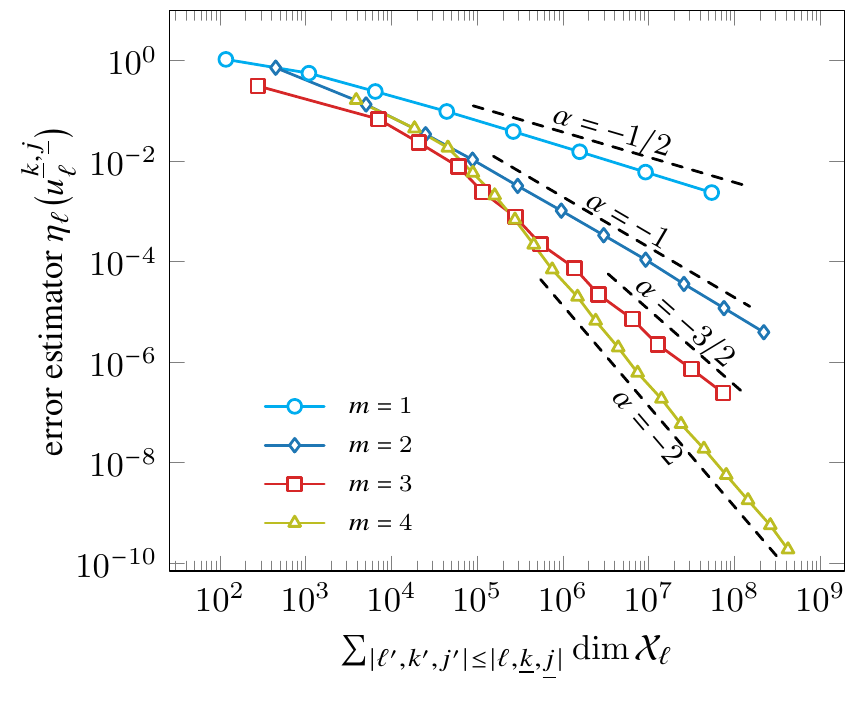}
		\label{fig:zshape:convergence}
	}
	\subfloat{
		\includegraphics[scale=0.8]{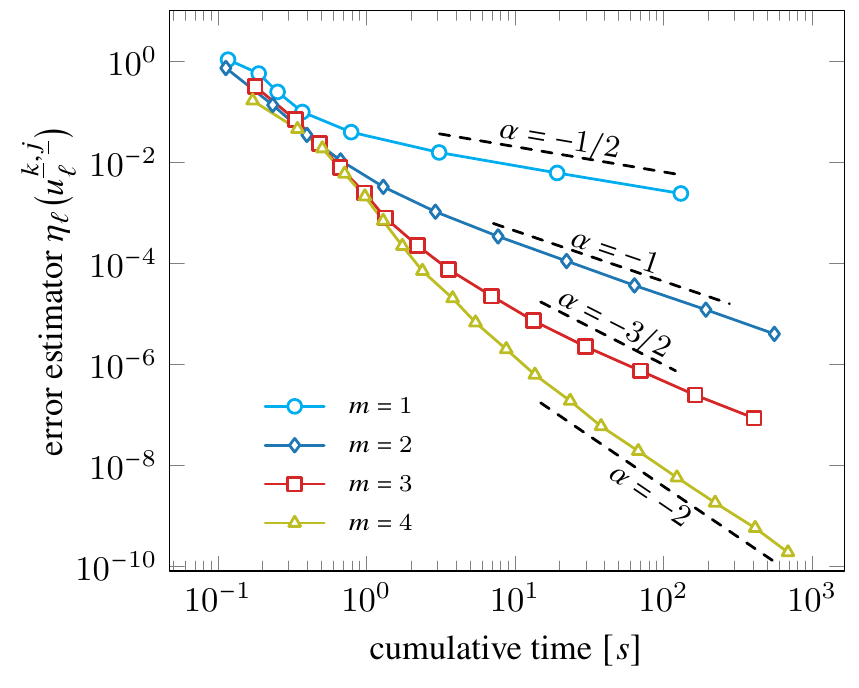}
		\label{fig:zshape:complexity}
	}
	\caption{\label{fig:zshape:optimality}\textit{Optimality of AISFEM for the strong convection problem on the Z-shaped domain from Section~\ref{section:Zshape}.} Convergence history plot of the error estimator $\eta_{\ell}(u_\ell^{\kk, \jj})$ over the computational cost (left) and the elapsed computational time (right).}
\end{figure}

\begin{figure}[htp!]
	\centering
	\subfloat{
		\includegraphics[scale=0.7]{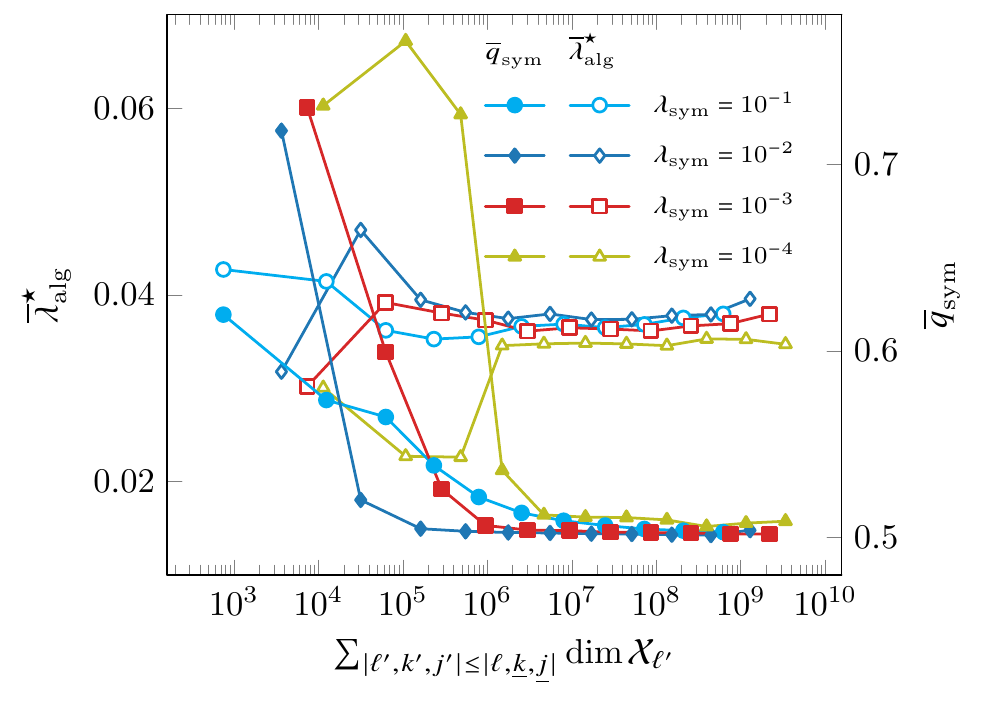}
	}
	\subfloat{
		\includegraphics[scale=0.7]{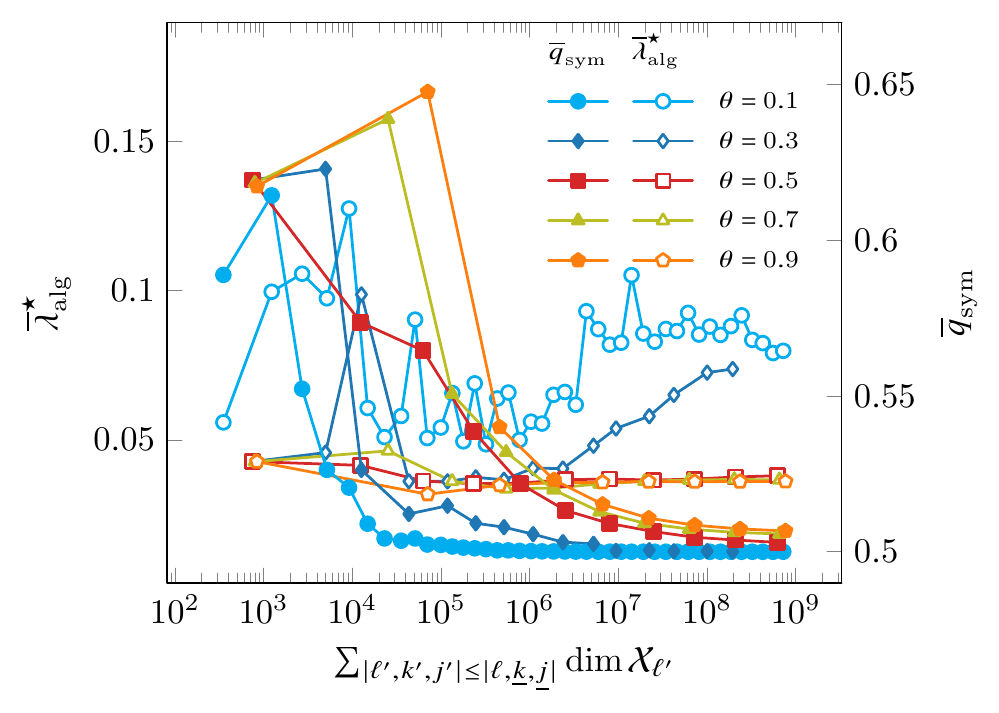}
	}
	\caption{\label{fig:zshape:contractions}\textit{Uniform contraction of the combined solver for the strong convection problem on the Z-shaped domain from Section~\ref{section:Zshape}.} Contraction factor $\overline{q}_{\rm sym}$ and computed upper bound for $\lambda_{\rm alg}^\exact < \overline{\lambda}_{\rm alg}^\exact$ for various symmetrization stopping parameter $\lambda_{\rm sym}$ with fixed $\theta = 0.5$ (left) and different marking parameter $\theta$ with fixed $\lambda_{\rm sym} = 0.1$ (right), where we emphasize the double scaling of the $y$-axis for $\lambda_{\rm alg}^\exact$ resp. $ \overline{q}_{\rm sym}$ in both figures.}
\end{figure}

%\clearpage

%%%%%%%%%%%%%%%%%%%%%%%%%%%%%%%%%%%%%%%%%%%%%%%%%%%%%%%%%%%%%%%%
\section{Conclusion and future work}\label{section:conclusion}
%%%%%%%%%%%%%%%%%%%%%%%%%%%%%%%%%%%%%%%%%%%%%%%%%%%%%%%%%%%%%%%%

In this work, we have developed and analyzed an adaptive finite element method for nonsymmetric second-order linear elliptic PDEs~\eqref{eq:intro:model_pb}. From a conceptual point of view, the crucial assumption is that the weak formulation takes the form
\begin{align}\label{eq:conclusion:1} 
 b(u^\exact, v) \coloneqq a(u^\exact, v) + \dual{\KK u^\exact}{v} = F(v)
 \quad \text{for all } v \in \XX,
\end{align}
where $F \in \XX'$ is a linear and continuous functional, $a(\cdot, \cdot)$ is a symmetric, continuous, and elliptic bilinear form on $\XX$, and $\KK \colon \XX \to \XX'$ is a compact operator such that the bilinear form $b(\cdot,\cdot)$ is still elliptic on $\XX$. Let $\enorm{\, \cdot \,}$ denote the $a(\cdot, \cdot)$-induced energy norm. For the discrete formulation
\begin{align}\label{eq:conclusion:2}
 b(u_\ell^\exact, v_\ell) = F(v_\ell)
 \quad \text{for all } v_\ell \in \XX_\ell,
\end{align}
we require an (abstract) inexact iterative solver with iteration map given by \linebreak[4]$\overline\Phi_\ell(F;\cdot) \colon \XX_\ell \to \XX_\ell$ that contracts the \emph{error} in the energy norm, i.e.,
\begin{align}\label{eq:conclusion:3} 
 \enorm{u_\ell^\exact - \overline{u}_\ell^{k+1}}
 \le \qsymm \, \enorm{u_\ell^\exact - \overline{u}_\ell^k}
 \quad \text{with } \overline{u}_\ell^{k+1} \coloneqq \overline\Phi_\ell(F; \overline{u}_\ell^k) 
 \text{ for all } k \in \N,
\end{align}
where the contraction constant $0 < \qsymm < 1$ is independent of $\overline{u}_\ell^0 \in \XX_\ell$. Under such assumptions and with the usual residual \textsl{a~posteriori} error estimator $\eta_\ell(\cdot)$ (satisfying the abstract assumptions~\eqref{axiom:stability}--\eqref{axiom:discrete_reliability}) on nested conforming discrete spaces $\XX_\ell \subseteq \XX_{\ell+1} \subset \XX$, the present work proves that the analysis from~\cite{ghps2021} can be generalized from symmetric PDEs (with $\KK = 0$) to the general formulation~\eqref{eq:conclusion:1}: Restricting Algorithm~\ref{algorithm} to the outer $\ell$-loop (for mesh refinement) and the inner $k$-loop (for the solver associated to $\overline \Phi_\ell$), we obtain a simplified index set
\begin{align}\label{eq:conclusion:4}
 \overline\QQ \coloneqq \set{(\ell,k) \in \N_0^2 \given \overline{u}_\ell^k \text{ is computed by the simplified algorithm}}
\end{align}
together with the canonical step counter $|\ell,k| \in \N_0$ on $\overline\QQ$ defined analogously to~\eqref{eq:stepcounter}.
Then, Lemma~\ref{lemma:new:ghps} (lucky non-termination of the solver), Lemma~\ref{prop:plain-convergence} (\textsl{a~priori} convergence), Lemma~\ref{lemma:quasi-pythagoras} (quasi-Pythagorean estimate), and Lemma~\ref{lemma:ghps} (contraction of weighted discretization and solver error) hold verbatim (and the proof of Lemma~\ref{lemma:quasi-pythagoras} indeed relies on the compactness of $\KK$) if we replace $u_\ell^{k,\jj}$ in the given proofs by $\overline{u}_\ell^k$ in the current solver setting. Therefore, we obtain full linear convergence in the spirit of Theorem~\ref{theorem:aisfem:linear-convergence}:
For arbitrary adaptivity parameters $0 < \theta \le 1$ and $\lamsym > 0$, there exist constants $\Clin > 0$ and $0 < \qlin < 1$ as well as an index $\ell_0 \in \N_0$ such that
\begin{align}\label{eq:conclusion:5} 
 \overline\Delta_{\ell}^{k} \le \Clin \, \qlin^{|\ell,k|-|\ell',k'|} \, \overline\Delta_{\ell'}^{k'}
 \quad \text{for all } (\ell',k'),(\ell,k) \in \overline\QQ 
 \text{ with } |\ell',k'| \le |\ell,k| \text{ and } \ell' \ge \ell_0,
\end{align}
where $\overline\Delta_\ell^k \coloneqq \enorm{u^\exact - \overline{u}_\ell^k} + \eta_\ell(\overline{u}_\ell^k)$ denotes the corresponding quasi-error.
In particular, also Corollary~\ref{corollary:aisfem:linear-convergence} holds verbatim with $\QQ$ replaced by $\overline\QQ$ and $\Delta_{\ell}^{k,j}$ replaced by $\overline\Delta_\ell^k$, i.e., convergence rates with respect to the number of degrees of freedom coincide with rates with respect to the overall computational cost. Finally, it is easy to check that also Theorem~\ref{theorem:aisfem:complexity} holds verbatim and proves that, for sufficiently small adaptivity parameters $0 < \theta \ll 1$ and $0 < \lamsym \ll 1$ in the sense of~\eqref{eqxx:theorem:aisfem:complexity}, it holds that
\begin{align}\label{eq:conclusion:6} 
		\norm{u^\exact}_{\A_s(\TT_{0})} < \infty
		\quad \Longleftrightarrow \quad
		\sup_{(\ell,k) \in \QQ}
		\Bigl(\sum_{\substack{(\ell',k') \in \overline\QQ \\ |\ell',k'| \le |\ell,k|}} \#\TT_{\ell'}\Bigr)^s \, \overline\Delta_\ell^{k} < \infty,
\end{align} 
which yields optimal complexity of the simplified algorithm.

In the current analysis, the combined Zarantonello symmetrization with a contractive SPD algebraic solver is used as one solver module to guarantee~\eqref{eq:conclusion:3} for $\overline{u}_\ell^k \coloneqq u_\ell^{k,\jj}$ (see Lemma~\ref{lem:contraction_perturbed}, where contraction, however, only holds for $1 \le k < \kk[\ell]$), leading to all results being formulated over the triple index set $\QQ \subset \N_0^3$ (see Section~\ref{section:algorithm}--\ref{section:main_results}).

 We note that another choice for solving the arising nonsymmetric FEM systems would be preconditioned GMRES (see, e.g.,~\cite{MR848568,MR1990645}), where an optimal preconditioner for the symmetric part would be employed. Then, it is well-known from the field-of-value analysis (see, e.g.,~\cite{MR1483571}) that the algebraic solver would satisfy a \emph{generalized} contraction for the algebraic residual (in a discrete vector norm). However, the link between the algebraic residual and the functional setting appears to be open. Moreover, the \textsl{a~posteriori} error control of the algebraic error for such a GMRES solver is still to be developed. 
 
 While these questions are left for future work, we already note some results that can be achieved along the arguments of~\cite{ghps2021}: If the solver $\overline\Phi_\ell(F;\cdot)$ provides iterates $(\overline{u}_\ell^k)_{k \in \N_0}$ satisfying only the generalized contraction
\begin{align}\label{eq:conclusion:7} 
 \enorm{u_\ell^\exact - \overline{u}_\ell^k}
 \le \Cctr \, \qsymm^k \, \enorm{u_\ell^\exact - \overline{u}_\ell^0}
% \quad \text{with } u_\ell^{k+1} \coloneqq \overline\Phi_\ell(F;u_\ell^k) 
\quad  \text{ for all } k \in \N
\end{align}
together with the \textsl{a~posteriori} error control
\begin{align}\label{eq:conclusion:8} 
 \enorm{u_\ell^\exact - \overline{u}_\ell^k}
 \le \Cctr' \, \enorm{u_\ell^k - \overline{u}_\ell^{k-1}}
 \quad \text{for all } k \in \N,
\end{align}
where $\Cctr, \Cctr' > 0$ and $0 < \qsymm < 1$ are given constants independently of $\overline{u}_\ell^0 \in \XX_\ell$, then full linear convergence~\eqref{eq:conclusion:5} can be proved for all $0 < \theta \le 1$ under the additional assumption that $\lamsym$ has to be sufficiently small. However, the proof of full linear convergence~\eqref{eq:conclusion:5} for arbitrary $0 < \theta \le 1$ and arbitrary $\lamsym >0$ is open, while optimal complexity~\eqref{eq:conclusion:6} for sufficiently small $0 < \theta < 1$ and $\lamsym$ in the sense of~\eqref{eqxx:theorem:aisfem:complexity} remains valid (even with the same proof).

%%%%%%%%%%%%%%%%%%%%%%%%%%%%%%%%%%%%%%%%%%%%
\printbibliography

\end{document}